\documentclass[11pt,a4paper]{article}

\usepackage{amsmath}
\usepackage{amssymb,latexsym}
\usepackage{amsthm}
\usepackage[shortlabels]{enumitem}
\usepackage{array}
\usepackage{ifthen}
\usepackage[hang,flushmargin,bottom]{footmisc}
\usepackage[margin=0.35in,format=hang,font=small,labelfont=bf]{caption}
\usepackage[normalem]{ulem}
\usepackage{needspace}
\usepackage{graphicx,tikz}
\usepackage[colorlinks=true,urlcolor=blue,linkcolor=blue,citecolor=blue]{hyperref}
\usepackage{palatino}
\usepackage{eulervm}
\usepackage{pdftexcmds}
\usepackage{mathdots}

\makeatletter
\newcommand{\strequal}[2]{\pdf@strcmp{#1}{#2}==0}
\makeatother

\newtheorem{thm}{Theorem}[section]
\newtheorem*{thm*}{Theorem}
\newtheorem{prop}[thm]{Proposition}
\newtheorem{cor}[thm]{Corollary}
\newtheorem{lemma}[thm]{Lemma}

\newtheorem{obs}[thm]{Observation}
\newtheorem{conj}[thm]{Conjecture}
\newtheorem*{conj*}{Conjecture}
\newtheorem{question}[thm]{Question}

\theoremstyle{definition}
\newtheorem{defn}[thm]{Definition}
\newtheorem{defn*}{Definition}

\newtheorem*{example*}{Example}

\newtheorem*{comment*}{Comment}

\newenvironment{bullets} {\vspace{-9pt}\begin{itemize}\itemsep0pt} {\end{itemize}\vspace{-9pt}}

\newenvironment{bulletnums} {\vspace{-9pt}\begin{enumerate}\itemsep0pt} {\end{enumerate}\vspace{-9pt}}

\newcommand{\AAA}{\mathcal{A}}

\newcommand{\CCC}{\mathcal{C}}
\let\C\CCC
\newcommand{\DDD}{\mathcal{D}}

\newcommand{\HHH}{\mathcal{H}}
\newcommand{\III}{\mathcal{I}}

\newcommand{\OOO}{\mathcal{O}}

\newcommand{\QQQ}{\mathcal{Q}}

\newcommand{\VVV}{\mathcal{V}}

\newcommand{\av}{\mathsf{Av}}

\newcommand{\Grid}{\mathsf{Grid}}
\newcommand{\grid}[1]{\Grid\!\left(\!#1\!\right)\!}
\newcommand{\Gridhash}{\Grid^\#}
\newcommand{\gridhash}[1]{\Gridhash\!\!\left(\!#1\!\right)\!}

\newcommand{\Geom}{\mathsf{Geom}}

\DeclareMathOperator{\Span}{Span}
\DeclareMathOperator{\wkSpan}{WeakSpan}
\DeclareMathOperator{\rg}{rg}
\newcommand{\gridded}{\#}
\DeclareFontFamily{U}{musix}{}%
\DeclareFontShape{U}{musix}{m}{n}{<-15> musix16}{}%
\DeclareTextFontCommand{\textmusix}{\usefont{U}{musix}{m}{n}\selectfont}
\newcommand*\ggridded{\raisebox{.6ex}{\textmusix{5}}}

\newenvironment{smallmx}[1][{}] {\left(\!\begin{smallmatrix}} {\end{smallmatrix}\!\right)}

\usepackage{ifthen}
\usetikzlibrary{calc}
\usetikzlibrary{decorations.pathreplacing,decorations.markings}

\newcommand{\plotptradius}{4pt}
\newcommand{\setplotptradius}[1]{\renewcommand{\plotptradius}{#1}}
\tikzset{permpt/.style={circle, draw, fill=black, inner sep=0pt, minimum width=\plotptradius}}
\tikzset{empty/.style={draw=none, fill=none}}

\newcommand\absdot[2]{
	\node[permpt] at #1 {};
}

\newcommand{\plotperm}[2][black]{ 
	\foreach \j [count=\i] in {#2} {
        \ifnum0=\j {} \else {
 		\node[permpt,fill=#1,draw=#1] (\j) at (\i,\j) {};
	} \fi
	};
}
\newcommand{\plotpermbox}[4]{
	\draw [darkgray, very thick, line cap=round, fill=white]
		({#1-0.5}, {#2-0.5}) rectangle ({#3+0.5}, {#4+0.5});
}

\newcommand{\plotpermborder}[1]{
	\foreach \i [count=\nn] in {#1} {\global\let\n\nn};    
	\plotpermbox{1}{1}{\n}{\n};
	\plotperm{#1};
}
\newcommand{\plotpermbordergrid}[1]{
	\foreach \i [count=\nn] in {#1} {\global\let\n\nn};    
	\plotpermbox{1}{1}{\n}{\n};
	\draw[step=1cm,gray!50,very thin] (1,1) grid (\n,\n);
	\plotperm{#1};
}
\newcommand{\plotgrid}[1]{
	\draw[step=1cm,gray!50,very thin] (0.5,0.5) grid (#1+0.5,#1+0.5);
}

\newcommand{\plotpermgrid}[1]{
	\foreach \i [count=\nn] in {#1} {\global\let\n\nn};    
	\plotgrid{\n};
	\plotperm{#1};
}

\newcommand{\plotpinsequence}[1]{
	\absdot{(0,0)}{};
	\edef\n{0}
	\edef\s{0}
	\edef\e{0}
	\edef\w{0}
	\edef\x{0}
	\edef\y{0}
	\foreach \pin [remember=\pin as \oldpin (initially 1), count=\i] in {#1} {
		\ifthenelse{\pin=1 \OR \pin=2}{
			\ifthenelse{\oldpin=3}{
				\xdef\x{\number\numexpr\e-1}
			}{
				\xdef\x{\number\numexpr\w+1}
			}
			\ifnum\i=1 
				\pgfmathparse{\e+1}
 				\xdef\e{\pgfmathresult}
			\fi	
		}{ 
			\ifthenelse{\oldpin=1}{
				\xdef\y{\number\numexpr\n-1}
			}{
				\xdef\y{\number\numexpr\s+1}
			}
			\ifnum\i=1 
				\pgfmathparse{\s-1}
 				\xdef\s{\pgfmathresult}
			\fi	
		}
		\ifnum\pin=1 
			\pgfmathparse{\n+2}
 			\xdef\n{\pgfmathresult}		
			\absdot{(\x,\n)}{};
			\ifnum\i>1
				\draw (\x,\n) -- (\x,\y-0.5);
			\else
				\draw[gray,very thick] (-0.5,-0.5) rectangle (\x+0.5,\n+0.5);
			\fi
		\fi
		\ifnum\pin=2 
			\pgfmathparse{\s-2}
 			\xdef\s{\pgfmathresult}
			\absdot{(\x,\s)}{};
			\ifnum\i>1
				\draw (\x,\s) -- (\x,\y+0.5);
			\else
				\draw[gray,very thick] (-0.5,0.5) rectangle (\x+0.5,\s-0.5);
			\fi
		\fi
		\ifnum\pin=3 
			\pgfmathparse{\e+2}
 			\xdef\e{\pgfmathresult}
			\absdot{(\e,\y)}{};
			\ifnum\i>1
				\draw (\e,\y) -- (\x-0.5,\y);
			\else
				\draw[gray,very thick] (-0.5,+0.5) rectangle (\e+0.5,\y-0.5);
			\fi
		\fi
		\ifnum\pin=4 
			\pgfmathparse{\w-2}
 			\xdef\w{\pgfmathresult}
			\absdot{(\w,\y)}{};
			\ifnum\i>1
				\draw (\w,\y) -- (\x+0.5,\y);
			\else
				\draw[gray,very thick] (0.5,0.5) rectangle (\w-0.5,\y-0.5);

			\fi
		\fi		
	};
}

\tikzset{
  on each segment/.style={
    decorate,
    decoration={
      show path construction,
      moveto code={},
      lineto code={
        \path [#1]
        (\tikzinputsegmentfirst) -- (\tikzinputsegmentlast);
      },
      curveto code={
        \path [#1] (\tikzinputsegmentfirst)
        .. controls
        (\tikzinputsegmentsupporta) and (\tikzinputsegmentsupportb)
        ..
        (\tikzinputsegmentlast);
      },
      closepath code={
        \path [#1]
        (\tikzinputsegmentfirst) -- (\tikzinputsegmentlast);
      },
    },
  },
  mid arrow/.style={postaction={decorate,decoration={
        markings,
        mark=at position .5 with {\arrow[#1]{stealth}}
      }}},
}
\newcommand{\gcgap}{0.18}
\newcommand{\gcscale}{0.25}
\newcommand{\gcptsize}{0.22}
\newcommand{\gcptcolor}{black}
\newcommand{\gcptgridscale}{1}
\newcommand{\gcmode}{0}         
\newcommand{\gcarrowtip}{latex}
\newcommand{\gcextra}{}

\newcommand{\gclass}[4][0.25]  
{
  \begin{tikzpicture}[scale=#1,>=\gcarrowtip]
    \draw[very thin] (0,0) grid (#2,#3);
    #4
    \gcextra
  \end{tikzpicture}
}

\newcommand{\gcrow}[2]  
{
  \foreach \d [count=\i] in {#2}
  {
    \ifnum0=\d {} \else
      \ifnum0=\gcmode 
        \ifnum0<\d
        { \draw[very thick] (\i-1+\gcgap,#1+\gcgap)--(\i-\gcgap,#1+\d-\gcgap); }
        \else
        { \draw[very thick] (\i-1+\gcgap,#1-\d-\gcgap)--(\i-\gcgap,#1+\gcgap); }
        \fi
      \else
        \ifnum1=\gcmode 
          \pgfmathparse{abs(\d)}
          \ifnum1=\pgfmathresult
            { \draw[very thick] (\i-1+\gcgap,#1+.5-\d/2+\gcgap*\d)--(\i-\gcgap,#1+.5+\d/2-\gcgap*\d); }
          \else \ifnum2=\pgfmathresult
            { \draw[very thick,->] (\i-1+\gcgap,#1+.5-\d/4+\gcgap*\d/2)--(\i-\gcgap,#1+.5+\d/4-\gcgap*\d/2); }
          \else \ifnum3=\pgfmathresult
            { \draw[very thick,<-] (\i-1+\gcgap,#1+.5-\d/6+\gcgap*\d/3)--(\i-\gcgap,#1+.5+\d/6-\gcgap*\d/3); }
          \fi \fi \fi
        \fi
      \fi
    \fi
  }
}

\newcommand{\gcpoints}[1]  
{
  \foreach \y [count=\x] in {#1}
  {
    \ifnum-1<\y
      {\fill[radius=\gcptsize,\gcptcolor] (\x/\gcptgridscale-1/\gcptgridscale,\y/\gcptgridscale) circle;}
    \fi
  }
}

\newcommand{\gcone}[3][]   {\raisebox{ -0.33pt}{\gclass[\gcscale]{#2}{1}
                             {\gcrow{0}{#3}\gcpoints{#1}}}}
\newcommand{\gctwo}[4][]   {\raisebox{ -4.00pt}{\gclass[\gcscale]{#2}{2}
                             {\gcrow{1}{#3}\gcrow{0}{#4}\gcpoints{#1}}}}
\newcommand{\gcthree}[5][] {\raisebox{ -7.67pt}{\gclass[\gcscale]{#2}{3}
                             {\gcrow{2}{#3}\gcrow{1}{#4}\gcrow{0}{#5}\gcpoints{#1}}}}
\newcommand{\gcfour}[6][]  {\raisebox{-11.33pt}{\gclass[\gcscale]{#2}{4}
                             {\gcrow{3}{#3}\gcrow{2}{#4}\gcrow{1}{#5}\gcrow{0}{#6}\gcpoints{#1}}}}
\newcommand{\gcfive}[7][]  {\raisebox{-15.00pt}{\gclass[\gcscale]{#2}{5}
                             {\gcrow{4}{#3}\gcrow{3}{#4}\gcrow{2}{#5}\gcrow{1}{#6}\gcrow{0}{#7}\gcpoints{#1}}}}
\newcommand{\gcsix}[8][]   {\raisebox{-18.67pt}{\gclass[\gcscale]{#2}{6}
                             {\gcrow{5}{#3}\gcrow{4}{#4}\gcrow{3}{#5}\gcrow{2}{#6}\gcrow{1}{#7}\gcrow{0}{#8}\gcpoints{#1}}}}

\newcommand{\floor}[1]{\left\lfloor #1 \right\rfloor}

\newcommand{\note}[1]{\marginpar{\raggedright{\scriptsize\it #1}}}

\newcommand{\commentdb}[1]{{\color{green!33!black}[\emph{\textbf{DB:}~#1}]}}
\newcommand{\noterb}[1]{\note{\color{red!80!black}\textbf{RB:}~#1}}
\newcommand{\commentrb}[1]{{\color{red!80!black}[\emph{\textbf{RB:}~#1}]}}
\definecolor{tealblue}{rgb}{0.21, 0.46, 0.53}

\definecolor{burntorange}{rgb}{0.8, 0.33, 0.0}

\newcommand{\ifstr}[4]{
  \def\ifstrA{#1}%
  \def\ifstrB{#2}%
  \ifx\ifstrA\ifstrB #3\relax
  \else #4\relax
  \fi}

\title{\textbf{On cycles in monotone grid classes of permutations}}

\author{
David Bevan\thanks{Email addresses: \texttt{david.bevan@strath.ac.uk, rbrignall@gmail.com,\\ nik.ruskuc@st-andrews.ac.uk.}}\\[10pt]
\small Department of Mathematics and Statistics\\
\small University of Strathclyde\\
\small Glasgow, G1 1XH\\
\small United Kingdom
\and
Robert Brignall\footnotemark[1]\\[10pt]
\small School of Mathematics and Statistics\\
\small The Open University\\
\small Milton Keynes, MK7 6AA\\
\small United Kingdom
\and
Nik Ru\v{s}kuc\footnotemark[1]\\[10pt]
\small School of Mathematics and Statistics\\
\small University of St Andrews\\
\small St Andrews, KY16 9SS\\
\small United Kingdom
}

\begin{document}
\maketitle


\begin{abstract}\noindent
 We undertake a detailed investigation into the structure of permutations in monotone grid classes whose row-column graphs do not contain components with more than one cycle. Central to this investigation is a new decomposition, called the $M$-sum,
which generalises the well-known notions of direct sum and skew sum, and enables a deeper understanding of the structure of permutations in these grid classes. Permutations which are indecomposable with respect to the $M$-sum play a crucial role in the structure of a grid class and of its subclasses, and this leads us to identify coils, a certain kind of permutation which corresponds to repeatedly traversing a chosen cycle in a particular manner.

Harnessing this analysis, we give a precise characterisation for when a subclass of such a grid class is labelled well quasi-ordered, and we extend this to characterise (unlabelled) well quasi-ordering in certain cases. We prove that a large general family of these grid classes are finitely based, but we also exhibit other examples that are not, thereby disproving a conjecture from 2006 due to Huczynska and Vatter.
\end{abstract}

%
%
%
%
%
%
%
%
\section{Introduction}

Throughout the study of various types of combinatorial structures, and indeed general relational structures on finite sets,
decomposition strategies play a prominent role.
Very loosely speaking, a decomposition consists of partitioning the structure in such a way that the individual parts are well understood and `simpler' than the whole, and so that the relationship between parts is also controlled.

One concrete example concerns the so-called modular decompositions. They date back to the 1950s and a talk of Fra{\"\i}ss\'e, of which only the abstract \cite{Fr:53} survives, and
their use in graph theory can be traced back to Gallai's work \cite{Ga:67,Ga:trans01}.
Since then modular decompositions have become a standard tool, especially in topics related to algorithms and complexity; see~\cite[Section 12.1]{BBS:99} for an introduction, and~\cite{BFKRT:18,BGW:12,FV:22,HMMZ:22} for some recent results.

The defining feature of modular decompositions is that the relationship between points belonging to different blocks is completely determined.
Moving away from such a strict requirement,
in extremal graph (and hypergraph) theory, various decompositions are at the heart of many significant developments: besides the classical results of Tur\'an, there is the celebrated Regularity Lemma of Szemeredi~\cite{szemeredi:regularity-lemma:}, the Container method (originally due to Kleitman and Winston~\cite{kleitman:the-asymptotic:}), and notions of $\chi$-bounded families of graphs (originating with Gy\'arf\'as~~\cite{gyarfas:problems-from:}, as a generalisation of perfect graphs).

In the study of hereditary properties of graphs, Balogh, Bollob\'as and Morris~\cite{balogh:a-jump-to-the-bell:} showed that any graph class whose speed falls below the Bell numbers can be partitioned into a union of cliques and independent sets, so that the connections between each part are either very dense, or very sparse. This was subsequently generalised by Atminas~\cite{atminas:classes-of-graphs:}, and this result is in fact the graph-theoretic equivalent of the characterisation of monotone griddability in permutations, due to Huczynska and Vatter~\cite{Huczynska2006}.

 For hereditary classes whose speed is smaller still, progressively finer structural decompositions are available, one notable example being the `letter graphs' of Petkov\v{s}ek~\cite{petkovsek:letter-graphs-a:}. These are the graph-theoretic analogue of geometrically griddable permutations~\cite{albert:geometric-grid-:}, as established by Alecu, Ferguson, Kant\'e, Lozin, Vatter and Zamaraev~\cite{alecu:letter-graphs-and-ggcs:}.
There are therefore parallels at several levels between hereditary graph classes and permutation classes.

The purpose of this paper is to advance the framework of structural decomposition in the area of permutation classes, by undertaking a detailed analysis of the so-called monotone griddable classes.  Central to our analysis is a new decomposition, called the $M$-sum (introduced in Section~\ref{sec-indiv-and-coils}), which significantly expands our understanding of this important family of permutation classes. A list of the consequences of our decomposition is provided later in this introduction.

The structural study of permutation classes has become a major area of endeavour in the last 30 years, motivated by questions in enumerative combinatorics, by parallels with the study of graph classes, and as a topic in its own right. Indeed, the phase transitions in the number and complexity of permutation classes whose growth rates are `small', as established by Kaiser and Klazar~\cite{kaiser:on-growth-rates:} and Vatter~\cite{vatter:small-permutati:,vatter:growth-rates-of:}, are exemplars of the interplay between enumeration and structure. Vatter~\cite{vatter:growth-rates-of:} also gives a good account of the parallels with the aforementioned study of speeds of hereditary graph classes, which started with Scheinerman and Zito~\cite{scheinerman:on-the-size:}.

Within this structural study, given a suitable description of some permutation class $\C$, one may ask a number of pertinent questions, such as:
\begin{itemize}
\item Does $\C$ have some kind of concise definition? For example, is it finitely based?
\item Can the structure of the permutations in $\C$ be 
succinctly represented?
\item Does $\C$ have some notable order-theoretic properties? For example, is it well quasi-ordered or even labelled well quasi-ordered?
\item
Does $\C$ have good enumerative properties, such as a tractable generating function, or at least an asymptotic growth that can be 
explicitly given or algorithmically computed?
\end{itemize}

Over recent years it has transpired that a key role in many developments in each of the above strands is played by the \emph{grid classes}. 
Roughly speaking, a grid class comprises permutations which, when plotted in the plane, can be subdivided into cells in a grid-like manner so that the entries in each cell belong to a simpler permutation class, specified by a matrix entry. 
This notion originates from the ``profile classes'' studied by Murphy and Vatter~\cite{murphy:profile-classes:}, and the name `grid classes' was introduced in \cite{Huczynska2006}.

In full generality, where there are no restrictions on the types of classes that the matrix may contain, we have little understanding about the behaviour of the resulting \emph{general} grid classes, beyond knowing their growth rates (see Albert and Vatter~\cite{albert:bevanstheorem}), and some results relating to well quasi-order and labelled well-quasi-order, see Brignall~\cite{brignall:pwo-grid-classes:,brignall:lwqo-juxt:}.

At the other end of the spectrum are the \emph{geometric grid classes}, in which not only must the entries in each cell of a subdivided permutation form a \emph{monotone} sequence,
but it is additionally required that the permutations in the class can be plotted in such a way that their points lie on the cell diagonals.
A major study, due to Albert, Atkinson, Bouvel, Ru\v{s}kuc and Vatter~\cite{albert:geometric-grid-:}, showed (among other results) that every geometric grid class is finitely based, well quasi-ordered and enumerated by a rational generating function. Brignall and Vatter~\cite{bv:lwqo-for-pp:} subsequently established that geometric grid classes are in fact labelled well quasi-ordered.

The monotone grid classes that form the basis of this article form an intermediate family of grid classes. As they are the only classes we consider in the sequel, we will henceforth refer to them simply as `grid classes'.
A \emph{gridding matrix} is a matrix $M$ whose entries are all in $\{0,\pm1\}$, and we let $\Grid(M)$ denote the \emph{grid class} of $M$, which comprises all permutations which can be divided into a rectangular grid of cells (of the same dimensions as $M$), such that each cell $ij$ of the permutation contains no entries if $M_{ij}=0$, contains entries forming a monotone increasing sequence if $M_{ij}=1$, and contains a monotone decreasing sequence if $M_{ij}=-1$.

For a gridding matrix $M$ with $m$ columns and $n$ rows, the \emph{row-column graph} of $M$, denoted $G_M$, is the bipartite graph with vertices
$\{1,\dots,m\} \,\cup\, \{1',\dots,n'\}$
such that $ij' \in E(G_M)$ whenever the corresponding entry $M_{ij}$ of $M$ is non-zero.
\footnote{A related concept, known as the \emph{cell graph}, has also been used in the literature (see, for example, Vatter~\cite{vatter:small-permutati:,vatter:growth-rates-of:}),
 but the row-column graph will prove more useful for our work.}
The graphs $G_M$ play an important role in the complexity of the class $\Grid(M)$. For example, the growth rate of $\Grid(M)$ is equal to the square of the spectral radius of $G_M$ (see Bevan~\cite{bevan:growth-rates-ggc,bevan:growth-rates:}, and also Albert and Vatter~\cite{albert:bevanstheorem} for a generalisation), while $\Grid(M)$ is well quasi-ordered if and only if $G_M$ is a forest (due to Murphy and Vatter~\cite{murphy:profile-classes:}; see also Vatter and Waton~\cite{vatter:on-partial-well:}).

With results like these in mind, the structure of the components in the graphs $G_M$ provides us with a taxonomy for the classes $\Grid(M)$, and we borrow terminology from graphs to name them.
\begin{itemize}
\item If $G_M$ contains no cycles (that is, every connected component of $G_M$ is a tree), then we say that $M$ is \emph{acyclic}.
\item If $G_M$ has exactly one cycle, then $M$ is \emph{unicyclic},\footnote{Note that our use of the term unicyclic is consistent with its general use in graph theory, but is different from how the term has been used in other studies of grid classes, notably in the PhD theses of Bevan~\cite{bevan:thesis:} and Opler~\cite{opler:phd}.} while if $G_M$ is isomorphic to a cycle then $M$ is \emph{cyclic}.
\item When every component of $G_M$ is either unicyclic or a tree, we say that $M$ is a \emph{pseudoforest}.
\item If $G_M$ contains a component with more than one cycle, then $M$ is \emph{polycyclic}.
\end{itemize}\unskip
Note that every cycle in $G_M$ necessarily has even length at least four, since $G_M$ is a bipartite graph.

We also use the above terms to describe the grid class $\Grid(M)$. While acyclic classes are already well understood (both by Murphy and Vatter~\cite{murphy:profile-classes:}, and by the fact that all such classes are geometric grid classes so the results of~\cite{albert:geometric-grid-:} apply), additional tools are required to understand more complex row-column graphs.

Our attention in this paper focuses on the structure of permutations in cyclic, unicyclic and pseudoforest grid classes. By analysing the permutations in $\Grid(M)$ that are indecomposable with respect to the $M$-sum (Definition~\ref{def-m-sum}), we identify an unavoidable family of permutations called coils (Definition~\ref{de-coil}), formed by repeatedly traversing a chosen cycle of $M$ in a particular manner.
Harnessing the analysis in Section~\ref{sec-indiv-and-coils}, our main results are as follows.

\begin{description}
  \item[Theorem~\ref{thm-lwqo-char}] A subclass of a pseudoforest grid class is labelled well quasi-ordered if and only if it contains only bounded length coils.
  \item[Theorem~\ref{thm-lwqo-decidable}] It is possible to decide whether a finitely based subclass of a pseudoforest grid class is labelled well quasi-ordered.
  \item[Theorem~\ref{thm-finite-basis}] Every unicyclic grid class is finitely based.
  \item[Proposition~\ref{prop-bicyclic-inf-basis}] There exist pseudoforest grid classes that are not finitely based.
  \item[Theorem~\ref{thm-cycle-wqo-char}] A subclass of a cyclic grid class is well quasi-ordered if and only if it contains only finitely many end-inflated coils.
\end{description}

Proposition~\ref{prop-bicyclic-inf-basis} disproves a 2006 conjecture due to Huczynska and Vatter~\cite[Conjecture~2.3]{Huczynska2006}, which claimed that all monotone grid classes must be finitely based. On the other hand, Theorem~\ref{thm-finite-basis} establishes that the conjecture does hold in a smaller family of grid classes.

For the rest of this section we provide some further basic definitions, and set our work in the context of the existing literature. Section~\ref{sec-grid-class-structure} initiates our in-depth study of grid classes with further introductory material and some new concepts, such as the ``orientation digraph''. This section also contains a proof of a result concerning ``negative cycles'' that previously only existed in Waton's PhD thesis. Section~\ref{sec-indiv-and-coils} establishes our underpinning decomposition of pseudoforest grid classes by means of the $M$-sum, and introduces coils.

Section~\ref{sec:lwqo} covers all the introductory material concerning labelled and unlabelled well quasi-ordering, before presenting the proofs of Theorems~\ref{thm-lwqo-char} and~\ref{thm-lwqo-decidable}. Section~\ref{sec-basis} contains the proofs of Theorem~\ref{thm-finite-basis} and Proposition~\ref{prop-bicyclic-inf-basis}, and Section~\ref{sec-wqo} establishes Theorem~\ref{thm-cycle-wqo-char}.

\paragraph{Permutations and permutation classes}
The uninitiated reader may wish to refer to Bevan~\cite{bevan2015defs} or Vatter~\cite{vatter:survey} for a broader introduction to the area of permutation classes than that given here.
A permutation $\pi = \pi(1)\cdots\pi(n)$ of length $|\pi|=n$ is an ordering of the numbers $1,\dots,n$, though we frequently refer to a permutation $\pi$ by its point set $\{(i,\pi(i)) \,:\,i=1,\dots,n\}$. Indeed, this point set leads us to the crucial graphical perspective upon which the notion of grid classes depends: the \emph{plot} of $\pi$ is the plot of its point set in the plane.

Given two permutations $\pi$ and $\sigma$ of lengths $n$ and $k$, respectively, we say that $\sigma$ is \emph{contained} in $\pi$ if there exists a subsequence $i_1,\dots,i_k$ such that $1\leq i_1<i_2<\cdots<i_k\leq n$, and the sequence $\pi(i_1)\cdots\pi(i_k)$ forms a set of points in the same relative order as those in $\sigma$ (that is, $\pi(i_1)\cdots\pi(i_k)$ is \emph{order isomorphic} to $\sigma$). A witness of this containment forms a \emph{copy}, or an \emph{embedding}, of $\sigma$ in $\pi$. Containment forms a partial order on the set of all permutations, and when $\sigma$ is contained in $\pi$ we can write $\sigma\leq\pi$; if, on the other hand, there is no copy of $\sigma$ in $\pi$, then we write $\sigma\not\leq \pi$, and say that $\pi$ \emph{avoids} $\sigma$.

A \emph{permutation class} is a set of permutations that is closed under containment. That is, if $\pi$ belongs to a permutation class $\C$ and $\sigma\leq\pi$, then $\sigma\in\C$. The \emph{basis} of a class $\C$ is the set of minimal forbidden permutations not in $\C$; such a set is unique, but need not be finite.
For a set $B$ we write $\av(B)$ for the permutation class consisting of all permutations that avoid all elements of~$B$.
If $B$ is the basis of $\C$, then of course $\C=\av(B)$, and if $B$ is finite then we say that $\C$ is \emph{finitely based}.

\paragraph{The bases of grid classes} Prior to this work, relatively little general progress had been made on Huczynska and Vatter's conjecture, other than the results mentioned earlier for geometric grid classes in~\cite{albert:geometric-grid-:}.
In his PhD thesis~\cite[Theorem 4.7.5]{waton:on-permutation-:} Waton showed that one $2\times 2$ grid class is finitely based, and subsequently Albert and Brignall~\cite{ab:grid-basis} extended Waton's argument to cover all $2\times 2$ grid classes.

\paragraph{Well quasi-ordering in grid classes}
A permutation class is \emph{well quasi-ordered} if it contains no infinite antichains with respect to the containment ordering.
We defer the core definitions regarding this property and its labelled version until Section~\ref{sec:lwqo}, but give here some context for our work.
 In general, the property of being well quasi-ordered is seen as an indicator of `tameness', as described by Cherlin~\cite{cherlin:forbidden-subst:}. For example, every well quasi-ordered permutation class contains at most countably many distinct subclasses, whereas those that are not well quasi-ordered must have uncountably many subclasses (and, therefore, also uncountably many subclasses whose enumeration sequences are intractably difficult). As such, the presence (or otherwise) of infinite antichains dictates much of what we know regarding the structure and growth rates of the so-called ``small'' permutation classes, see Vatter~\cite{vatter:small-permutati:,vatter:growth-rates-of:}.

As mentioned earlier, Murphy and Vatter~\cite{murphy:profile-classes:} established that a grid class is well quasi-ordered if and only if its graph is acyclic. It may seem something of a surprise, therefore, to note how much of this paper is dedicated to the study of well quasi-ordering in grid classes. Our interest here is in the \emph{subclasses} of pseudoforest grid classes, and there are two principal reasons for this:
\begin{itemize}
	\item Our result that unicyclic grid classes are finitely based fundamentally relies on the characterisation of labelled well quasi-order given in Theorem~\ref{thm-lwqo-char}; in fact, our counterexamples to the 2006 conjecture are also unlikely to have been discovered without the development of $M$-sums and coils.
	\item Our work here is intended to initiate a programme of study similar to that undertaken for subclasses of the 321-avoiding permutations in Albert, Brignall, Ru\v{s}kuc and Vatter~\cite{abrv:321-subclasses:}.
Indeed, it is shown in~\cite{abrv:321-subclasses:} that every finitely based or well quasi-ordered subclass of $\av(321)$ is enumerated by a rational generating function, even though $\av(321)$ itself is not.\footnote{The 321-avoiding permutations are famously one of the classes of combinatorial structures enumerated by the Catalan numbers, which means it has an algebraic, but not rational, generating function.}

The comparison with $\av(321)$ is not a mere coincidence: although $\av(321)$ is not a monotone grid class in our sense, one could view it as the class of the ``infinite staircase'' matrix
\[
\setlength\arraycolsep{3pt}
\scriptsize\begin{pmatrix}
&&&&&&\iddots\\
&&&&1&1&\\
&&&1&1&&\\
&&1&1&&&\\
&1&1&&&&\\
\iddots&&&&&&
\end{pmatrix}.
\]
The resulting ``staircase decomposition'' of $\av(321)$, which is one of the main tools in~\cite{abrv:321-subclasses:}, very strongly resembles the decomposition we obtain here from coils.
\end{itemize}

We stop short of attempting to derive enumerative consequences of our work for subclasses of pseudoforest grid classes, but refer the curious reader to our concluding remarks on the topic.


%
%
%
%
%
%
%
%
\section{Grid class structure}\label{sec-grid-class-structure}

In this section, we review several concepts  from the literature (including some work that has previously only appeared in Waton's PhD thesis~\cite{waton:on-permutation-:}) relating to grid classes, and set these alongside a novel concept -- the orientation digraph.

\subsection{Grid classes and gridded permutations}

Recall that a \emph{gridding matrix}
is any matrix whose entries come from $\{0,\pm 1\}$.
In order to reflect the way we view permutations graphically, we deviate from the standard convention for indexing matrix entries:
we index them starting from the lower left corner, and we record the column number first, followed by the row number.
Thus, an $m\times n$ matrix has $m$ columns and $n$ rows, and, for example, $M_{21}$ is the entry in the second column from the left in the bottom row of the matrix $M$.

It is often helpful to adopt a more graphical representation of gridding matrices, by replacing the non-zero entries of each cell with small increasing or decreasing line segments as appropriate. For example, if $M={}${\setlength\arraycolsep{2pt}\small $\begin{pmatrix} -1&0&1\\[-2pt] 0&1&1\\[-2pt] 1&-1&0\end{pmatrix}$}, then we might alternatively write $M=\gcthree{3}{-1,0,1}{0,1,1}{1,-1,0}$.

Given an $m\times n$ gridding matrix $M$, an \emph{$M$-gridding} of a permutation $\pi$ of length $\ell$ is a division of the points in the plot of $\pi$ into a grid of $m\times n$ rectangles,  called \emph{cells}, such that the cell $(i,j)$ contains no points if $M_{ij}=0$, and otherwise the points in this cell form an increasing sequence if $M_{ij}=1$, or a decreasing sequence if $M_{ij}=-1$.
The latter two options also include the possibility of the cell containing no points.
More formally, such a gridding can be viewed as a pair
$(\VVV,\HHH)$ where $\VVV=\{v_1,\dots,v_{m-1}\}$ and $\HHH=\{h_1,\dots,h_{n-1}\}$ are collections of $m-1$ vertical and $n-1$ horizontal lines, represented by non-integers
\begin{gather*}
\tfrac12 \:=\: v_0 \:<\: v_1\:<\: \cdots \:<\: v_{m-1} \:<\: v_m \:=\: \ell+\tfrac12\text{, ~~and} \\
\tfrac12 \:=\: h_0 \:<\: h_1\:<\: \cdots \:<\: h_{n-1} \:<\: h_n \:=\: \ell+\tfrac12 .\phantom{\text{ ~~and}}
\end{gather*}
The permutation $\pi$ together with the gridding is referred to as a \emph{gridded permutation}, and we write
$\pi^\gridded=(\pi,\VVV,\HHH)$.
The rectangle $\{ (x,y)\::\: v_{i-1}<x<v_i,\ h_{j-1}<y<h_j\}$ will be referred to as the \emph{cell} $C_{ij}$.
Sometimes we will identify a cell with the set of all points of $\pi^\gridded$ that belong to $C_{ij}$.
A permutation $\pi\in\Grid(M)$ may possess several $M$-griddings; see Figure~\ref{fig:multi-grids}. We will occasionally use $\pi^\natural$ to refer to a second gridding of the permutation $\pi$.

The collection of all permutations that possess an $M$-gridding forms a permutation class known as the \emph{grid class} of $M$, which we denote by $\Grid(M)$.
If $\C$ is a permutation class such that every $\pi\in\C$ possesses an $M$-gridding,
i.e. if $\C\subseteq \Grid(M)$, then we say that $\C$ is $M$-\emph{griddable}.

\begin{figure}
{\centering
\begin{tikzpicture}[scale=0.3]
	\plotpermgrid{8,1,2,5,4,3,6,9,7}
	\draw[thick] (0.5,5.5) -- (9.5,5.5);
	\draw[thick] (0.5,7.5) -- (9.5,7.5);
	\draw[thick] (4.5,0.5) -- (4.5,9.5);
	\draw[thick] (6.5,0.5) -- (6.5,9.5);
\begin{scope}[shift={(14,0)}]
	\plotpermgrid{8,1,2,5,4,3,6,9,7}
	\draw[thick] (0.5,5.5) -- (9.5,5.5);
	\draw[thick] (0.5,7.5) -- (9.5,7.5);
	\draw[thick] (3.5,0.5) -- (3.5,9.5);
	\draw[thick] (7.5,0.5) -- (7.5,9.5);
\end{scope}
\begin{scope}[shift={(28,0)}]
	\plotpermgrid{8,1,2,5,4,3,6,9,7}
	\draw[thick] (0.5,4.5) -- (9.5,4.5);
	\draw[thick] (0.5,7.5) -- (9.5,7.5);
	\draw[thick] (3.5,0.5) -- (3.5,9.5);
	\draw[thick] (7.5,0.5) -- (7.5,9.5);
\end{scope}
\end{tikzpicture}\par}
\caption[]{The permutation $812543697$ possesses six griddings in $\Grid(M)$ where
$M={}$\protect\gcthree{3}{-1,0,1}{0,1,1}{1,-1,0},
three of which are shown here.}\label{fig:multi-grids}
\end{figure}

We denote the collection of all $M$-gridded permutations by $\Gridhash(M)$.
Note that the elements of $\Gridhash(M)$ are \emph{not} permutations, and if $\sigma^\gridded$, $\pi^\gridded$ are two $M$-gridded permutations such that $\sigma\leq \pi$, then it need not be the case that $\sigma$ can be embedded in $\pi$ in such a way as to respect the two given griddings. However, there is the notion of \emph{gridded containment}: we say that $\sigma^\gridded \leq \pi^\gridded$ if there exists a subsequence $1\leq i_1\leq\cdots\leq i_{|\sigma|}\leq |\pi|$ of $\pi$ that is order isomorphic to $\sigma$, such that the entries
$(j,\sigma(j))$ of $\sigma^\gridded$ and $(i_j,\pi(i_j))$ of $\pi^\gridded$
lie in the cells that correspond to the same entry of~$M$, for each $j=1,\dots,|\sigma|$.

Equipped with gridded containment, we observe that for any $M$-griddable class $\C\subseteq\Grid(M)$, the set of gridded permutations $\C^\gridded$ is a downset: that is, if $\pi^\gridded\in\C^\gridded$ and $\sigma^\gridded\le\pi^\gridded$, then $\sigma^\gridded\in\C^\gridded$.

%
%
%
\subsection{Partial multiplication matrices and griddings}\label{subsec-pmms}

Let $M$ be an $m\times n$ gridding matrix, and recall that the vertices of $G_M$ are
$\{1,\dots,m\} \,\cup\, \{1',\dots,n'\}$.
For any cycle $C$ in $G_M$, the \emph{sign} of the cycle is the product
$\prod_{i j'\in E(C)}M_{ij}$,
where $E(C)$ is the set of edges of $C$.
Thus, each cycle in $G_M$ is either positive or negative, and one of our aims in this subsection is to demonstrate that matrices whose graphs possess negative cycles can be removed from our considerations.


We say that an $m\times n$ gridding matrix $M$ is a \emph{partial multiplication matrix} if there exist sequences $c_1,\dots,c_m$ and $r_1,\dots,r_n$ with entries $\pm 1$ such that each non-zero entry $M_{ij}$ of $M$ is equal to $c_ir_j$. Such matrices are characterised by the following proposition.

\begin{prop}[Vatter and Waton~\cite{vatter:on-partial-well:}]\label{prop-pmm-neg-cycles}
A gridding matrix $M$ is a partial multiplication matrix if and only if its row-column graph $G_M$ contains no negative cycles.\end{prop}

The sequences $c_1,\dots,c_m$ and $r_1,\dots,r_n$ from the above definition will be called the \emph{column} and \emph{row} sequences, respectively. They are not uniquely determined by $M$:
for example, if $c_1,\dots,c_m$ and $r_1,\dots,r_n$ are column and row sequences for $M$, then so are $-c_1,\dots,-c_m$ and $-r_1,\dots,-r_n$.
In what follows, whenever we have a gridding matrix that is a partial multiplication matrix, we will assume that a particular choice of row and column sequences has been fixed.

Column and row sequences lead to the notion of
\emph{orientation} for cells, columns and rows.
For example, in the case that $c_i=r_j=1$, we order the points in cell $ij$ from bottom left (the \emph{first point}) to top right (the \emph{last point}), and we can succinctly denote this orientation by $\nearrow$. All four cases are given in the following table.\par
{\centering
\begin{tabular}{rr|c}
$c_i$&$r_j$&orientation of $M_{ij}$\\\hline
$1$&$1$&$\nearrow$\\
$-1$&$-1$&$\swarrow$\\
$1$&$-1$&$\searrow$\\
$-1$&$1$&$\nwarrow$
\end{tabular}\par}
Note that when $r_i=1$, then the vertical component of orientation for all non-empty cells in that row is the same, namely from bottom to top~($\uparrow$), and when $r_i=-1$, it is from top to bottom~($\downarrow$).
An analogous statement can be made in relation to the horizontal orientation within columns. See Figure~\ref{fig:orientation}.

\begin{figure}
	{\centering\begin{tikzpicture}[scale=0.3]
	\draw[->,ultra thick, black!20] (1.5,1.5) -- ++(2,2);
	\draw[<-,ultra thick, black!20] (5.5,3.5) -- ++(2,-2);
	\draw[<-,ultra thick, black!20] (5.5,5.5) -- ++(2,2);
	\draw[<-,ultra thick, black!20] (9.5,5.5) -- ++(2,2);
	\draw[<-,ultra thick, black!20] (9.5,11.5) -- ++(2,-2);
	\draw[->,ultra thick, black!20] (1.5,9.5) -- ++(2,2);
	\plotpermgrid{10,1,12,2,5,4,6,3,7,11,9,8}
	\draw[thick] (0.5,4.5) -- (12.5,4.5);
	\draw[thick] (0.5,8.5) -- (12.5,8.5);
	\draw[thick] (4.5,0.5) -- (4.5,12.5);
	\draw[thick] (8.5,0.5) -- (8.5,12.5);	
	\draw[->] (1.5,0) -- (3.5,0);
  \draw[<-] (5.5,0) -- (7.5,0);
  \draw[<-] (9.5,0) -- (11.5,0);
  \draw[->] (0,1.5) -- (0,3.5);
  \draw[<-] (0,5.5) -- (0,7.5);
  \draw[->] (0,9.5) -- (0,11.5);	
	\end{tikzpicture}\par}
\caption{A gridding of the permutation $10\,1\,12\,2\,5\,4\,6\,3\,7\,11\,9\,8$ in $\Grid(M)$ where $M={}$\protect\gcthree{3}{1,0,-1}{0,1,1}{1,-1,0} is a partial multiplication matrix with column sequence $1,-1,-1$ and row sequence $1,-1,1$. The inherited orientations for each non-empty cell are shown in grey.}\label{fig:orientation}
\end{figure}

We will use this notion of orientation to define (in Section~\ref{sec-indiv-and-coils}) a decomposition of gridded permutations called the $M$-sum. However, before we can do this, we need to consider matrices that \emph{cannot} be expressed as partial multiplication matrices, since we have not yet said how (or, indeed, whether it is possible) to orient these. 

For any $m\times n$ gridding matrix $M$, the \emph{doubling} of $M$, denoted $M^{\times 2}$, is the $2m\times 2n$ matrix obtained from $M$ using the substitution rules
\[ 0 \longmapsto \begin{pmatrix}
	0&0\\0&0
\end{pmatrix} \qquad
1 \longmapsto \begin{pmatrix}
	0&1\\1&0
\end{pmatrix} \qquad
-1 \longmapsto \begin{pmatrix}
	-1&0\\0&-1
\end{pmatrix}.
\]
Note that each row of $M$ gives rise to two rows of $M^{\times 2}$, and the number of non-zero entries in each of the two rows of $M^{\times 2}$ is the same as the number in the corresponding row of $M$. A similar comment applies to the columns of $M$ and $M^{\times 2}$.

For example, if $M=\begin{pmatrix}1&1\\1&1\end{pmatrix}$ and $N=\begin{pmatrix}1&-1\\1&1\end{pmatrix}$, then
\[M^{\times2} = \begin{pmatrix}
	0&1&0&1\\
	1&0&1&0\\
	0&1&0&1\\
	1&0&1&0
\end{pmatrix}
\qquad
N^{\times2} = \begin{pmatrix}
	0&1&-1&0\\
	1&0&0&-1\\
	0&1&0&1\\
	1&0&1&0
\end{pmatrix}.\]
Although $G_M$ and $G_N$ are both isomorphic to the 4-cycle $C_4$, the graphs $G_{M^{\times 2}}$ and $G_{N^{\times 2}}$ are not isomorphic: specifically, $G_{N^{\times 2}}$ is a cycle of length 8 (of positive sign), while $G_{M^{\times 2}}$ comprises two disjoint copies of $C_4$. Indeed, the submatrix formed on the rows and columns of $M^{\times 2}$ corresponding to the vertices from either of the two components of $G_{M^{\times 2}}$ is equal to $M$. These observations hold more generally.


\begin{prop}\label{prop-cycles-in-doubled-matrix}
	Let $M$ be a gridding matrix. If  $C$ is is a cycle of length $\ell$ in the graph $G_M$  then the following hold:
\begin{enumerate}[(i)]
	\item If $C$ has positive sign, then the vertices in $G_{M^{\times 2}}$ that arise from $C$ form two disjoint cycles each of length $\ell$ and of positive sign.
	\item If $C$ has negative sign, then the vertices in $G_{M^{\times 2}}$ that arise from $C$ form a single cycle of length $2\ell$, of positive sign.
\end{enumerate}
Furthermore, if $M$ is a pseudoforest, then every cycle of $M^{\times 2}$ arises in this way.
\end{prop}

\begin{proof}
Denote the vertices around the cycle $C$ of $G_M$ by $x_1,\dots,x_\ell$. Each vertex $x_i$ gives rise to two vertices in the row-column graph of $G_{M^{\times 2}}$, which we will label as follows: If $x_i$ is a row vertex, then we denote the lower of the two corresponding vertices by $x_i^A$ and say it is of \emph{type} $A$, and the upper by $x_i^B$ (of type $B$), while if $x_i$ is a column vertex, then $x_i^A$ (type $A$) refers to the leftmost, and $x_i^B$ (type $B$) the rightmost.

Now consider a non-zero entry $M_{pq}$ that corresponds to the edge between $x_i$ and $x_{i+1}$, where we reduce indices modulo $\ell$ as appropriate.%
\footnote{Throughout this paper, we will adopt the convention that $(\text{mod~}\ell)$ reduces an integer to the set of residues $\{1,\dots, \ell\}$, rather than to $\{0,\dots,\ell-1\}$.} %
If $M_{pq}=1$, then in $G_{M^{\times 2}}$ we find the edges $x_i^A x_{i+1}^A$ and $x_i^B x_{i+1}^B$, while if $M_{pq}=-1$ then we instead have the edges $x_i^A x_{i+1}^B$ and $x_i^B x_{i+1}^A$.

Now consider the walk, starting from $x_1^A$, that sequentially follows these edges in turn, until we return to one of the two vertices $x_1^A$ or $x_1^B$. Note that this walk moves between vertices of type $A$ and type $B$ precisely when the corresponding edge of $G_M$ arises from an entry of $M$ that is equal to $-1$.

Consequently, if the cycle $C$ has positive sign, then the walk switches between vertices of types $A$ and $B$ an even number of times, and thus returns to $x_1^A$ after $\ell$ steps. Similarly, the walk that starts at $x_1^B$ will return to $x_1^B$ after $\ell$ steps. Thus we have two cycles of length $\ell$, both of positive sign.

 On the other hand, if the cycle $C$ has negative sign, then the walk makes an odd number of switches between vertices of type $A$ and type $B$, and thus reaches vertex $x_1^B$ after $\ell$ steps. Furthermore, the walk starting at $x_1^B$ reaches vertex $x_1^A$ after $\ell$ steps, and by combining these two walks we find a cycle of length $2\ell$. This cycle necessarily passes through all of the $2\ell$ vertices in $G_{M^{\times2}}$ arising from the vertices on the cycle $C$ of $G_M$, and since this collection of $2\ell$ vertices contains an even number of $-1$s, the cycle has positive sign.
 
 For the final assertion, the case when $G_M$ is connected is verified by direct inspection. For the general case we observe that $G_{M^{\times 2}}$ is the disjoint union of the graphs $G_{N^{\times 2}}$ where $N$ runs through all submatrices of $M$ corresponding to the connected components of $G_M$.
\end{proof}

The following result, and the proof we present, first appeared in Waton's PhD thesis.

\begin{prop}[Waton~{\cite[Theorem 4.5.8]{waton:on-permutation-:}}]\label{prop-doubling-equal}
Let $M$ be a gridding matrix. If every component of $G_M$ that contains a negative cycle contains no other cycle, then $\Grid(M) = \Grid(M^{\times 2})$.
\end{prop}

\begin{proof}
Suppose $M$ has $m$ columns and $n$ rows.
First, observe that if $\pi\in\Grid(M^{\times 2})$, then there is some $\pi^\gridded = (\pi,\VVV,\HHH)\in\Gridhash(M^{\times 2})$ where $\VVV=\{v_1,\dots,v_{2m-1}\}$ and $\HHH=\{h_1,\dots,h_{2n-1}\}$.
The gridding $(\pi,\VVV',\HHH')$ where $\VVV'=\{v_2,v_4,\dots,v_{2m-2}\}$ and $\HHH' =\{h_2,h_4,\dots,h_{2n-2}\}$ (that is, the sets formed by removing all the odd-numbered lines from $\VVV$ and $\HHH$) demonstrates that $\pi\in\Grid(M)$.
Thus $\Grid(M^{\times 2})\subseteq\Grid(M)$.

For the other direction, without loss of generality, we may assume that $G_M$ contains a single component, otherwise the following argument can be applied to each component separately.

First, if $M$ contains no negative cycle, then by Proposition~\ref{prop-pmm-neg-cycles}, $M$ can be expressed as a partial multiplication matrix, and we identify sequences $c_1,\dots,c_m$ and $r_1,\dots,r_n$ that witness this fact.
Define maps $\kappa:[m]\to[2m]$ and $\rho:[n]\to [2n]$ as follows. For $i\in [m]$, $j\in[n]$,
\begin{align*}
\kappa(i) &= \begin{cases}2i-1&\text{if }c_i=1\\
 2i&\text{otherwise.}	
 \end{cases}
\\
\rho(j) &= \begin{cases}2j-1&\text{if }r_j=1\\
 2j&\text{otherwise.}	
 \end{cases}
\end{align*}
We have $M_{ij}=M^{\times 2}_{\kappa(i)\rho(j)}$, and furthermore the submatrix of $M^{\times 2}$ defined on the columns $\{\kappa(1),\dots,\kappa(m)\}$ and the rows $\{\rho(1),\dots,\rho(n)\}$ is equal to $M$.
Thus, any permutation that belongs to $\Grid(M)$ also belongs to $\Grid(M^{\times 2})$.

It remains to consider the case that $G_M$ contains a negative cycle, which we take to have length $2\ell$.
Consider $\pi\in\Grid(M)$, and fix a specific gridding $\pi^\gridded=(\pi,\VVV,\HHH)\in\Gridhash(M)$. Our aim is to refine this gridding by slicing each row and column of $\pi^\gridded$ in order to show that $\pi\in\Grid(M^{\times 2})$, but some care is required. Suppose, for example, that we have identified a horizontal line that slices some row of $\pi^\gridded$. For any non-empty cell $c$ we must identify a vertical line to slice $c$, so that the points in $c$ occupy only the lower-left and upper-right regions if they form an increasing sequence, and the upper-left and lower-right regions if they form a decreasing sequence. The vertical line we identify slices the whole column containing $c$, and for any other non-empty cells in this column we will similarly need to identify suitable horizontal lines. We call this process \emph{propagation}, and it presents no issues unless we revisit an already-sliced cell, at which point we need to ensure that this cell is sliced only once horizontally and once vertically. This situation arises when we propagate slices around the cells corresponding to the edges of the cycle of $G_M$, so our first task is to show how to slice the cells of the negative cycle.

For $i\in[m],j\in[n]$ such that $M_{ij}\neq 0$, define the continuous function $f_{ij} :[v_{i-1},v_i]\to [h_{j-1},h_j]$ formed by the piecewise linear map between the entries of $\pi$ in cell $ij$, together with the corners of the cell, namely $(v_{i-1},h_{j-1})$, $(v_i,h_j)$ if $M_{ij}=1$, and $(v_{i-1},h_{j})$, $(v_i,h_{j-1})$ if $M_{ij}=-1$. See Figure~\ref{fig-piecewise}, and note that $f_{ij}$ is monotone increasing if $M_{ij}=1$, and monotone decreasing if $M_{ij}=-1$.

\begin{figure}
{\centering
\begin{tikzpicture}[scale=0.3]
  \plotpermgrid{2,11,14,16,3,6,4,7,1,8,12,5,9,13,10,15}
  \draw[thick] (5.5,0.5) -- (5.5,16.5);
  \draw[thick] (10.5,.5) -- (10.5,16.5);
  \draw[thick] (.5,4.5) -- (16.5,4.5);
  \draw[thick] (.5,10.5) -- (16.5,10.5);
  \draw[black!50,thin] (0.5,10.5) -- (11) -- (14) -- (16) -- (5.5,16.5);
  \draw[black!50,thin] (0.5,0.5) -- (2) -- (3) -- (5.5,4.5)
  						 -- (4) -- (1) -- (10.5,0.5);
  \draw[black!50,thin] (5.5,4.5) -- (6) -- (7) -- (8) -- (10.5,10.5)
 						 -- (12) -- (13) -- (15) -- (16.5,16.5);
  \draw[black!50,thin] (10.5,4.5) -- (5) -- (9) -- (10) -- (16.5,10.5);
  \node[empty] at (3,1.5) {$f_{11}$};
  \node[empty] at (3.5,12) {$f_{13}$};
  \node[empty] at (9,3) {$f_{21}$};
  \node[empty] at (8,8.3) {$f_{22}$};
  \node[empty] at (13.5,7) {$f_{32}$};
  \node[empty] at (14,14.5) {$f_{33}$};
\end{tikzpicture}\par}
\caption{The piecewise maps $f_{ij}$, described in the proof of Proposition~\ref{prop-doubling-equal}, for a gridding of the permutation 2\,11\,14\,16\,3\,6\,4\,7\,1\,8\,12\,5\,9\,13\,10\,15 in $\Grid(M)$, where $M={}$\protect\gcthree{3}{1,0,1}{0,1,1}{1,-1,0}.}\label{fig-piecewise}
\end{figure}

Denote the negative cycle of $M$ by a sequence of alternating rows and columns, \[i_1,j_1,i_2,j_2,\dots,i_\ell,j_\ell.\]
Thus the cells of the cycle in $\pi^\gridded$ are $i_1j_1$, $i_2j_1$, $i_2j_2$, \dots, $i_\ell j_\ell$, $i_1 j_\ell$.
Define the map $f: [v_{j_1-1},v_{j_1}]\to [v_{j_1-1},v_{j_1}]$ as the composition (taken from right to left)
\[f = f_{i_1 j_\ell}^{-1} \circ \cdots \circ f_{i_2j_2}\circ f^{-1}_{i_2j_1}\circ f_{i_1j_1},\]
formed by following the functions and their inverses around the cycle.
Note that since the cycle has negative sign, the function $f$ is a decreasing function. As such, there exists $x_1\in (v_{j_1-1},v_{j_1})$ such that $f(x_1)=x_1$ (this follows, for example, by applying the Intermediate Value Theorem to the function $f(x)-x$). Let $y_1 = f_{i_1j_1}(x_1)$, and for $k=2,\dots,\ell$ iteratively define
\[ x_{k} = f_{i_{k}j_{k-1}}^{-1}(y_{k-1}),\qquad y_{k} = f_{i_{k}j_{k}}(x_{k}).\]
Note that, for $1\le k \le \ell$, by construction $y_k$ is a point on the vertical axis (which we can associate with a horizontal line) that is propagated by $x_k$, and $x_{k+1\pmod{\ell}}$ is a point on the horizontal axis (associated with a vertical line) that is propagated by $y_k$ (since $x_1 = f_{i_1 j_\ell}(y_\ell)$).
Thus, the $2\ell$ vertical and horizontal lines defined by $x_1,\dots,x_\ell$ and $y_1,\dots,y_\ell$ slice the rows and columns of $\pi^\gridded$ corresponding to the vertices on the cycle of $G_M$, as required. Note that if any point of $\pi^\gridded$ shares one or both coordinates with these lines, then we can adjust the coordinates of the point by a small amount so that it still lies on the curve defined by the appropriate piecewise map $f_{ij}$, but no longer meets any of the horizontal or vertical lines, and so that the plot is still order isomorphic to $\pi^\gridded$.

It remains to show that any non-empty cells of $\pi^\gridded$ not on the cycle can be sliced both horizontally and vertically which we establish by induction on the number of such cells. In the case there are no such cells, then all non-empty cells of $\pi^\gridded$ lie on the cycle, and are sliced as described above. So we now suppose there is at least one cell not on the cycle.

Identify a non-empty cell $c$ of $\pi^\gridded$ that either has no other non-empty cells in its column or its row. (Such a cell exists since any cell that is not on the cycle corresponds to an edge in $G_M$ on a unique path from the cycle to some leaf of $G_M$, and hence we may choose the last non-empty cell along this path.) By symmetry, we may assume that $c$ is the only non-empty entry in its column, and furthermore that the points in $c$ form an increasing sequence.

By induction, the gridded subpermutation of $\pi^\gridded$ formed by removing the entries of $c$ has a refinement that gives an $M^{\times 2}$-gridding. In $\pi^\gridded$, this refinement creates a horizontal line $h$ that slices the cell $c$. We now propagate in $c$: choose a vertical line $v$ slicing $c$ so that each entry in $c$ belongs to the lower-left or upper-right regions. Since there are no other non-empty cells in the column, the line $v$ slices no other non-empty cells of $\pi^\gridded$, and thus by including the line $v$ we have an $M^{\times 2}$-gridding of $\pi$, as required.
\end{proof}

The previous three propositions now enable us to conclude the following.

\begin{prop}\label{prop-pseudoforest-on-a-pmm}
If $M$ is an acyclic, unicyclic or pseudoforest gridding matrix, then there exists, respectively, an acyclic, unicyclic or pseudoforest partial multiplication matrix $N$ such that $\Grid(M)=\Grid(N)$.
\end{prop}


\begin{proof}
If $M$ is acyclic
we can take $N=M$ by
Proposition~\ref{prop-pmm-neg-cycles}. 
The same is true when $M$ is unicyclic and the unique cycle is positive.
When $M$ is unicyclic and the unique cycle is negative we can take $N=M^{\times 2}$.
Indeed, $M^{\times 2}$ has a unique cycle which is positive by Proposition~\ref{prop-cycles-in-doubled-matrix}, is connected by inspection, and $\Grid(M^{\times 2})=\Grid(M)$ by Proposition~\ref{prop-doubling-equal}.
Finally, if $M$ is a pseudoforest, $G_{M^{\times 2}}$ is the disjoint union of the 
graphs $G_{P^{\times 2}}$, where $P$ runs through all submatrices of $M$ corresponding to the connected components of $G_M$, and we can again take $N=M^{\times 2}$.
\end{proof}

In light of this, without loss of generality we may now restrict our attention only to partial multiplication matrices.
%
%
%
%
%
%
%
%
%
\section{Indivisibility and coils}\label{sec-indiv-and-coils}

In this section we develop a detailed structural understanding of grid classes through a method of decomposition that we call the $M$-sum. This method can be viewed of as an adaptation to gridded permutations of the classical constructions of direct and skew sums.

\subsection{\texorpdfstring{$M$}{M}-sums and \texorpdfstring{$M$}{M}-indivisibility}\label{subsec-m-sums}

Let $M$ be an $m\times n$ partial multiplication matrix, and fix column and row sequences $c_1,\dots,c_m$ and $r_1,\dots,r_n$, respectively. As noted in the proof of Proposition~\ref{prop-doubling-equal}, the submatrix of $M^{\times 2}$ defined on the columns $\kappa(1),\dots,\kappa(m)$ and rows $\rho(1),\dots,\rho(n)$ is equal to $M$; we call this the \emph{first} copy of $M$ in $M^{\times 2}$. The submatrix defined on the other rows and columns of $M^{\times 2}$ is also equal to $M$, and we call this the \emph{second} copy. Together these two submatrices include all the non-zero elements of $M^{\times 2}$: Each non-zero entry $M_{ij}$ of $M$ corresponds to two non-zero entries in $M^{\times 2}$, one belonging to each of the two submatrices.

\begin{defn}[$M$-sum]\label{def-m-sum}
Let $M$ be an $m\times n$ partial multiplication matrix with fixed column and row sequences, and let $\sigma^\gridded$ and $\tau^\gridded$ be two $M$-gridded permutations. The \emph{$M$-sum} of $\sigma^\gridded$ and $\tau^\gridded$, denoted $\sigma^\gridded\boxplus\tau^\gridded$, is the $M$-gridded permutation obtained from the following process: in a $2m\times 2n$ grid (which we associate with the doubled matrix $M^{\times 2}$), insert the points of $\sigma^\gridded$ into the cells that correspond to the entries of the first copy of $M$, and the points of $\tau^\gridded$ into the cells corresponding to the second copy. Now remove the odd-numbered grid lines to be left with the $M$-gridded permutation $\sigma^\gridded\boxplus\tau^\gridded$.

We say that a gridded permutation $\pi^\gridded$ is \emph{$M$-divisible} if it can be expressed as the $M$-sum of two non-empty gridded permutations, and \emph{$M$-indivisible} otherwise. Where the matrix $M$ is clear from the context, we may simply refer to \emph{divisible} and \emph{indivisible} permutations.
\end{defn}

Alongside this notion of $M$-divisibility, it will be convenient to construct an auxiliary directed graph in the following way. Let $M$ be a partial multiplication matrix, and let $\pi^\gridded$ be an $M$-gridded permutation. The \emph{orientation digraph} $D_{\pi^\gridded}$ has vertex set equal to the points of $\pi^\gridded$, with directed edges as follows: $x\rightarrow y$ if $x$ and $y$ share a row and/or column in $\pi^\gridded$, and $x$ precedes $y$ in the common orientation of these two points. Thus, the directed edges between the points of $D_{\pi^\gridded}$ within a single row or column define a total order. See Figure~\ref{fig:digraph}.

\begin{figure}
	{\centering\begin{tikzpicture}[scale=0.4]
	\plotpermgrid{4,1,5,3,6,2}
	\draw[thick] (0.5,3.5) -- (6.5,3.5);
	\draw[thick] (2.5,0.5) -- (2.5,6.5);
	\draw[->] (.75,0) -- (2.25,0);
  \draw[<-] (3.5,0) -- (5.5,0);
  \draw[->] (0,1) -- (0,3);
  \draw[<-] (0,4) -- (0,6);
  \begin{scope}[shift={(10,0)}]
	\plotperm{4,1,5,3,6,2}
\foreach \s/\t in {1/2,1/3,2/3,2/5,2/6,3/5,4/1,5/4,6/3,6/5}
	\draw (\s) edge[mid arrow] (\t);
  \draw (6) edge[mid arrow,out=180,in=60] (4);
  \end{scope}
	\end{tikzpicture}\par}
\caption{A gridding of the permutation $415362$ in $\Grid(M)$ where $M={}$\protect\gctwo{2}{-1,1}{1,-1} is a partial multiplication matrix, together with its orientation digraph.}\label{fig:digraph}
\end{figure}

Given any gridded permutation $\pi^\gridded$ and a gridded subpermutation $\sigma^\gridded$ obtained by removing a single point $p$, the digraph $D_{\sigma^\gridded}$ is an induced subdigraph of $D_{\pi^\gridded}$ obtained by removing the vertex $p$. Iterating this process, we have the following.

\begin{obs}\label{obs-pi-to-D-pi}
	The mapping $\pi^\gridded \mapsto D_{\pi^\gridded}$ is order-preserving.
\end{obs}

Note that this mapping from gridded permutations to directed graphs is not injective.

The orientation digraph $D_{\sigma^\gridded\boxplus\tau^\gridded}$ of an $M$-sum has the property that for any $x\in \sigma^\gridded$ and $y\in\tau^\gridded$, either $x\rightarrow y$ or there is no edge between $x$ and $y$. Consequently, $D_{\sigma^\gridded\boxplus\tau^\gridded}$ is not strongly connected. Conversely, if the orientation digraph $D_{\pi^\gridded}$ of some $M$-gridded permutation $\pi^\gridded$ is not strongly connected, then we may express $\pi^\gridded$ as an $M$-sum of two smaller permutations by considering a partition of the vertices of $D_{\pi^\gridded}$ into non-empty sets $S$ and $T$ such that there is no edge starting in $T$ and ending in~$S$. Thus we have:

\begin{lemma}\label{lem-Dpi-strongly-connected}
Let $M$ be a partial multiplication matrix, and let $\pi^\gridded\in\Gridhash(M)$. Then $\pi^\gridded$ is indivisible if and only if $D_{\pi^\gridded}$ is strongly connected.
\end{lemma}

As a consequence of Lemma~\ref{lem-Dpi-strongly-connected}, we make a few remarks. First, we make a simple observation concerning subpermutations of divisible permutations.

\begin{obs}\label{obs-divisible-subpermutation}
Let $\pi^\gridded=\sigma^\gridded\boxplus\tau^\gridded$ be a divisible $M$-gridded permutation. Any gridded subpermutation of $\pi^\gridded$ that contains at least one point from each of $\sigma^\gridded$ and $\tau^\gridded$ is divisible.
\end{obs}

Next, by considering the decomposition of $D_{\pi^\gridded}$ into strongly connected components, we can express any $\pi^\gridded$ as an $M$-sum of indivisibles. This decomposition is illustrated in Figure~\ref{fig-decomposing}.

\begin{lemma}\label{lem-grid-decomp}
	Let $\pi^\gridded$ be an $M$-gridded permutation. Then $\pi^\gridded$ can be expressed as an $M$-sum of indivisible permutations,
	\[\pi^\gridded = \pi_1^\gridded \boxplus \cdots \boxplus\pi_k^\gridded.\]
	Furthermore, any other expression of $\pi^\gridded$ as an $M$-sum of indivisible permutations,
	\[\pi^\gridded = \varphi_1^\gridded \boxplus \cdots \boxplus\varphi_{\ell}^\gridded,\]
	satisfies $k=\ell$ and $\{\varphi_1^\gridded,\dots,\varphi_{\ell}^\gridded\} = \{\pi_1^\gridded,\dots,\pi_k^\gridded\}$ as multisets.
	\end{lemma}

\begin{figure}
\begin{center}
\begin{tabular}{m{4.2cm}m{0.3cm}m{2.8cm}m{0.3cm}m{1.0cm}m{0.3cm}m{1.9cm}}
\begin{tikzpicture}[scale=0.25]
  \foreach \x/\y/\z/\w in {1/1/3/2,9/1/10/2,9/8/10/10,13/8/16/10,13/13/16/16,1/13/3/16}
      \draw[blue, thin, fill=blue!10] (\x-0.5,\y-0.5) rectangle (\z+0.5,\w+0.5);
  \foreach \x/\y/\z/\w in {8/7/8/7,13/7/12/7,8/3/8/2}
      \draw[green, thin, fill=green!10] (\x-0.5,\y-0.5) rectangle (\z+0.5,\w+0.5);
  \foreach \x/\y/\z/\w in {4/3/5/4,6/3/7/4,6/5/7/6,11/5/12/6,11/11/12/12,4/11/5/12}
      \draw[red, thin, fill=red!10] (\x-0.5,\y-0.5) rectangle (\z+0.5,\w+0.5);
  \plotpermgrid{2,16,14,11,3,6,4,7,1,8,12,5,9,13,10,15}
  \draw[thick] (5.5,0.5) -- (5.5,16.5);
  \draw[thick] (10.5,.5) -- (10.5,16.5);
  \draw[thick] (.5,4.5) -- (16.5,4.5);
  \draw[thick] (.5,10.5) -- (16.5,10.5);
  \draw[->] (1.5,0) -- (4.5,0);
  \draw[<-] (6.5,0) -- (9.5,0);
  \draw[<-] (11.5,0) -- (15.5,0);
  \draw[->] (0,1.5) -- (0,3.5);
  \draw[<-] (0,5.5) -- (0,9.5);
  \draw[<-] (0,11.5) -- (0,15.5);
\end{tikzpicture}%
&$=$&
\begin{tikzpicture}[scale=0.3]
\draw[fill=blue!10,draw=none] (0.5,0.5) rectangle (9.5,9.5);
\plotpermgrid{2,9,7,1,3,4,6,5,8}
\draw[thick] (3.5,0.5) -- (3.5,9.5);
\draw[thick] (5.5,0.5) -- (5.5,9.5);
\draw[thick] (0.5,2.5) -- (9.5,2.5);
\draw[thick] (0.5,5.5) -- (9.5,5.5);
\end{tikzpicture}
&$\boxplus$&%
\begin{tikzpicture}[scale=0.3]
\draw[fill=green!10,draw=none] (0.5,0.5) rectangle (3.5,3.5);
\draw[step=1cm,gray!50,very thin] (0.5,0.5) grid (3.5,3.5);
\node[permpt] at (2,2) {};
\draw[thick] (1.5,0.5) -- (1.5,3.5);
\draw[thick] (2.5,0.5) -- (2.5,3.5);
\draw[thick] (0.5,1.5) -- (3.5,1.5);
\draw[thick] (0.5,2.5) -- (3.5,2.5);
\end{tikzpicture}
&$\boxplus$&
\begin{tikzpicture}[scale=0.3]
\draw[fill=red!10,draw=none] (0.5,0.5) rectangle (6.5,6.5);
\plotpermgrid{5,1,4,2,6,3}
\draw[thick] (2.5,0.5) -- (2.5,6.5);
\draw[thick] (4.5,0.5) -- (4.5,6.5);
\draw[thick] (0.5,2.5) -- (6.5,2.5);
\draw[thick] (0.5,4.5) -- (6.5,4.5);
\end{tikzpicture}%
\end{tabular}
\end{center}
\caption{For $M={}$\protect\gcthree{3}{-1,0,1}{0,1,1}{1,-1,0}, the decomposition into $M$-indivisibles of the $M$-gridded permutation $\pi^\gridded = 2\,16\,14\,11\,3\,6\,4\,7\,1\,8\,12\,5\,9\,13\,10\,15$. 
}
\label{fig-decomposing}
\end{figure}

We will refer to any decomposition of $\pi^\gridded$ into $M$-indivisibles as guaranteed by Lemma~\ref{lem-grid-decomp} as an \emph{$M$-decomposition of $\pi^\gridded$}.
Note that in the $M$-decomposition $\pi^\gridded = \pi_1^\gridded \boxplus \cdots \boxplus\pi_k^\gridded$, the order in which the indivisibles appear may not be uniquely determined. For example, if the points of $\pi_i^\gridded $ and $\pi_{i+1}^\gridded$ lie in cells that correspond to edges of $G_M$ in different components, then $\pi_i^\gridded\boxplus \pi_{i+1}^\gridded = \pi_{i+1}^\gridded\boxplus \pi_{i}^\gridded$. This gives:

\begin{obs}\label{obs-single-component}
	The points of any $M$-indivisible permutation must belong to cells that induce a connected subgraph of $G_M$.
\end{obs}

For ease of presentation, we will refer to $M$-indivisibles as being \emph{associated with} a component of $G_M$.

In any strongly connected component of some $D_{\pi^\gridded}$ that contains a directed cycle (and hence more than one point), the points on the directed cycle must belong to cells that correspond to edges of $G_M$ in some closed path. This gives:

\begin{obs}\label{obs-non-singleton}
	Any non-singleton $M$-indivisible must contain at least one point from each cell around some cycle of $M$.
\end{obs}


Before we proceed, we note here that the $M$-indivisible permutations that belong to the \emph{geometric} grid class
 $\Geom(M)$, which was discussed in the introduction, 
 are necessarily singletons. It is essentially this fact (although cast in different terms) that underpins the key results concerning geometric grid classes from~\cite{albert:geometric-grid-:}.

With Observation~\ref{obs-non-singleton} in mind, we now uncover some more information about the structure of non-singleton $M$-indivisibles. Recall that the last point of a cell is the final one according to the orientation of that cell.

\begin{lemma}\label{lem-lastpoints}Let $M$ be a partial multiplication matrix, and let $\pi^\gridded$ be a non-singleton indivisible $M$-gridded permutation. Then there exists an indivisible subpermutation of $\pi^\gridded$ formed from the last points in cells of $\pi^\gridded$ that correspond to the edges of some cycle of $G_M$.\end{lemma}

\begin{proof}
Consider the permutation $\alpha^\gridded$, formed by taking the last point from every non-empty cell of $\pi^\gridded$. We first show that $\alpha^\gridded$ must contain a non-singleton indivisible.

Suppose to the contrary that $\alpha^\gridded$ can be decomposed as an $M$-sum of singleton gridded permutations, $\alpha^\gridded = \alpha_1^\gridded \boxplus \cdots \boxplus\alpha_k^\gridded$ with each gridded $\alpha_i^\gridded$ being a gridding of the permutation 1. Now let $a_i$ denote the vertex in $D_{\pi^\gridded}$ corresponding to the singleton entry of $\alpha_i^\gridded$. Note that $a_i\rightarrow a_j$ only if $i<j$.

%

Since $D_{\pi^\gridded}$ is strongly connected, there exists $v\in D_{\pi^\gridded}$ such that $a_k\rightarrow v$.
As $a_k$  is the last entry in its cell, $v$ cannot lie in that same cell.
Let $a_i$ (for some $i\neq k$) be the last entry in the cell containing $v$, so that either $v=a_i$, or $v\rightarrow a_i$. In either case, since $v$, $a_i$ and $a_k$ share a row or column, we conclude that $a_k\rightarrow a_i$, a contradiction with $i<k$. Thus indeed $\alpha^\gridded$ contains a non-singleton indivisible.

The digraph of this non-singleton indivisible in $\alpha^\gridded$ is strongly connected, and so it contains an induced subdigraph that is isomorphic to a directed cycle, which we write as $b_1\rightarrow b_2\rightarrow\cdots\rightarrow b_\ell\rightarrow b_1$ for some $\ell$. The corresponding gridded permutation $\beta^\gridded$ on the points $b_1,\dots,b_\ell$ is indivisible, and by Observation~\ref{obs-non-singleton} it has at least one point (and hence, by the definition of $\alpha^\gridded$, exactly one point) in each cell around some cycle of $G_M$.
But this cycle must consist of all of the $\ell$ non-empty cells in $\beta^\gridded$, or else $D_{\beta^\gridded}$ would have a proper subdigraph isomorphic to a cycle, which is not possible because $D_{\beta^\gridded}$ is a 
cycle.
Thus, $\beta^\gridded$ satisfies the claim in the lemma.
\end{proof}

As a consequence of Lemma~\ref{lem-lastpoints}, for a fixed cycle in a grid $M$ there are exactly two $M$-indivisible permutations that contain precisely one entry in each cell of the cycle, formed by traversing the cycle in one direction or the other. Since the only proper indivisible subpermutations of these $M$-indivisibles are singletons, we will refer to them as the \emph{minimal} indivisibles. For example, in Figure~\ref{fig-decomposing}, the third
component in the $M$-decomposition
is one of the two minimal indivisible permutations.

%
%
%
%
\subsection{Coils}\label{subsec-coils}

In this subsection and the next, we give a description of the structure of $M$-indivisible permutations in the specific case that $M$ is a pseudoforest partial multiplication matrix. Our aim is Theorem~\ref{thm-coil-regridding}, which establishes that an $M$-indivisible permutation $\pi^\gridded$ can be regridded as an element of some $\Gridhash(N)$ for an \emph{acyclic} matrix $N$, the dimensions of which are determined by the size of a substructure of $\pi^\gridded$ known as a \emph{coil}.

This description of $M$-indivisible permutations can in part be regarded as a generalisation of the decomposition given by Albert and Vatter~\cite{albert:generating-and-:} of the skew-merged permutations, \[\av(2143,3412) = \grid{\gctwo{2}{-1,1}{1,-1}}.\] Note that this class is the grid class of a pseudoforest (indeed, cyclic) partial multiplication matrix.

To begin with, we restrict our attention to cyclic gridding matrices. Thus, throughout this subsection and unless stated otherwise, we fix $M$ to be a partial multiplication matrix whose graph $G_M$ is a cycle of length $\ell$. Recall that $\ell\ge 4$, and is even.

\begin{defn}
\label{de-coil}
A \emph{gridded $M$-coil} (or, where the context is clear, \emph{gridded coil}) is an $M$-gridded permutation $\pi^\gridded$ of length $n>\ell$ for which there exists an ordering $v_1,\dots, v_n$ on the vertices $D_{\pi^\gridded}$, and a labelling $1,\dots,\ell$ of the non-empty cells such that:
\begin{enumerate}[label=\textbf{C\arabic*}]
	\item\label{c1} $v_i$ lies in cell $i\pmod{\ell}$ for all $1\le i\le n$;
	\item\label{c2} $v_{i-1}\rightarrow v_i$ for all $1 < i\le n$;
	\item\label{c3} $v_i\rightarrow v_{i-\ell-1}$ for all $\ell+1<i\le n$;
	\item\label{c4} $v_{\ell+1}\rightarrow v_1$.
\end{enumerate}
An \emph{$M$-coil} (or simply \emph{coil}) is an ungridded permutation $\pi$ with the property that some $M$-gridding $\pi^\gridded$ of $\pi$ is a gridded coil.
\end{defn}

It is worth noting that coils and gridded coils are conceptually strongly related to the sequences of points used by Murphy and Vatter~\cite{murphy:profile-classes:} to construct infinite antichains in grid classes that contain a cycle.

Notice that conditions \ref{c1}--\ref{c3} specify the placement of the point $v_i$ in relation to earlier points in the sequence, which gives us an iterative method to construct coils. We begin by constructing the base case of length $\ell+1$: by~\ref{c1} we have  $v_1\rightarrow v_2\rightarrow  \cdots \rightarrow v_{\ell+1}$, and by~\ref{c4} $v_{\ell+1}\rightarrow v_1$.
Since $v_1$, $v_\ell$ and $v_{\ell+1}$ share a row or column, we must also have $v_\ell\rightarrow v_1$. Thus the points $v_1,\dots,v_\ell$ form a directed cycle (corresponding to a minimal indivisible), and the point $v_{\ell+1}$ must be placed in cell 1, in the horizontal or vertical strip defined by $v_\ell$ and $v_1$. All placements of $v_{\ell+1}$ in this region yield the same result.

Now suppose that $i>\ell+1$, and that $v_1,\dots,v_{i-1}$ have been placed. The following discussion is accompanied by Figure~\ref{fig-coil-placement}. By ~\ref{c1}, the point $v_i$ must be placed in cell $i\pmod{\ell}$. Since $i-1\equiv i-\ell-1\pmod{\ell}$, the two vertices $v_{i-1}$ and $v_{i-\ell-1}$ both lie in cell $i-1\pmod{\ell}$, and this cell shares a row or column with cell $i\pmod{\ell}$. Without loss of generality, suppose that these cells share a row.

\begin{figure}
{\centering
\begin{tikzpicture}[scale=0.5]
\foreach \x/\y in {0/0,0/4,0/6,0/10,8/0,8/4}
	\draw[dotted] (\x,\y) -- ++(-2,0) (\x+4,\y) -- ++(2,0);
\foreach \x/\y in {0/0,4/0,8/0,12/0,0/6,4/6}
	\draw[dotted] (\x,\y) -- ++(0,-1) (\x,\y+4) -- ++(0,1);
\draw [draw=none,fill=black!20] (0,1) rectangle ++ (1.5,1);
\draw[dashed] (-0.5,1) -- ++(13,0) (-0.5,2) -- ++(13,0) (1.5,-0.5) -- ++ (0,11);
\draw[->] (1,-1) -- ++(2,0);
\draw[->] (9,-1) -- ++(2,0);
\draw[->] (-1,1) -- ++(0,2);
\draw[->] (-1,7) -- ++(0,2);
\draw (0,0) rectangle (4,4);
\draw (8,0) rectangle ++(4,4);
\draw (0,6) rectangle ++(4,4);
\node[permpt,label={[label distance=-3pt]below right:\footnotesize$v_{i-1}$}] at (9,1) {};
\node[permpt,label={[label distance=-5pt]above right:\footnotesize$v_{i-\ell-1}$}] at (10,2) {};
\node[permpt,label={[label distance=-5pt]above right:\footnotesize$v_{i-2\ell-1}$}] at (11,3) {};
\node[permpt,label={[label distance=-3pt]right:\footnotesize$v_{i-\ell}$}] at (1.5,2.5) {};
\node[permpt,label={[label distance=-3pt]right:\footnotesize$v_{i-2\ell}$}] at (2.5,3.5) {};
\node[permpt,label={[label distance=-3pt]right:\footnotesize$v_{i-\ell+1}$}] at (2,8) {};
\node[permpt,label={[label distance=-3pt]right:\footnotesize$v_{i-2\ell+1}$}] at (3,9) {};
\end{tikzpicture} \par}
\caption{Building a coil: for $i> \ell+1$, the point $v_i$ must be placed in the shaded region.}\label{fig-coil-placement}	
\end{figure}

Conditions~\ref{c2} and~\ref{c3} together require that $v_{i-1}\rightarrow v_i\rightarrow v_{i-\ell-1}$, which means that $v_i$ must be placed in the horizontal strip between $v_{i-1}$ and $v_{i-\ell-1}$. As there are only two non-empty cells in this row, there are no other points in this strip: by~\ref{c2} and~\ref{c3}, the points in these two cells satisfy
\[ v_{i-1}\rightarrow v_{i-\ell-1}\rightarrow v_{i-\ell} \rightarrow v_{i-2\ell-1} \rightarrow v_{i-2\ell}\rightarrow \cdots \]
Thus, any vertical placement of $v_i$ in cell $i$ between $v_i$ and $v_{i-\ell-1}$ is permitted, and equivalent.

We now consider the horizontal placement of $v_i$. The column containing cell $i\pmod{\ell}$ has precisely two non-empty cells, namely $i\pmod{\ell}$ and $i+1\pmod{\ell}$. By~\ref{c2} and~\ref{c3}, the points in this column satisfy the following linear ordering,
\[v_{i-\ell}\rightarrow v_{i-\ell+1} \rightarrow v_{i-2\ell}\rightarrow v_{i-2\ell+1} \rightarrow\cdots\]
Since $v_{i-\ell-1}\rightarrow v_{i-\ell}$, and $v_i\rightarrow v_{i-\ell-1}$, it follows that $v_i$ must be placed earlier than all of the points in this column. Thus, any horizontal placement of $v_i$ before $v_{i-\ell}$ satisfies the criteria, and all such placements are equivalent.

The above discussion establishes that coils exist, and that the placement of each successive point $v_i$ is determined by $v_1,\dots,v_{i-1}$. More precisely, after fixing the cell in which $v_1$ is placed ($\ell$ choices) and then the cell that contains $v_2$ (two choices), the rest of the coil is uniquely determined. In particular, there are therefore precisely $2\ell$ distinct coils of any length $n>\ell$, although we do not require this fact in what follows. Finally,  observe that this collection of $2\ell$ coils naturally partitions into two sets depending on the direction in which the cycle is traversed.
In some simple grids this partition is determined by whether each coil proceeds clockwise or anticlockwise around the cycle, but in general we consider the order in which the cells of the cycle are visited by each coil. We call this order the \emph{chirality} of a coil. An example of a coil is illustrated in Figure~\ref{fig-coil-example}.


\begin{figure}
{\centering
\begin{tikzpicture}[scale=0.25]
\plotpermgrid{2,4,19,6,17,7,15,5,8,3,10,1,12,9,14,11,16,13,18}
\draw[thick] (7.5,0.5) -- (7.5,19.5);
\draw[thick] (13.5,0.5) -- (13.5,19.5);
\draw[thick] (0.5,7.5) -- (19.5,7.5);
\draw[thick] (0.5,13.5) -- (19.5,13.5);
\path [draw=gray!75,postaction={on each segment={mid arrow=gray}}]
          (7)--(15)--(14)--(9)--(8)--(5)
        --(6)--(17)--(16)--(11)--(10)--(3)
        --(4)--(19)--(18)--(13)--(12)--(1)--(2);
  \draw[->] (1.5,0) -- (6.5,0);
  \draw[<-] (8.5,0) -- (12.5,0);
  \draw[<-] (14.5,0) -- (18.5,0);
  \draw[->] (0,1.5) -- (0,6.5);
  \draw[<-] (0,8.5) -- (0,12.5);
  \draw[<-] (0,14.5) -- (0,18.5);
\end{tikzpicture} \par}
\caption{A gridded $M$-coil of length 19, where $M={}$\protect\gcthree{3}{-1,0,1}{0,1,1}{1,-1,0}. The grey edges have been drawn to illustrate the order of the vertices $v_1,\dots,v_{19}$, where $v_1$ corresponds to the uppermost entry in the bottom-left cell.}
\label{fig-coil-example}
\end{figure}

We now record a number of other observations about coils. First, in any row or column, the vertices are linearly ordered according to the ordering
\[v_{i-1}\rightarrow v_i\rightarrow v_{i-\ell-1}\rightarrow v_{i-\ell}\rightarrow \cdots\]
and when we restrict to cell $i\pmod{\ell}$ we have $v_i\rightarrow v_{i-\ell}\rightarrow v_{i-2\ell}\rightarrow\cdots$. This implies that the points $v_1,\dots,v_\ell$ are the last points in their cells, and they form a directed cycle.

The following properties readily follow:
\begin{enumerate}[label=\textbf{\Alph*}]
	\item\label{coil-a} Every set of $\ell$ consecutive vertices $v_i,\dots,v_{i+\ell-1}$ forms a directed cycle.
	\item\label{coil-b} $D_{\pi^\gridded}$ is strongly connected for every gridded coil $\pi^\gridded$.
	\item\label{coil-c} For every $v_i$, the only 
                        $j>i$ for which $v_i\rightarrow v_j$ is $j=i+1$.
\end{enumerate}

From~\ref{coil-a}, the points corresponding to a sequence of $\ell$ consecutive vertices of a coil must form one of the two minimal indivisible permutations.

From~\ref{coil-b} and by Lemma~\ref{lem-Dpi-strongly-connected}, gridded coils are $M$-indivisible. Furthermore,~\ref{coil-c} shows that the removal of any interior entry results in a divisible subpermutation:

\begin{lemma}\label{lem-coil-splits}
Let $\pi^\gridded$ be a gridded $M$-coil with points $v_1,\dots,v_n$ in $\Gridhash(M)$ where $M$ is a cycle partial multiplication matrix. Then the removal of a single entry $v_i$, where $1< i< n$, produces a gridded permutation that is $M$-divisible. Specifically,
\[\pi^\gridded - v_i \;=\; \sigma^\gridded\boxplus\tau^\gridded\]
where $\sigma^\gridded$ and $\tau^\gridded$ are the gridded permutations on the points $v_{i+1},\dots,v_n$ and $v_1,\dots,v_{i-1}$, respectively.
\end{lemma}

\begin{proof}
By~\ref{coil-c}, the only path from $v_1$ to $v_n$ in $D_{\pi^\gridded}$ is $v_1\rightarrow v_2\rightarrow \cdots \rightarrow v_n$. Thus, the removal of an interior point $v_i$ from $\pi^\gridded$ renders a digraph with no directed path from any vertex in $\{v_1,\dots,v_{i-1}\}$ to any vertex in $\{v_{i+1},\dots,v_n\}$. The decomposition stated in the lemma now follows by Lemma~\ref{lem-Dpi-strongly-connected}.
\end{proof}

Note that it is not necessarily the case that the gridded permutations
$\sigma^\gridded$ and $\tau^\gridded$
in Lemma~\ref{lem-coil-splits} are themselves $M$-indivisible: this will occur if and only if each part contains at least $\ell$ points. Furthermore, if either contains at least $\ell+1$ points, then it is itself a coil.

%
%
%
%
\subsection{Coils and the structure of indivisibles}\label{subsec-coils-and-indivs}

We now turn our attention to the role of coils in the structure of $M$-indivisible permutations. For now, $M$~is still a partial multiplication matrix whose graph $G_M$ is a cycle of length $\ell$. Let $\pi^\gridded\in\Gridhash(M)$ be a non-singleton $M$-indivisible permutation. By Lemma~\ref{lem-lastpoints}, the last points of $\pi$ in each cell of the cycle form a directed cycle in $D_{\pi^\gridded}$. Denote these points by $z_1,z_2,\dots,z_\ell$ in such a way that in $D_{\pi^\gridded}$ we have $z_1\rightarrow z_2\rightarrow\cdots\rightarrow z_\ell\rightarrow z_1$, and label the cells of $\pi^\gridded$ on the cycle as $1,\dots,\ell$ so that $z_i$ belongs to cell $i$.

Since $\pi^\gridded$ is indivisible, $D_{\pi^\gridded}$ is strongly connected, so there exists a path from $z_1$ to any other vertex obtained by following only forward edges. We partition the points of $\pi^\gridded$ based on their distance from $z_1$.


Formally speaking, set $B_1=\{z_1\}$, and for $i\ge 1$ let
\[
B_{i+1} \;=\; \bigg\{ v \in V(D_{\pi^\gridded})\setminus\bigcup_{j\le i} B_j \::\: u\rightarrow v\text{ for some }u\in B_{i} \bigg\}.
\]
Thus, $B_i$ denotes the set of vertices of $D_{\pi^\gridded}$ (or, equivalently, of points of $\pi^\gridded$) whose shortest path from $z_1$ has length $i-1$. Since $\pi^\gridded$ is indivisible, every vertex of $D_{\pi^\gridded}$ belongs to some set $B_i$, and we let $k$ be the index of the last non-empty set $B_k$.

\begin{prop}
In the above partition of $\pi^\gridded$, the set $B_i$ is contained in cell $i\pmod{\ell}$ of $\pi^\gridded$, for all $1\le i\le k$.
\end{prop}
\begin{proof}
We proceed by induction, noting that $z_1$ belongs to cell 1, by definition.

Suppose that the sets $B_1,\dots, B_i$ satisfy the hypothesis, and consider the set $B_{i+1}$. By construction, since $B_{i+1}$ is defined in terms of vertices that can be reached by following a single directed edge from $B_i$, it follows that $B_{i+1}$ must be contained in the union of cells $i-1$, $i$ and $i+1 \pmod{\ell}$.	

Now suppose that $B_{i+1}$ contains an entry $w$ in cell $i-1\pmod{\ell}$, and note that the shortest path from $z_1$ to $w$ has length $i$. By definition, there is some $v\in B_i$ such that $v\rightarrow w$, and also some $u\in B_{i-1}$ such that $u\rightarrow v$.
Thus we have $u\rightarrow v\rightarrow w$, and (by the inductive hypothesis) $u$ belongs to cell $i-1\pmod{\ell}$, along with $w$, which implies that $u\rightarrow w$. Now, any shortest path from $z_1$ to $u$ of length $i-2$ can be extended to a path of length $i-1$ to $w$, which means that the path of length $i$ from $z_1$ to $w$ was not the shortest one, which is a contradiction. A similar analysis applies in the case that $B_{i+1}$ contains a point in cell $i\pmod{\ell}$, and the result follows.
\end{proof}

\begin{figure}
{\centering
\begin{tikzpicture}[scale=0.55]
\plotgrid{13}
\foreach \i/\j/\k/\n in {1/1/E/19,2/2/E/13,3/3/E/7,
					  6/3/D/6,7/2/D/12,8/1/D/18, 6/6/B/5,7/7/B/11,8/8/B/17,
					  10/6/C/4,11/7/C/10,12/8/C/16, 10/10/B/3,11/11/B/9,12/12/B/15,
					  2/12/A/14,3/11/A/8,4/10/A/2}{%
	\ifthenelse{\equal{\k}{A}}{%
		\draw[thick] (\i+0.3,\j) rectangle ++ (0.7,.9);
		\node[permpt,label={[label distance=-3pt]above:\tiny$\ v_{\n}$}] at (\i+0.3,\j+.9) {};
		\node[empty] at (\i+0.65,\j+0.45) {\tiny $B_{\n}$};	
		}{}
	\ifthenelse{\equal{\k}{B}}{%
		\draw[thick] (\i,\j) rectangle ++ (.9,0.7);
		\node[permpt,label={[label distance=-4pt]above left:\tiny$v_{\n}$}] at (\i+.9,\j+0.7) {};
		\node[empty] at (\i+0.45,\j+0.35) {\tiny $B_{\n}$};	
		}{}
	\ifthenelse{\equal{\k}{C}}{%
		\draw[thick] (\i,\j) rectangle ++ (0.7,.9);
		\node[permpt,label={[label distance=-4pt]below right:\tiny$v_{\n}$}] at (\i+0.7,\j+.9) {};
		\node[empty] at (\i+0.35,\j+0.45) {\tiny $B_{\n}$};	
		}{}
	\ifthenelse{\equal{\k}{D}}{%
		\draw[thick] (\i,\j+.1) rectangle ++ (0.7,.9);
		\node[permpt,label={[label distance=-3pt]below:\tiny$v_{\n}\ \ $}] at (\i+0.7,\j+.1) {};
		\node[empty] at (\i+0.35,\j+0.55) {\tiny $B_{\n}$};	
		}{}
	\ifthenelse{\equal{\k}{E}}{%
		\draw[thick] (\i+.1,\j+0.3) rectangle ++ (.9,0.7);
		\node[permpt,label={[label distance=-4pt]below right:\tiny$v_{\n}$}] at (\i+.1,\j+0.3) {};
		\node[empty] at (\i+0.55,\j+0.65) {\tiny $B_{\n}$};	
		}{}
}
\node[permpt] at (4.1,4.3) [label={[label distance=-4pt]above:{\tiny $B_1=\{v_1\}$}}] {};
\draw[thick] (5.5,0.5) -- (5.5,13.5);
\draw[thick] (9.5,0.5) -- (9.5,13.5);
\draw[thick] (0.5,5.5) -- (13.5,5.5);
\draw[thick] (0.5,9.5) -- (13.5,9.5);
  \draw[->] (1.5,0) -- (4.5,0);
  \draw[<-] (6.5,0) -- (8.5,0);
  \draw[<-] (10.5,0) -- (12.5,0);
  \draw[->] (0,1.5) -- (0,4.5);
  \draw[<-] (0,6.5) -- (0,8.5);
  \draw[<-] (0,10.5) -- (0,12.5);
\end{tikzpicture} \par}
\caption{A coil decomposition, showing the partition $B_1,\dots,B_{19}$ and the corresponding (gridded) coil $v_1,\dots,v_{19}$, for an indivisible permutation in $\gridhash{\protect\gcthree{3}{-1,0,1}{0,1,1}{1,-1,0}}$.}\label{fig-grid-refinement}
\end{figure}

Since each set $B_i$ belongs to cell $i\pmod{\ell}$, the vertices in $B_i$ are linearly ordered, and we denote by $v_i$ the vertex of $B_i$ that comes first in this ordering, that is, $v_i\rightarrow v$ for every $v\in B_i$ distinct from $v_i$. Note that $v_1=z_1$. See Figure~\ref{fig-grid-refinement} for an example.

We record the following observation.

\begin{obs}\label{obs-coil-decomposition}
	If $k>\ell$, the sequence $v_1,\dots,v_{k}$ constructed above forms a gridded $M$-coil.
\end{obs}

Note that the requirement that $k>\ell$ is required simply because coils are defined to have length at least $\ell+1$. However, if $k\leq \ell$, then in fact it must be the case that $k=\ell$ or $k=1$, by Observation~\ref{obs-non-singleton}.

We refer to the partition $B_1,\dots,B_{k}$ as a \emph{coil decomposition} of $\pi^\gridded$ of \emph{length} $k$. Note that this decomposition is not unique (the choice of $B_1=\{z_1\}$ was arbitrary, and could have been replaced with any of the last vertices in the cells $1,\dots,\ell$), but every $M$-indivisible must possess a coil decomposition.

The coil decomposition of $\pi^\gridded$ can be used to define a refinement $N$ of $M$ whose dimensions are bounded by the size of the longest coil in $\pi^\gridded$, whose row-column graph is a path, and such that $\pi\in\Grid(N)$. Indeed, in Figure~\ref{fig-grid-refinement}, the feint grey lines indicate such a refined gridding. We formalise this in the following lemma.

\begin{lemma}\label{lem-griddable-acyclic}
Let $M$ be a cyclic partial multiplication matrix, and let $\pi^\gridded$ be an $M$-indivisible permutation that contains no coil of length greater than $k$. Then $\pi$ is $N$-griddable where $N$ is an acyclic matrix containing at most $k$ non-zero entries.
\end{lemma}

\begin{proof}
Fix any coil decomposition of $\pi^\gridded$, say $B_1,\dots,B_{k'}$ where $k'$ denotes the length of the decomposition. By Observation~\ref{obs-coil-decomposition}, we must have $k'\leq k$. We claim that this partition of $\pi^\gridded$ defines a refined gridding whose corresponding gridding matrix is acyclic.

By construction, each non-empty cell in $\pi^\gridded$ is partitioned into a finite number of sets $B_i,B_{i+\ell},\dots$, so that $B_i\leftarrow B_{i+\ell}$, $B_{i+\ell}\leftarrow B_{i+2\ell}$, and so on.
To see this, suppose that there is $u\in B_i$, $v\in B_{i+\ell}$ such that $u\rightarrow v$. Then there would be a shortest path from $z_1$ to $v$ of length $i$, which is a contradiction since the shortest path to $v$ should have length $i+\ell-1 > i+2$.
Similarly, we also have $B_{i+j\ell\pm 1}\rightarrow B_i$, and $B_i\rightarrow B_{i-j\ell\pm 1}$ for all $j\geq 1$.
This tells us that non-consecutive sets in the coil decomposition interact in a uniform way in the following sense: If $B_i$ and $B_{i+1}$ share a row or column, then every $B_{i-j\ell}$ and $B_{i+1-j\ell}$ precede both $B_i$ and $B_{i+1}$ in the consistent orientation, while every $B_{i+j\ell}$ and $B_{i+1+j\ell}$ succeed both.

Consider a pair of consecutive sets $B_{i}$ and $B_{i+1}$ which share a common row.
Define two (horizontal) slicing lines to separate by value the points in $B_i\cup B_{i+1}$ from the other points in their common row, i.e.\ the remaining points  in cells $i,i+1\pmod{l}$.
For consecutive sets that share a common column, we may similarly define two vertical slicing lines to separate the points by position.

This process yields a gridding that is a refinement of that for $\pi^\gridded$, and we let $N$ denote the corresponding matrix whose non-zero entries correspond to the non-empty cells of this refinement. (Note that these non-empty cells are precisely the sets $B_i$.)
Since the refined gridding is constructed such that each $B_i$ has at most one other entry in its row and column (namely $B_{i-1}$ and $B_{i+1}$), it follows that $G_N$ is a path of length $k'\le k$.
\end{proof}

Our final task for this section is to extend Lemma~\ref{lem-griddable-acyclic} from the case where the matrix $M$ is cyclic to the case where $M$ can be any pseudoforest partial multiplication matrix. To begin this task, we first need to extend the notion of a coil to these matrices.

Let $M$ be a pseudoforest partial multiplication matrix. A \emph{gridded $M$-coil} is an $M$-gridded permutation $\pi^\gridded$ for which there exists an ordering $v_1,\dots,v_n$ of the vertices in $D_{\pi^\gridded}$, and a labelling $1,\dots,\ell$ on the cells corresponding to the edges of some cycle of $G_M$, such that conditions \ref{c1}--\ref{c4} on page~\pageref{c1} are satisfied. Similarly, an \emph{$M$-coil} is an ungridded permutation $\pi$ such that there exists some $M$-gridding $\pi^\gridded$ which forms a gridded $M$-coil.

Note that while $M$ may have entries that do not lie on the cycle, by \ref{c1} the corresponding cell in an $M$-coil contains no points. Indeed, $M$ may contain more than one cycle in different components of $G_M$ (and each cycle can be used to construct $M$-coils), but any single $M$-coil must reside entirely within the cells corresponding to edges of just one cycle. Thus, gridded coils defined by pseudoforest partial multiplication matrices are exactly the same as gridded coils on cyclic partial multiplication matrices. As such, many properties of gridded coils are the same irrespective of which type of matrix is being used.

\begin{thm}\label{thm-coil-regridding}
Let $M$ be a pseudoforest partial multiplication matrix containing $T$ non-zero entries, and let $\pi^\gridded$ be an $M$-indivisible permutation that contains no gridded coil of length greater than~$k$. Then $\pi$ is $N$-griddable where $N$ is an acyclic matrix containing at most $k(T-3)$ non-zero entries.
\end{thm}

\begin{proof}
By Observation~\ref{obs-single-component}, the points of the $M$-indivisible permutation $\pi^\gridded$ belong to a single component $H$ of $G_M$. If $\pi^\gridded$ is a singleton then the statement in the theorem is clear, so we may assume that  $\pi^\gridded$ is not a singleton. By Observation~\ref{obs-non-singleton}, $\pi^\gridded$ contains at least one point from each cell around some cycle, which means that the component $H$ must contain a cycle. Let $\ell \ge 4$ denote the length of this cycle, and let $t\le T$ denote the number of edges in $H$.

We now proceed inductively on the number $t-\ell$ of non-zero entries of $M$ that belong to $H$ but do not lie on the cycle. Our hypothesis is that $\pi^\gridded$ has a refined gridding using an acyclic matrix that has at most $k(t-\ell+1)$ non-zero entries, and that this refined gridding divides each row and column using at most $k-1$ additional horizontal and vertical lines, respectively. Since $\ell\geq 4$, we see that $k(t-\ell+1) \leq k(T-3)$ and so this induction will complete the proof.

The base case, in which $t=\ell\ge 4$, follows by Lemma~\ref{lem-griddable-acyclic}, since the number of non-zero entries in this case is $k =k(t-\ell+1)$, and the refined gridding of $\pi^\gridded$ in this acyclic grid divides the points in each cell on the cycle among at most $\lceil k/\ell\rceil \le k$ smaller cells, which is achieved by adding at most $k-1$ horizontal lines in each row, and at most $k-1$ vertical lines in each column.

Now suppose that $M$ has $t-\ell > 0$ non-zero entries in $H$ that are not on the cycle. Identify a non-empty cell $c$ of $\pi^\gridded$ that has no other non-empty cells in either its column or its row. (Such a cell must correspond to an edge in $H$ since all points of $\pi^\gridded$ belong to $H$. Furthermore, $c$ exists since each edge of $H$ not on the cycle lies on a unique path from the cycle to some leaf of $G_M$, and so we may choose the last non-empty cell along this path.) Without loss of generality, we may assume that $c$ is the only non-empty cell in its column, and that the points in $c$ form an increasing sequence.

The gridded subpermutation of $\pi^\gridded$ formed by removing the points in $c$ is also $M$-indivisible, since for any directed path in $D_{\pi^\gridded}$ that passes through a point $x$ in $c$, there exists a directed path on all the same vertices but with $x$ removed.  By induction, this gridded subpermutation has a refined gridding into an acyclic matrix with at most $k(t-\ell)$ non-zero entries. Furthermore, in $\pi^\gridded$ the row containing $c$ has been divided by at most $k-1$ horizontal lines of this refined gridding. We now propagate these lines through $c$: for each of the $k-1$ additional horizontal lines, choose a vertical line so that the points in $c$ occupy the regions up the diagonal from the lower left to the upper right. Since there are no other non-empty cells in the column containing $c$, this process introduces no cycles and thus ensures that the resulting gridding is acyclic. This gridding has at most $k + k(t-\ell) =k(t-\ell+1)$ non-zero entries, and each cell has been divided into at most $k$ smaller cells. This completes the inductive step.
\end{proof}

%
%
%
%
%
%
%
%
\section{Labelled well quasi-ordering}\label{sec:lwqo}

In this section, we work towards the following characterisation for subclasses of pseudoforest grid classes. Recall that an $M$-coil is an ungridded permutation $\pi$ for which there is a gridding $\pi^\gridded$ that forms a gridded $M$-coil.

\begin{thm}\label{thm-lwqo-char}Let $M$ be a pseudoforest partial multiplication matrix. 
Then a subclass $\C\subseteq\Grid(M)$ is labelled well quasi-ordered if and only if 
there is a bound on the length of the longest $M$-coils that are contained in $\C$.
\end{thm}

This section is organised as follows. Subsection~\ref{subsec-wqo-toolkit} contains the
introductory definitions
and auxilliary results from the literature, and sets the context for the section. The proof of Theorem~\ref{thm-lwqo-char} is given in Subsection~\ref{subsec-lwqo-char-proof}. Finally, Subsection~\ref{subsec-lwqo-decide} uses Theorem~\ref{thm-lwqo-char} to derive a decision procedure for the following question: given a pseudoforest partial multiplication matrix $M$, is a given finitely based subclass $\C\subseteq\Grid(M)$ labelled well quasi-ordered?

%
%
%
%
\subsection{The well quasi-ordering toolkit}\label{subsec-wqo-toolkit}

In this introductory subsection, we give only the minimal terminology required for our purposes, and refer the reader to Brignall and Vatter~\cite{bv:lwqo-for-pp:} for a fuller treatment.

 Given a quasi-ordered set%
\footnote{Recall that a \emph{quasi-ordered} set $(X,\le)$ is a set $X$ equipped with a binary relation $\le$ that is both reflexive and transitive. In particular, every partial order is a quasi order.} %
 $(X,\leq)$, we say that $X$ is \emph{well quasi-ordered} (wqo) if in every infinite sequence $x_1,x_2,\dots$ of entries from $X$, there exists a pair $x_i,x_j$ with $i<j$ such that $x_i\leq x_j$.
A useful equivalent characterisation is given by the following folklore proposition. An \emph{antichain} in a quasi-order $(X,\leq)$ is a set $\{x_1,x_2,\dots\}$ such that $x_i\not\leq x_j$ for all $i\neq j$.

\begin{prop}
A quasi-order $(X,\leq)$ is well quasi-ordered if and only if $X$ contains neither an infinite antichain nor an infinite strictly descending chain, $x_1 >x_2 >\cdots$.	
\end{prop}
In the context of collections of combinatorial structures with some notion of embedding, the requirement that there are no infinite strictly descending chains is trivially satisfied, since such quasi-orders are always well-founded.

One widely-used tool in the study of wqo is Higman's Lemma~\cite{higman:ordering-by-div:}, which we now state. Given a quasi-order $(X,\leq)$, let $X^*$ denote the set of all finite sequences of elements of $X$ (which can be thought of as the set of all words over the alphabet $X$). Given two sequences $x_1\cdots x_k$ and $y_1\cdots y_n$ in $X^*$, we say that $x_1\cdots x_k$ is \emph{contained} in $y_1\cdots y_n$ if there exists a subsequence $1\leq i_1< \cdots < i_k\leq n$ such that $x_j\leq y_{i_j}$ for all $j=1,\dots,k$. We call this ordering the \emph{generalised subword ordering} on $X^*$.

\begin{thm}[Higman's Lemma~\cite{higman:ordering-by-div:}]\label{thm-higman}
If $(X,\leq)$ is wqo, then $X^*$ is wqo under the generalised subword ordering.
\end{thm}

Another tool we require is the following well-known fact (for a proof, see, for example,~\cite[Proposition~1.3]{bv:lwqo-for-pp:}.

\begin{prop}\label{prop-direct-prod-wqo}
	Let $(X,\leq_X)$ and $(Y,\leq_Y)$ be quasi orders. If $(X,\leq_X)$ and $(Y,\leq_Y)$ are wqo, then so is $(X\times Y,\leq_X\times\leq_Y)$.
\end{prop}

Let $(L,\leq_L)$ be any quasi-order. For a permutation $\pi$ of length $n$, an \emph{$L$-labelling} of $\pi$ is a function $\ell_\pi$ from the indices of $\pi$ to $L$, and we call the pair $(\pi,\ell_\pi)$ an \emph{$L$-labelled permutation}. Informally speaking, one may regard the label $\ell_\pi(i)$ as belonging to the point $(i,\pi(i))$ of $\pi$. We may similarly consider $L$-labellings of other combinatorial structures, most notably in our context gridded permutations: given an $M$-gridded permutation $\pi^\gridded$ and an $L$-labelling $\ell_\pi$ of the (ungridded) permutation $\pi$, we call the pair $(\pi^\gridded,\ell_\pi)$ an \emph{$L$-labelled $M$-gridded permutation}.

Given two $L$-labelled permutations $(\pi,\ell_\pi)$ and $(\sigma,\ell_\sigma)$ of lengths $n$ and $k$, respectively, we say that $(\sigma,\ell_\sigma)$ is contained in $(\pi,\ell_\pi)$ if there exists a subsequence $1\leq i_1\leq \cdots\leq i_k\leq n$ such that $\pi(i_1)\pi(i_2)\cdots\pi(i_k)$ is order isomorphic to $\sigma$, and such that $\ell_\sigma(j)\leq_L\ell_\pi(i_j)$ for all $j=1,\dots,k$. We refer to this ordering as the \emph{($L$-)labelled containment ordering}.

Given a set of permutations $\C$, we use $\C\wr L$ to denote the collection of all $L$-labelled permutations from $\C$. If $\C$ is a permutation class, then $\C\wr L$ is downwards-closed set under labelled containment: if $(\pi,\ell_\pi)\in\C\wr L$ and $(\sigma,\ell_\sigma)$ is contained in $(\pi,\ell_\pi)$, then $(\sigma,\ell_\sigma)$ is also in $\C\wr L$. We say that a set of permutations $\C$ is \emph{labelled well quasi-ordered} (lwqo) if $\C\wr L$ is well quasi-ordered under the labelled containment ordering for \emph{any} well quasi-ordered set $(L,\leq_L)$.

Similarly, for a set of $M$-gridded permutations $\C^\gridded$, we write $\C^\gridded\wr L$ for the collection of all $L$-labelled $M$-gridded permutations from $\C^\gridded$, and we can use the term \emph{labelled well quasi-ordered} in the obvious manner.

Let $(X,\leq_X)$ and $(Y,\leq_Y)$ be two quasi-orders, and let $\phi:X\to Y$ be a mapping. We say that $\phi$ is \emph{order-preserving} if $x_1\leq_X x_2$, implies $\phi(x_1)\leq_Y\phi(x_2)$. Conversely, we say that $\phi$ is \emph{order-reflecting} if $\phi(x_1)\leq_Y\phi(x_2)$ implies $x_1\leq_X x_2$.

\begin{lemma}\label{lem-order-preserving-reflecting}
Let $\phi:X\to Y$ be a mapping between two quasi-ordered sets.
\begin{enumerate}[(i)]
\item	If $\phi$ is an order-preserving surjection, then $Y$ is (labelled) well quasi-ordered whenever $X$ is.
\item If $\phi$ is order-reflecting, then $X$ is (labelled) well quasi-ordered whenever $Y$ is.
\end{enumerate}
\end{lemma}

For a proof, see~\cite[Propositions 1.10 and 1.13]{bv:lwqo-for-pp:}.

The notion of labelling also offers us an alternative viewpoint for gridded permutations. Let $M$ be an $m\times n$ gridding matrix, and let $G$ be the antichain on the elements $[m]\times[n]$. For a permutation $\pi\in\Grid(M)$ and a corresponding gridded version $\pi^\gridded$, define the function $g_\pi:\pi\to G$ by
$g_\pi(k) = (i,j)$ where in $\pi^\gridded$ the point $(k,\pi(k))$ lies in cell $ij$ of $\pi^\gridded$. It is then clear that $\phi: \Gridhash(M) \to \Grid(M)\wr G $ given by
\[ \phi(\pi^\gridded) = (\pi,g_\pi)\]
is an injective map that is both order-preserving and order-reflecting. (Note, however, that it is not surjective except in trivial cases.)

It is a well-known fact that the set of all permutations under the containment ordering is not wqo, and hence also not lwqo. On the other hand, it is still possible for permutation classes to be wqo or lwqo, which leads to the following general question.

\begin{question}\label{q-wqo-decidable}
Given a permutation class $\C$ (specified, for example, by its basis), can one decide whether $\C$ is well quasi-ordered, labelled well quasi-ordered, or neither?
\end{question}

This question is open, and seems hard at present.

Theorem~\ref{thm-lwqo-char} (proved in the next subsection) characterises lwqo for any class $\C$ contained in a pseudoforest grid class, and we then answer the question above in the affirmative for these classes in Subsection~\ref{subsec-lwqo-decide}.

%
%
%
%
\subsection{The proof of Theorem~\ref{thm-lwqo-char}}\label{subsec-lwqo-char-proof}

Our first task in proving Theorem~\ref{thm-lwqo-char} is to convert the question of labelled well quasi-ordering for a subclass $\C$ of some grid class to one for the $M$-indivisibles in $\C^\gridded$ corresponding to a single component of $G_M$. We do this in a succession of lemmas.

\begin{lemma}\label{lem-gridded-ungriddded-lwqo}
Let $M$ be a 
gridding matrix, and let $\C\subseteq \Grid(M)$. The set of $M$-gridded permutations $\C^\gridded\subseteq\Gridhash(M)$ is labelled well quasi-ordered if and only if $\C$ is.	
\end{lemma}

\begin{proof}
Let $\pi,\sigma\in\C$. If there exist griddings of $\pi$ and $\sigma$ such that $\sigma^\gridded\leq\pi^\gridded$, then any embedding as gridded permutations induces an embedding of $\sigma$ in $\pi$ as ungridded permutations, and hence $\sigma\leq\pi$. Thus the surjection $\phi:\C^\gridded\to\C$ given by $\phi(\pi^\gridded)=\pi$ that ``removes'' the gridding is order-preserving, and the converse implication is complete by Lemma~\ref{lem-order-preserving-reflecting}.

For the direct part, let $(L,\leq_L)$ be an arbitrary wqo set of labels. Suppose that $M$ is an $m\times n$ matrix, and let $G=[m]\times [n]$ denote the indices for the cells of $M$. We consider $G$ to be an antichain, and since it is finite it is wqo. By Proposition~\ref{prop-direct-prod-wqo}, the product $L\times G$ is also wqo.
Since $\C$ is lwqo, it follows that $\C\wr (L\times G)$ is wqo. 

We now adopt the viewpoint in which each gridded permutation is interpreted as a labelled permutation. Consider the mapping $\phi:\C^\gridded\wr L \to \C\wr (L\times G)$ given by
\[ \phi((\pi^\gridded,\ell_{\pi^\gridded})) = (\pi,(\ell_{\pi^\gridded},g_\pi))\]
where $g_\pi(k) = (i,j)$ records that the point $(k,\pi(k))$ lies in cell $ij$. This mapping is clearly order-reflecting, and since $\C\wr(L\times G)$ is wqo, so is $\C^\gridded\wr L$, which shows that $\C^\gridded$ is lwqo.
\end{proof}

We now restrict to the $M$-indivisible permutations of $\C^\gridded$.

\begin{lemma}\label{lem-indiv-wqo}
Let $M$ be a partial multiplication matrix, let $\C\subseteq \Grid(M)$, let $\III^\gridded$ denote the $M$-indivisible elements of $\C^{\gridded}$, and let $L$ be any non-empty ordered set. Then $\III^\gridded\wr L$ is well-quasi-ordered if and only if $\C^\gridded\wr L$ is well quasi-ordered.
%
\end{lemma}

\begin{proof} First, trivially, if $\C^\gridded\wr L$ is wqo, then so is any subset of $\C^\gridded\wr L$, such as $\III^\gridded\wr L$.

Conversely, suppose that $\III^\gridded\wr L$ is wqo. Higman's Lemma (Theorem~\ref{thm-higman})  tells us that $(\III^\gridded\wr L)^*$ is wqo as well.

Consider the sequence $(\pi_1^\gridded,\ell_{\pi_1^\gridded}),(\pi_2^\gridded,\ell_{\pi_2^\gridded}),\dots,(\pi_k^\gridded,\ell_{\pi_k^\gridded})$, and define $\pi^\gridded = \pi_1^\gridded\boxplus \cdots\boxplus\pi_k^\gridded$. Each point of $\pi^\gridded$ belongs to the natural copy of some $\pi_i^\gridded$ in $\pi^\gridded$. To such a point, we assign a label using $\ell_{\pi_i^\gridded}$. This gives a labelling $\ell_{\pi^\gridded}$ of $\pi^\gridded$, and we define $\psi:(\III^\gridded\wr L)^*\to \C^\gridded\wr L$ by
\[\psi((\pi_1^\gridded,\ell_{\pi_1^\gridded})\cdots(\pi_k^\gridded,\ell_{\pi_k^\gridded})) = (\pi,\ell_{\pi^\gridded}).\]
This map is clearly order-preserving. It is also surjective since every permutation $\pi^\gridded\in\C^\gridded$ possesses a decomposition into $M$-indivisibles by Lemma~\ref{lem-grid-decomp}. Thus, by Lemma~\ref{lem-order-preserving-reflecting}(i), $\C^\gridded\wr L$ is wqo, as required.
%
\end{proof}

Note that by Observation~\ref{obs-single-component}, any gridded $M$-indivisible is associated with a single component of $G_M$. Since $M$ has only finitely many components, we conclude the following.

\begin{lemma}\label{lem-one-component}
Let $M$ be a partial multiplication matrix and let $\C\subseteq \Grid(M)$. Then $\C^\gridded$ is labelled well quasi-ordered if and only if, for every component of $M$, the $M$-indivisibles in $\C^\gridded$ associated with this component are labelled well quasi-ordered.
\end{lemma}

By our earlier observation that any non-singleton $M$-indivisible must contain at least one point from each cell around some cycle of $M$, we can now state the following strengthening of one half of the main result of Murphy and Vatter~\cite{murphy:profile-classes:}.

\begin{prop}\label{prop-no-cycles}
Let $M$ be a matrix for which $G_M$ is a forest. Then the set of $M$-indivisibles in $\Gridhash(M)$ is equal to the set of all possible $M$-griddings of the singleton permutation.
Consequently, $\Grid(M)$ is labelled well quasi-ordered.
\end{prop}

We can now complete one half of the proof of Theorem~\ref{thm-lwqo-char}: Suppose that the longest coil contained in $\C$ has length $k$. By Lemma~\ref{lem-gridded-ungriddded-lwqo}, $\C$ is lwqo if and only if $\C^\gridded$ is. Furthermore, by Lemma~\ref{lem-one-component}, it suffices to show that the $M$-indivisible permutations in $\C^\gridded$ associated with each component of $G_M$ are lwqo. Fix any component of $G_M$, and consider the set $\III^\gridded$ of $M$-indivisible permutations in $\C^\gridded$ associated with this component. If the component is a tree, then $\III^\gridded$ is lwqo by Proposition~\ref{prop-no-cycles}. Thus we may now assume that the component is unicyclic.

By hypothesis, any $M$-indivisible $\pi^\gridded\in\III^\gridded$ contains a coil of length at most $k$, and thus by Theorem~\ref{thm-coil-regridding} the corresponding ungridded permutation $\pi$ is $N_\pi$-griddable for some acyclic matrix $N_\pi$ containing a bounded number of nonzero entries. This gives us a superset of the \emph{ungridded} permutations corresponding to $M$-indivisibles in $\III^\gridded$:
\[ \III \subseteq \bigcup_{\pi^\gridded\in\III^\gridded} \Grid(N_\pi).\]
The set $\{N_\pi:\pi^\gridded\in\III^\gridded\}$ must be finite, since it comprises a collection of acyclic gridding matrices each with a bounded number of non-zero entries.
Furthermore, each class $\Grid(N_\pi)$ is lwqo by Proposition~\ref{prop-no-cycles}, and hence we conclude that $\III$ is contained in a finite union of lwqo sets, and thus $\III$ is lwqo. A final application of Lemma~\ref{lem-gridded-ungriddded-lwqo} shows that $\III^\gridded$ is lwqo, as required.

To complete the proof of Theorem~\ref{thm-lwqo-char}, it remains to show that classes containing arbitrarily long coils contain infinite labelled antichains. In fact, we construct the elements of the antichain directly from the gridded coils.

\begin{lemma}\label{lem-coil-antichains}
Let $M$ be a unicyclic partial multiplication matrix, let $\pi^\gridded$ be a coil with points $v_1,\dots,v_n$, and let $\sigma^\gridded$ a coil with points $u_1,\dots,u_m$, where $m<n$.
If $\pi^\gridded$, $\sigma^\gridded$ are labelled via
\[
\ell_{\pi^\gridded}(v_i)=\begin{cases}	\bullet&1<i<n\\
\circ & i\in\{1,n\} 
,
\end{cases}
\qquad\text{and}\qquad
\ell_{\sigma^\gridded}(u_i)=\begin{cases}	\bullet&1<i<m\\
\circ & i\in\{1,m\}, 
\end{cases}
\]
where $L=\{\bullet,\circ\}$ is a $2$-element antichain,
then there is no labelled embedding of $(\sigma^\gridded,\ell_{\sigma^\gridded})$ in $(\pi^\gridded,\ell_{\pi^\gridded})$.
\end{lemma}

\begin{proof}
Suppose to the contrary that $\phi$ is an embedding of $(\sigma^\gridded,\ell_{\sigma^\gridded})$ into $(\pi^\gridded,\ell_{\pi^\gridded})$.
Then the image $\phi((\sigma^\gridded,\ell_{\sigma^\gridded}))$ in $(\pi^\gridded,\ell_{\pi^\gridded})$ must include $v_1$ and $v_n$, because they are the only entries whose labels match those of $u_1$ and $u_m$ in $(\sigma^\gridded,\ell_{\sigma^\gridded})$. Since $m<n$, this means that there is at least one entry $v_k$ (with $1<k<n$) not in the image $\phi((\sigma^\gridded,\ell_{\sigma^\gridded}))$.

By Lemma~\ref{lem-coil-splits}, the removal of $v_k$ from $\pi^\gridded$ leaves an $M$-divisible gridded subpermutation in which $v_1$ and $v_n$ are in different components. The image $\phi((\sigma^\gridded,\ell_{\sigma^\gridded}))$ is contained in $\pi^\gridded-v_k$, and it contains $v_1,v_n$; hence the image is also $M$-divisible by Observation~\ref{obs-divisible-subpermutation}. This is a contradiction, since $\sigma^\gridded$ is indivisible.
\end{proof}


Note that the following result applies to \emph{any} permutation class, but in particular it completes our proof of Theorem~\ref{thm-lwqo-char}.

\begin{prop}\label{prop-long-coils}
Let $\C$ be a permutation class. If there exists a cyclic partial multiplication matrix $M$ such that $\C$ contains arbitrarily long $M$-coils, then $\C$ is not labelled well quasi-ordered.
\end{prop}

\begin{proof}
Suppose that $M$ is an $m\times n$ matrix, and let $G=[m]\times [n]$ denote the indices for the cells of $M$ (which, as before, we consider to be an antichain). Denote the (infinite) set of $M$-coils in $\C$ by~$\OOO$.

For each $\pi\in\OOO$, we fix an $M$-gridded version, $\pi^\gridded$, of $\pi$, chosen so that $\pi^\gridded$ is a gridded $M$-coil. Now define a labelling $\ell_\pi:\pi\to G\times \{\bullet,\circ\}$ as follows. For a point $p$ belonging to cell $(i,j)$ in $\pi^\gridded$, we have
\[
\ell_\pi(p) = \begin{cases}\left((i,j),\circ\right)&\text{if }p\text{ is the first or last point of the coil}\\
 	\left((i,j),\bullet\right)&\text{otherwise.}
 \end{cases}
\]
For distinct $\sigma,\pi\in \OOO$, the labelled permutations $(\sigma,\ell_\sigma)$ and $(\pi,\ell_\pi)$ are not comparable by Lemma~\ref{lem-coil-antichains}, and thus $\C$ is not lwqo.
\end{proof}

%
%
%
%
%
\subsection{Decidability}\label{subsec-lwqo-decide}

Before we move on to consider the bases of pseudoforest grid classes, we take a brief diversion to show that lwqo in finitely based subclasses of pseudoforest grid classes is decidable, thus answering an instance of Question~\ref{q-wqo-decidable}.

\begin{thm}\label{thm-lwqo-decidable}
	There exists an algorithm that answers the following:\\
\strut\hspace{.5in} Given a pseudoforest 
gridding matrix $M$ and permutations $\pi_1,\dots,\pi_k\in\Grid(M)$,\\
\strut\hspace{.5in} is $\C=\av(\pi_1,\dots,\pi_k)\cap \Grid(M)$ labelled well quasi-ordered?
\end{thm}

\begin{proof}
Without loss of generality and by the discussions in Section~\ref{sec-grid-class-structure}, we may assume that $M$ is a partial multiplication matrix.
 By Theorem~\ref{thm-lwqo-char}, we need to check whether $\C$ contains arbitrarily long coils.
 Since every coil stays within a single unicyclic component, we may assume without loss that $G_M$ is connected, and has a cycle; let $\ell$ denote the length of this cycle. Let $n=\max\{|\pi_1|,\dots,|\pi_k|\}$. We claim the following:

\textbf{Claim.} \emph{If some $\pi_i$ embeds in a coil of length $m > (n+5)\ell+n$, then it embeds in a coil of length $m-\ell$. }

Put another way, if $\pi_i$ embeds in a coil, then it embeds in a coil of length at most $(n+5)\ell+n$.
   The underlying idea behind proving this claim is that if a `long' coil is `disconnected' into the sum of two coils by the removal of some `middle' subcoil which starts and ends in the same cell $c$, it can be `reconnected' into a coil again by the addition of one new point in $c$.

Given $\pi_i$ embeds in a coil of length $m > (n+5)\ell+n$, there exists an $M$-gridding $\pi_i^\gridded$ that embeds in a gridded coil $\xi^\gridded$.
Let the points of $\xi^\gridded$ be $v_1,v_2,\dots,v_m$, and recall Definition~\ref{de-coil} and the notation used therein.
The embedding of $\pi_i^\gridded$ into $\xi^\gridded$ corresponds to a subsequence $\sigma$ of $v_1,\dots,v_m$ of length $n'\leq n$.

We now consider points of the coil $\xi^\gridded$, other than the first $2\ell$ or the final $2\ell$ points, that are \emph{not} part of the subsequence $\sigma$.
The points of $\sigma$ split the other points of $v_{2\ell+1},\dots, v_{m-2\ell}$ into at most $n'+1$ contiguous subsequences, some of which may be empty. The total number of points in these subsequences is at least
\[
m-4\ell-n'\geq m-4\ell-n>(n+1)\ell\geq (n'+1)\ell.
\]
Hence, by the pigeonhole principle, we can find a contiguous subsequence of $v_{2\ell+1},\dots, v_{m-2\ell}$ of length at least $\ell+1$ which contains no points of $\sigma$. Thus, fix an index $j$ with $2\ell<j\leq m-3\ell$ such that $v_j,\dots,v_{j+\ell}$ contains no points of $\sigma$.

The subpermutation $\zeta^\gridded$ of $\xi^\gridded$ corresponding to the sequence
$v_1,\dots ,v_{j-1},v_{j+\ell+1},\dots,v_m$ also contains the same embedding of $\pi_i^\gridded$.
Furthermore, by Lemma~\ref{lem-coil-splits}, $\zeta^\gridded$ can be written as an $M$-sum
\[\zeta^\gridded = \beta^\gridded \boxplus \alpha^\gridded\]
where $\alpha^\gridded$ is the permutation defined on the points $v_1,\dots,v_{j-1}$, and $\beta^\gridded$ is the permutation defined on the points $v_{j+\ell+1},\dots,v_m$.

Note that the sequence $v_j,\dots,v_{j+\ell}$ begins and ends in cell $j\pmod{\ell}$.
We now insert a single point $z$ into  $\zeta^\gridded$, which will reconnect $\alpha$ and $\beta$ into a coil, as illustrated in Figure~\ref{fig-coil-reinsertion}.
Specifically, the point $z$ goes into cell $j\pmod{\ell}$, and satisfies
\[
v_{j+2\ell+1}\rightarrow z\rightarrow v_{j+\ell+1}\quad\text{and}\quad v_{j-1}\rightarrow z\rightarrow v_{j-\ell-1}.
 \]
 We note that the four points  $v_{j+2\ell+1},v_{j+\ell+1},v_{j-1}, v_{j-\ell-1}$ all exist and are present in $\zeta^\gridded$
 because $2\ell<j\leq m-3\ell$.
Therefore,  the sequence $v_1,\dots,v_{j-1},z,v_{j+\ell+1},\dots,v_m$ satisfies~\ref{c1}--\ref{c4}, and thus forms a gridded coil of length $m-\ell$. Furthermore, this coil contains $\pi_i^\gridded$, as witnessed by the subsequence $\sigma$, which has not been affected by the construction. This establishes the claim.

\begin{figure}
{\centering
\begin{tikzpicture}[scale=0.4]
\foreach \x/\y in {0/0,0/5,0/7,0/12,9/0,9/5}
	\draw[dotted] (\x,\y) -- ++(-2,0) (\x+5,\y) -- ++(2,0);
\foreach \x/\y in {0/0,5/0,9/0,14/0,0/7,5/7}
	\draw[dotted] (\x,\y) -- ++(0,-1) (\x,\y+5) -- ++(0,1);
\draw[->] (1,-1) -- ++(3,0);
\draw[->] (10,-1) -- ++(3,0);
\draw[->] (-1,1) -- ++(0,3);
\draw[->] (-1,9) -- ++(0,3);
\draw (0,0) rectangle (5,5);
\draw (9,0) rectangle ++(5,5);
\draw (0,7) rectangle ++(5,5);
\node[permpt,thick,fill=white,label={[label distance=-3pt]above left:\footnotesize$z$}] (z) at (2,3.5) {};
\draw[black!45] (z) -- ++(11.3,0) (z) -- ++(0,5.3);
\begin{scope}[label distance=-3pt,
		black label/.style={label={below right:\footnotesize#1}}]
\node[permpt,black label={$v_{j+2\ell-1}$}] at (10.5,1) {};
\node[permpt,black label={$v_{j-1}$}] at (12.5,3) {};
\node[permpt,black label={$v_{j-\ell-1}$}] at (13.5,4) {};
\node[permpt,black label={$v_{j+2\ell}$}] at (1,1.5) {};
\node[permpt,black label={$v_{j-\ell}$}] at (4,4.5) {};
\node[permpt,black label={$v_{j+2\ell+1}$}] at (1.5,8) {};
\node[permpt,black label={$v_{j+\ell+1}$}] at (2.5,9) {};
\node[permpt,black label={$v_{j-\ell+1}$}] at (4.5,11) {};
\end{scope}
\begin{scope}[label distance=-3pt,
		grey label/.style={label={[text=black!45]below right:\footnotesize#1}},
		permpt/.append style={fill=black!45,draw=black!45}]
\node[permpt,grey label={$v_{j+1}$}] at (3.5,10) {};
\node[permpt,grey label={$v_{j+\ell-1}$}] at (11.5,2) {};
\node[permpt,grey label={$v_{j+\ell}$}] at (2,2.5) {};
\node[permpt,grey label={$v_{j}$}] at (3,3.5) {};
\end{scope}

\end{tikzpicture} \par}
\caption{Inserting the entry $z$ into $\zeta^\gridded$ to construct a coil of length $m-\ell$. Some of the $\ell+1$ entries removed from $\xi^\gridded$ to form $\zeta^\gridded$ are shown in grey
(proof of Theorem \ref{thm-lwqo-decidable}).}\label{fig-coil-reinsertion}	
\end{figure}

We now complete the proof of the theorem. Recall that there are at most $2\ell$ coils of any length $n>\ell$, and these are partitioned into two sets by their chirality.
Furthermore, note that if some coil $\xi$ contains some element $\pi_i$, then every coil of the same chirality as $\xi$ and of length at least $|\xi|+\ell$ also contains $\pi_i$ (since any such coil also contains $\xi$).

Thus, $\C$ contains arbitrarily long coils if and only if it contains all the coils of length $(n+5)\ell+n$ of one chirality.

This leads to the following decision procedure in the case where $G_M$ is connected:
for each coil $\gamma$ of length $(n+5)\ell+n$ in $\Grid(M)$, check whether $\gamma$ contains $\pi_i$ for each $1\le i\le k$. If, for one chirality, none of the coils of that chirality contains any $\pi_i$, then return that $\C$ admits arbitrarily long coils and is not lwqo. Otherwise at least one coil of each chirality contains a basis element, and we return that $\C$ admits only bounded length coils and is lwqo.

The decision procedure for an arbitrary pseudoforest partial multiplication matrix $M$ consists of using the preceding procedure on each submatrix of $M$ corresponding to a connected component of $G_M$. If for any single component we find that $\C$ admits arbitrarily long coils, then $\C$ is not lwqo; if $\C$ admits only bounded length coils in every component, then $\C$ is lwqo.
\end{proof}

%
%
%
%
%
%
%
%
\section{Bases of pseudoforest grid classes}\label{sec-basis}

We now turn our attention to the bases of grid classes.
If $M$ is acyclic, then, by the comments in the introduction, $\Grid(M)$ is a geometric grid class, and so the results of~\cite{albert:geometric-grid-:} apply. In particular Theorem~6.2 of~\cite{albert:geometric-grid-:} states that $\Grid(M)$ is finitely based, although the proof is not constructive and no bounds on the lengths of the basis elements are known.

We exhibit two pseudoforest grid classes that are not finitely based (thus disproving Conjecture 2.3 of~\cite{Huczynska2006}), but we also show that unicyclic grid classes \emph{are} finitely based. The consequence of these two results is that we now have a reasonably precise understanding regarding where the (now false) conjecture `first' breaks down.

Our counterexamples can be found in Subsection~\ref{subsec-not-finitely-based}. Our positive result, whose proof is given in Subsection~\ref{subsec-proof-finite-basis}, is as follows.

\begin{thm}\label{thm-finite-basis}
Let $M$ be a unicyclic gridding matrix. Then $\Grid(M)$ is finitely based.
\end{thm}

As usual, we can assume that $M$ is a partial multiplication matrix, but note here that we must appeal to the unicyclic case of Proposition~\ref{prop-pseudoforest-on-a-pmm}, rather than the more general assumptions for pseudoforest gridding matrices that we have been using thus far.

%
%
\subsection{Unique griddings}

Before we can prove Theorem~\ref{thm-finite-basis}, we need some results concerning possible griddings of coils. The first result in this direction shows that sufficiently long coils grid uniquely in their cycles.

\newcommand{\gridbound}{(\ell+1)\ell^2+1}
\newcommand{\gridboundp}{(\ell+1)\ell^2+2}
\begin{lemma}[Murphy and Vatter~{\cite[Lemma 4.2]{murphy:profile-classes:}}]\label{lem-coil-unique-gridding}
Let $M$ be a cyclic partial multiplication matrix containing $\ell$ non-empty cells. Then every coil of length at least $\gridbound$ has only one $M$-gridding.
\end{lemma}

We note here that the bound above is almost certainly not optimal. Indeed, consideration of the short coils on relatively small cycles suggests the following should hold.

\begin{conj}\label{conj-unique-griddings}
	Let $M$ be a cyclic partial multiplication matrix containing $\ell$ non-empty cells. Then every coil of length at least $2\ell+1$ has only one $M$-gridding.
\end{conj}

This result would be the best possible in general, since there exists a coil of length 8 in a $2\times 2$ matrix that has more than one gridding; see Figure~\ref{fig-coil-non-unique-gridding} parts (i) and (ii).

\begin{figure}
	{\centering
	\begin{tikzpicture}[scale=0.3]
	\plotpermgrid{2,8,4,7,6,1,5,3}
	\draw[thick] (0.5,4.5) -- ++(8,0);
	\draw[thick] (4.5,0.5) -- ++(0,8);
	\draw[->] (1.5,0) -- ++(2,0);
	\draw[->] (5.5,0) -- ++(2,0);
  	\draw[->] (0,1.5) -- ++(0,2);
  	\draw[<-] (0,5.5) -- ++(0,2);
  	\draw (5) -- (3) -- (4) -- (7) -- (6) -- (1) -- (2) -- (8);
  	\node[draw=none] at (4.5,-1.2) {\small (i)};
	\end{tikzpicture}
	\qquad
	\begin{tikzpicture}[scale=0.3]
	\plotpermgrid{2,8,4,7,6,1,5,3}
	\draw[thick] (0.5,2.5) -- ++(8,0);
	\draw[thick] (3.5,0.5) -- ++(0,8);	
  	\node[draw=none] at (4.5,-1.2) {\small (ii)};
	\end{tikzpicture}
	\qquad
	\begin{tikzpicture}[scale=0.3]
	\plotpermgrid{2,9,4,7,8,6,1,5,3}
	\draw[thick] (0.5,4.5) -- ++(9,0);
	\draw[thick] (4.5,0.5) -- ++(0,9);
	\draw[->] (1.5,0) -- ++(2,0);
	\draw[->] (5.5,0) -- ++(2,0);
  	\draw[->] (0,1.5) -- ++(0,2);
  	\draw[<-] (0,5.5) -- ++(0,2);
  	\draw (5) -- (3) -- (4) -- (7) -- (6) -- (1) -- (2) -- (9)--(8);
  	\node[draw=none] at (5,-1.2) {\small (iii)};
	\end{tikzpicture}
	\qquad
	\begin{tikzpicture}[scale=0.3]
	\plotpermgrid{2,9,4,7,8,6,1,5,3}
	\draw[thick] (0.5,5.5) -- ++(9,0);
	\draw[thick] (4.5,0.5) -- ++(0,9);
	\draw[thick] (7.5,0.5) -- ++(0,9);
  	\node[draw=none] at (5,-1.2) {\small (iv)};
	\end{tikzpicture}	\par}
	\caption[]{(i)~A gridded coil in $\gridhash{\gctwo{2}{-1,-1}{1,1}}$ of length 8, order isomorphic to 28476153. (ii)~A different gridding of 28476153 in the same grid. (iii)~A gridded coil of length 9 in $\gridhash{\gctwo{2}{-1,-1}{1,1}}$ order isomorphic to 294786153. (iv)~A gridding of 294786153 in $\gridhash{\gctwo{3}{-1,-1,0}{1,1,-1}}$ with two misplaced points.}\label{fig-coil-non-unique-gridding}
\end{figure}

Given a gridding of a coil in a unicyclic grid class, we say that a point is \emph{misplaced} if it is placed in a cell that does not correspond to an edge of the cycle.
Lemma~\ref{lem-coil-unique-gridding} does not extend directly to unicyclic classes, since griddings with misplaced points are possible;
see Figure~\ref{fig-coil-non-unique-gridding} parts (iii) and (iv). Rather than attempting to exclude misplaced points entirely, we will instead bound how many such points there can be in coils gridded in unicyclic classes.

Given a permutation (or gridded permutation), we say that two points $u$ and $v$ are \emph{horizontally separated} if there exists a third point $w$ such that $w$ lies between $u$ and $v$ by value, but not by position. Similarly, if there exists a point $x$ that lies between $u$ and $v$ by position, but not by value, then we say that $u$ and $v$ are \emph{vertically separated}.
Note that neither of the sets of points $\{u,v,w\}$ or $\{u,v,x\}$ can form a monotone pattern, and thus if $u$ and $v$ belong to the same cell of a gridded permutation, then neither $w$ nor $x$ can belong to that cell.

The following observation follows readily from the axioms defining coils, \ref{c1}--\ref{c4}.

\begin{obs}\label{obs-3-in-a-coil}
If $u,v$ are nonconsecutive points in some cell of a gridded $M$-coil, then $u$ and $v$ are separated both horizontally and vertically.	
\end{obs}

Next, we establish a technical proposition that provides a bound, for any gridding of a coil, on the number of points in any cell, in terms of the number of points in other cells in the same row or column.

\begin{prop}\label{prop-row-bound}
Let $M$ be a partial multiplication matrix with at least one cycle, and let $\pi^\natural$ be a gridded $M$-coil defined on a cycle in $M$ of length $\ell$.

Let $\pi^\gridded$ be any other $M$-gridding of $\pi$, and denote by $c_1,\dots,c_r$ the non-empty cells of $\pi^\gridded$ that belong to some single row or column of $M$. If cell $c_i$ contains $k_i$ points for $1\le i\le r$, and $S_i=\sum_{j\neq i}k_j$, then
$
	k_i \le 2\ell (S_i+ 1).
$
\end{prop}

\begin{proof}
Without loss of generality, suppose that $c_1,\dots,c_r$ are the non-empty cells belonging to some row of $\pi^\gridded$. For a contradiction, suppose that the cell $c_i$ contains at least $2\ell (S_i+ 1)+1$ points. By the pigeonhole principle, this implies that $c_i$ must contain at least
$2(S_i+ 1)+1$
points that belong to a single cell in the gridded $M$-coil $\pi^\natural$.

Within this collection of points we can identify at least $S_i+ 1$ nonoverlapping nonconsecutive pairs of points, as illustrated in Figure~\ref{fig-many-pairs}. By Observation~\ref{obs-3-in-a-coil} each of these pairs must be separated in $\pi^\natural$ by some other entry in the same row, giving a collection of at least $S_i+ 1$ distinct separating points in $\pi^\natural$. In the gridding $\pi^\gridded$, all these separating points must belong to the cells $c_1,\dots,c_r$, except none can lie in $c_i$. However, there are only $S_i$ points available in the other cells of this row.
\begin{figure}
\centering
\begin{tikzpicture}[scale=0.25]
\draw (0,0) rectangle (10,10);
\foreach \i in {1,3,5,7,9}
	\node[permpt] at (\i,\i) {};
\foreach \i in {2,4,6,8} {
	\node[permpt,fill=black!30,draw=black!30] at (\i,\i) {};
	\node[permpt] (p\i) at (\i+10,\i+.5) {};
	\draw[dashed] (p\i) -- ++(-10.5,0);
}
\end{tikzpicture}
\caption{Nine points of a coil in a cell can be divided into four nonoverlapping nonconsecutive pairs of points, each of which must be separated horizontally
(proof of Proposition \ref{prop-row-bound}).}\label{fig-many-pairs}
\end{figure}
\end{proof}

Now let $M$ be a unicyclic partial multiplication matrix, and define $\pi^\natural$ and $\pi^\gridded$ as in Proposition~\ref{prop-row-bound}. Consider a cell $c$ of $\pi^\gridded$ that corresponds to an edge incident with a leaf of $G_M$. Such a cell is necessarily the only non-empty cell in its row or column, and by applying Proposition~\ref{prop-row-bound} to this row or column, we conclude that cell $c$ contains at most $2\ell$ (misplaced) points.

Similarly, if $c_1,\dots, c_r$ are the non-empty cells of a row of $\pi^\gridded$ with the property that $c_1,\dots,c_{r-1}$ correspond to edges incident with leaves of $G_M$, then the remaining non-empty cell $c_r$ in that row can contain at most
\[
	2\ell \big(2\ell(r-1)+ 1\big) = 4\ell^2(r-1)+ 2\ell
\]
misplaced points.

We now iterate this process: as soon as all but one of the non-empty cells in a row or column have a bound on the number of misplaced points they contain, we can apply Proposition~\ref{prop-row-bound} to establish a bound for the misplaced points in the final non-empty cell (which by definition is not part of the cycle). This process will eventually bound the number of points in each non-empty cell of $\pi^\gridded$ that is not associated with an edge from the single cycle of $M$. While these bounds are surely a significant overestimate of the actual number of points that the cells can contain, this quantity depends only on properties of the matrix $M$, rather than of the coil $\pi$. Considering the sum of the bounds on all such cells, we conclude the following.

\begin{lemma}\label{lem-coil-gridding-off-cycle}
Let $M$ be a unicyclic partial multiplication matrix.
There is a bound $B$, depending only on $M$, such that the following holds:
if $\pi^\gridded$ is any $M$-gridding of an $M$-coil $\pi$, then $\pi^\gridded$ has at most $B$ misplaced points.
\end{lemma}

As well as there being a bound on the number of misplaced points, there are also restrictions on which points can be misplaced, as a consequence of the following simple observation concerning points sandwiched between two others in a cell.

\begin{obs}\label{obs-sandwich}
Given any gridding matrix $M$, let $\pi^\natural$ and $\pi^\gridded$ be two $M$-griddings of an arbitrary  $\pi\in \Grid(M)$.
Suppose $u,v,w$ is a monotone triple of points from a single cell $c$ of $\pi^\natural$.
If $u$ and $w$ also share a common cell $c'$ in $\pi^\gridded$, then $v$ lies in $c'$ too.
\end{obs}

In the case that $M$ is unicyclic and $\pi$ is a coil, we have the following specialisation, given that $\pi$ has a standard gridding on the cycle.

\begin{cor}\label{cor-sandwich}
  Let $M$ be a unicyclic partial multiplication matrix whose cycle has length $\ell$, and
  let $\pi^\gridded$ be any gridded $M$-coil with points $v_1,\ldots,v_n$.
  For any suitable $i,j,k$, if $v_{i-j\ell}$ and $v_{i+k\ell}$ share a common cell in $\pi^\gridded$, then $v_i$ also lies in that cell in $\pi^\gridded$.
\end{cor}

%
%

%
%
%
\subsection{Proof of Theorem~\ref{thm-finite-basis}}\label{subsec-proof-finite-basis}

The proof of Theorem~\ref{thm-finite-basis} will be completed by considering the longest possible coil that can be contained in a basis element $\beta$ of some $\Grid(M)$, where $M$ is a unicyclic partial multiplication matrix.
We will establish a bound on the length of such coils that depends only on $M$ and not on $\beta$. We can then conclude that each basis element can be formed by adding a single point to a permutation that belongs to an lwqo subclass of the given grid class. In turn, this is sufficient to establish that the basis elements themselves belong to an lwqo class, and consequently the basis (which forms an unlabelled antichain) must be finite. The bulk of the remaining work is in the following lemma.

\begin{lemma}\label{lem-basis-bound}
Let $M$ be a unicyclic partial multiplication matrix, and let $\C=\Grid(M)$.
There is a bound, which depends only on $M$, on the length of the longest coil contained in any basis element $\beta$ of $\C$.
\end{lemma}

\begin{proof}
Let the length of the cycle in $M$ be $\ell$, and let $L=\gridbound$.
Recall that, by Lemma~\ref{lem-coil-unique-gridding}, any coil of length at least $L$ has only one gridding on the cycle of $M$.
Also, let $B$ denote the bound from Lemma~\ref{lem-coil-gridding-off-cycle} on the total number of misplaced points in any $M$-gridding of a coil.
We claim that any coil contained in $\beta$ has length less than $K=(4B+2)L+2\ell+1$.

To derive a contradiction, suppose that $\beta$ contains an $M$-coil $\xi$ with at least $K$ points.
Note that $\beta$ has at least one point not on $\xi$ because $\xi$ is $M$-griddable; indeed $\xi$ can be gridded on the cycle.
Note also that, by Lemma~\ref{lem-coil-unique-gridding}, $\xi$ only has one gridding, $\xi^\gridded$ say, on the cycle.

Fix an embedding $v_1,\dots,v_K$ of $\xi$ in $\beta$.
Let $m=(2B+1)L+2\ell+1$ and let $p$ be the point $v_m$.
Let $\sigma$ be the subpermutation of $\beta$ formed of the points of $\xi$ before $p$, and $\tau$ the subpermutation formed of the points of $\xi$ after $p$.
Note that both $\sigma$ and $\tau$ are $M$-coils of length at least $(2B+1)L$, this bound being tight for $\tau$.

Consider $\beta_p=\beta-p$, and note that since $\beta$ is a basis element of $\Grid(M)$, we have $\beta_p\in\Grid(M)$.
Fix any $M$-gridding $\beta_p^\gridded$ of $\beta_p$, and let $\sigma^\gridded$ and $\tau^\gridded$ denote the inherited $M$-griddings of the subpermutations $\sigma$ and $\tau$ of $\beta_p$.

Let us now consider in detail the structure of $\sigma^\gridded$; exactly the same analysis will apply to~$\tau^\gridded$.
By Lemma~\ref{lem-coil-gridding-off-cycle}, at most $B$ points of $\sigma$ are misplaced in $\sigma^\gridded$,
the remaining points of $\sigma$ being gridded on the cycle of~$\C$.
These points are divided (by the removal of the misplaced points) into at most $B+1$ contiguous portions.

If one of these portions of $\sigma$ has length $L$ or greater, then by Lemma~\ref{lem-coil-unique-gridding}, there is only one way in which it can be gridded on the cycle.
Suppose there were two such portions, $\varphi=v_a,\ldots,v_b$ and $\psi=v_c,\ldots,v_d$ (where $b<c-1$), each of length at least $L$.
Then each would have a unique gridding, $\varphi^\gridded$ and $\psi^\gridded$ say, on the cycle.
Moreover, since their $M$-sum $\psi^\gridded\boxplus\varphi^\gridded$ is contained in $\xi^\gridded$ (the unique gridding of $\xi$ on the cycle), we know that this is how $\varphi$ and $\psi$ would be gridded.
Thus, by Corollary~\ref{cor-sandwich}, each point $v_i$, $b<i<c$, would also be gridded on the cycle, and so would not be misplaced, since for each such $i$ there would be a point of $\varphi^\gridded$ and a point of $\psi^\gridded$ sandwiching $v_i$ in a common cell of the cycle.
Hence only one portion of $\sigma$ can have length $L$ or greater.

As a consequence, the only points of $\sigma$ that can be misplaced in any $M$-gridding are its first $BL$ points and its last $BL$ points.
In particular, there is a contiguous portion of $\sigma$ of length at least $L$ all of whose points are gridded on the cycle in any $M$-gridding.
Similarly, there is a contiguous portion of $\tau$ of length at least $L$ that contains no misplaced points in any $M$-gridding.
Call these portions $\alpha$ and $\omega$, respectively.
Because of their length and the fact that they are gridded on the cycle, by Lemma~\ref{lem-coil-unique-gridding}, their $M$-griddings, $\alpha^\gridded$ and $\omega^\gridded$, are unique.
Moreover, since their $M$-sum $\omega^\gridded\boxplus\alpha^\gridded$ is contained in $\xi^\gridded$, we know that this is how $\alpha$ and $\omega$ are gridded.
Furthermore, by Corollary~\ref{cor-sandwich}, no point of $\xi-p$ that succeeds $\alpha^\gridded$ and precedes $\omega^\gridded$ is misplaced.
Thus we may assume that $\alpha$ is a suffix of $\sigma$ and that $\omega$ is a prefix of $\tau$.

Let $c$ denote the cell that contains the points $v_{m-\ell}$ (belonging to $\alpha^\gridded$) and $v_{m+\ell}$ (belonging to $\omega^\gridded$), and let $S_c$ denote the entries of $\beta_p^\gridded$ that lie in cell $c$. Note that $S_c$ consists of entries from $\alpha^\gridded$ and $\omega^\gridded$, as well as possibly some other entries of $\beta_p$, but all of the entries must form a monotone sequence. Without loss of generality, we may assume that the entries in $S_c$ form a monotone increasing sequence, oriented from bottom left to top right.

We now consider the action of re-inserting $p$ into $\beta_p^\gridded$. It must be placed in cell $c$, and within this cell it must be placed above and to the right of $v_{m+\ell}$, and below and to the left of $v_{m-\ell}$. However, were $S_c\cup\{p\}$ still monotone increasing, then we would have a valid $M$-gridding of $\beta$, which is not possible. Thus there exists $p'\in S_c$ such that $\{p,p'\}$ is a copy of 21. Furthermore, since $S_c$ is monotone, we conclude that $p'$ must also lie above and to the right of $v_{m+\ell}$, and below and to the left of $v_{m-\ell}$.

Now let $q=v_{m-\ell}$ and $r=v_{m-2\ell}$, so that in $\beta_p^\gridded$ the points $p,q,r$ belong to cell $c$, and are (from left to right) successive entries of the coil.
Consider $\beta_r=\beta - r$.
We have $\beta_r\in\Grid(M)$, and in any gridding $\beta_r^\gridded$, the coil points $v_1,\dots,v_{m-2\ell-1}$ and $v_{m-2\ell+1},\dots,v_K$ form $M$-coils each of length at least $(2B+1)L$, this bound being tight in the former case.
By the same argument as that applied to $\sigma$ and $\tau$, we conclude that in $\beta_r^\gridded$ the cell $c$ must contain $v_{m+\ell}$, $p$, and $q$. Furthermore, cell $c$ must also contain $p'$ (since $p'$ is sandwiched by $v_{m+\ell}$ and $q$), but $\{p,p'\}$ was a copy of 21, while $\{p',q\}$ is a copy of $12$. This is a contradiction to the monotonicity of the entries in cell $c$ of $\beta_r^\gridded$, and completes the proof.
\end{proof}

We require one final ingredient in order to prove Theorem~\ref{thm-finite-basis}. Given a permutation class $\C$, the \emph{one-point extension} of $\C$ is the class
\[\C^{+1} = \{ \pi : \pi - p\in\C\text{ for some point }p\text{ of }\pi\}.\]

\begin{lemma}[{See Brignall and Vatter~\cite[Theorem 4.5]{bv:lwqo-for-pp:}}]\label{lem-lwqo-one-point-extension}
If $\C$ is an lwqo permutation class, then the one-point extension class $\CCC^{+1}$ is also lwqo.
\end{lemma}

We are finally in a position to complete the proof that $\Grid(M)$ is finitely based when $M$ is unicyclic.

\begin{proof}[Proof of Theorem~\ref{thm-finite-basis}]
By Proposition~\ref{prop-pseudoforest-on-a-pmm}, we can assume that $M$ is a unicyclic partial multiplication matrix.

Let $B$ be the basis of $\CCC=\Grid(M)$, and recall that $B$ must be an antichain. By Lemma~\ref{lem-basis-bound}, there is an absolute bound $K$ on the length of the longest coil contained in any basis element of $B$. The same statement holds for any one-point deletion of an element in $B$.

Let $\DDD$ denote the subclass of $\CCC$ comprising all permutations whose coils have length less than $K$. By Theorem~\ref{thm-lwqo-char}, $\DDD$ is labelled well quasi-ordered.

Take $\beta\in B$. Any one-point deletion $\beta^-$ must, by definition, belong to $\CCC$, but since $\beta^-$ contains only coils of length less than $K$, we must in fact have $\beta^-\in\DDD$. Thus, $\beta\in \DDD^{+1}$. Hence $B\subset \DDD^{+1}$, and $\DDD^{+1}$ is lwqo by Lemma~\ref{lem-lwqo-one-point-extension}, which means that $B$ must be finite.
\end{proof}

%
%
\subsection{Two grid classes that are not finitely based}\label{subsec-not-finitely-based}

Our proof of Theorem~\ref{thm-finite-basis} relies on Lemma~\ref{lem-coil-gridding-off-cycle}, which limits how many points of a coil can be gridded in cells that are not part of the unique cycle. By contrast, if $M$ is a gridding matrix which has two (or more) cycles that are identical (that is, the submatrices corresponding to the rows and columns involved in each cycle are equal), then any coil that can be gridded in one cycle can equally well be gridded in the other. Our counterexamples to Conjecture 2.3 of~\cite{Huczynska2006} harness this property.

\begin{prop}\label{prop-bicyclic-inf-basis}
The classes $\Grid(M)$ and $\Grid(N)$, where $M = \gcfour{5}{-1,0,0,1,1}{0,-1,1,0,0}{-1,1,-1,0,0}{1,0,0,-1,0}$ and $N = \gcfour{6}{0,-1,0,0,1,1}{0,0,-1,1,0,0}{-1,0,1,-1,0,0}{0,1,0,0,-1,0}$, are not finitely based.
\end{prop}

Note that in Proposition~\ref{prop-bicyclic-inf-basis}, $G_M$ comprises a single component with two identical cycles, while $G_N$ has two unicyclic components (and again the two cycles are identical).

\begin{proof}
	We claim that the sequence of permutations
	\[\pi_1=3\,5\,1\,6\,4\,8\,2\,7,\quad
		\pi_2=5\,9\,1\,7\,3\,8\,6\,10\,4\,12\,2\,11,\quad
		\pi_3=7\,13\,1\,11\,3\,9\,5\,10\,8\,12\,6\,14\,4\,16\,2\,15,\ \dots\]
is an infinite antichain in the basis of each of these classes.	These permutations are contained in $\grid{\gctwo{4}{0,-1,1,1}{-1,1,-1,0}}$, and the permutation $\pi_k$ is constructed within this class by taking a coil of length $4k+1$, and adding a single point to each of the cells corresponding to the edges not on the cycle -- see Figure~\ref{fig-antichain-basis}.

\begin{figure}
{\centering
\begin{tikzpicture}[scale=0.4]
	\plotpermgrid{3,5,1,6,4,8,2,7}
	\foreach \x in {1.5,3.5,7.5} \draw[ultra thick] (\x,0.5) -- ++(0,8);
	\draw[ultra thick] (0.5,4.5) -- ++(8,0);
	\draw[thin,gray] (3)--(4)--(6)--(5)--(1)--(2)--(8)--(7);
	\begin{scope}[shift={(10,0)},scale=0.67]
		\plotpermgrid{5,9,1,7,3,8,6,10,4,12,2,11}
		\foreach \x in {1.5,5.5,11.5} \draw[ultra thick] (\x,0.5) -- ++(0,12);
		\draw[ultra thick] (0.5,6.5) -- ++(12,0);
		\draw[thin,gray] (5)--(6)--(8)--(7)--(3)--(4)--(10)--(9)--(1)--(2)--(12)--(11);
	\end{scope}
	\begin{scope}[shift={(20,0)},scale=0.5]
		\plotpermgrid{7,13,1,11,3,9,5,10,8,12,6,14,4,16,2,15}
		\foreach \x in {1.5,7.5,15.5} \draw[ultra thick] (\x,0.5) -- ++(0,16);
		\draw[ultra thick] (0.5,8.5) -- ++(16,0);
		\draw[thin,gray] (7)--(8)--(10)--(9)--(5)--(6)--(12)--(11)--(3)--(4)--(14)--(13)--(1)--(2)--(16)--(15);
	\end{scope}
\end{tikzpicture}\par}
\caption{Griddings with respect to the matrix $\gctwo{4}{0,-1,1,1}{-1,1,-1,0}$ of the first three elements of the infinite antichain $\{\pi_1,\pi_2,\dots\}$, which is contained in the bases of both $\Grid(M)$ and $\Grid(N)$ (proof of Proposition \ref{prop-bicyclic-inf-basis}).}\label{fig-antichain-basis}
\end{figure}

To establish the claim, first note that $\Grid(N)\subseteq \Grid(M)$.
It suffices to show that each $\pi_k$ does not lie in $\Grid(M)$, while the removal of any point results in a permutation in $\Grid(N)$.

First, consider removing some point $p$ from $\pi_k$, to form the permutation $\pi_k - p$. If $p$ is the leftmost point, then the same gridding as was used to construct $\pi_k$ shows that $\pi_k-p$ belongs to $\grid{\gctwo{3}{-1,1,1}{1,-1,0}}$, and $\gctwo{3}{-1,1,1}{1,-1,0}$ is a submatrix of $N$, which establishes that $\pi_k-p\in\Grid(N)$. A similar argument applies if $p$ is the rightmost point.
Lastly, if $p$ is a point of the coil, then $\pi_k-p$ can be written as a $\boxplus$ sum of a permutation in $\Grid(\gctwo{3}{-1,1,1}{1,-1,0})$ and a permutation in $\Grid(\gctwo{3}{0,-1,1}{-1,1,-1})$.
This partition of the points of $\pi_k-p$ naturally defines an $N$-gridding of $\pi_k-p$, and establishes that $\pi_k-p\in\Grid(N)$ for any point $p$.


It remains to show that $\pi_k\not\in\Grid(M)$. To do this, we consider the permutation $\lambda=3516472$, and note that (1) $\lambda \leq \pi_k$ (since $\lambda$ is formed from the leftmost point of $\pi_k$ together with the first six points of the coil) and (2) $\lambda$ has a unique gridding in $\grid{\gctwo{4}{0,-1,1,1}{-1,1,-1,0}}$ (since no point can be placed in the rightmost column, thus points in the upper row must avoid $231$ and $132$, from which it follows routinely that there is a single gridding). These two facts are enough to establish that there is exactly one embedding of $\lambda$ in $\pi_k$.

We claim further that the following is a unique $M$-gridding of $\lambda$:

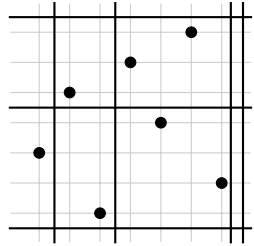
{\centering
\begin{tikzpicture}[scale=.4]
	\draw[thin,black!20] (0.05,0.05) grid ++(7.9,7.9);
	\plotperm{3,5,1,6,4,7,2}
	\foreach \x in {1.5,3.5,7.3,7.7}
		\draw[thick] (\x,0) -- ++(0,8);
	\foreach \y in {.5,4.5,7.5}
		\draw[thick] (0,\y) -- ++(8,0);
\end{tikzpicture}
\par}

This claim can be easily established by an exhaustive computer search, but below we give an outline of a direct proof.

\begin{quote}
Consider any $M$-gridding of $\lambda$. First, note that in this gridding no point can be in the rightmost column, as this would require at least the top six points to be in the top row, but these points form the pattern $245361$ which is not in $\Grid(\gcone{3}{-1,1,1})$.

Next, the leftmost point must be in the leftmost column. If not, then $\lambda\in\grid{\gcfour{3}{-0,0,1}{-1,1,0}{1,-1,0}{0,0,-1}}\subseteq \grid{\gctwo{2}{-1,1}{1,-1}}=\av(2143,3412)$, yet $\lambda$ contains both 2143 and 3412.

Since the rightmost column is empty, the top \emph{two} rows of any gridding must avoid $231$ and $132$, and thus the top two rows can contain at most the top three points of~$\lambda$. In particular, the leftmost point of $\lambda$ must be in one of the lower two rows. Furthermore, the leftmost point cannot be in the bottom row, as then the lowest three points (which form a copy of $312$) are in the bottom row, and this is not possible. Thus the leftmost point must lie in the first column, and in the second row.

Now, suppose that the top point is placed in the top row. It clearly cannot be in the leftmost column, so it would need to be in the second column from the right. In turn, this implies that the rightmost point is in the same column, but in the bottom row. We now follow the points round the coil, with every point belonging either to the top row or the bottom, which eventually leads to a contradiction with the placement of the leftmost point. Thus, there are no points in the top row.

A similar argument shows that the bottom point cannot be in the bottom row. Consequently, $\lambda$ must embed in the middle two rows, but since we know that $\lambda$ has a unique gridding as a member of $\grid{\gctwo{4}{0,-1,1,1}{-1,1,-1,0}}$, the claim follows.
\end{quote}

To complete the proof, we now attempt to grid $\pi_k$.
The leftmost point and the first six points of the coil of $\pi_k$ form a copy of $\lambda$, and thus have a unique $M$-gridding that uses only the middle two rows. We now consider each successive point from the coil in turn: each must also be placed in the middle two rows in any gridding. This eventually forces the rightmost point to be placed in the middle two rows, but this is impossible. Thus, $\pi_k\not\in\Grid(M)$.
\end{proof}

%
%
%
%
%
%
%
%
\section{Well quasi-ordering in cyclic classes}\label{sec-wqo}

We now turn our attention to characterising well quasi-ordering (without labels) in subclasses of grid classes. Here, we restrict our scope to cyclic grid classes only; we will discuss obstructions to the broader question of wqo in subclasses of pseudoforest grid classes in the concluding remarks.

Throughout this section, $M$ denotes a partial multiplication matrix whose graph is a cycle of length $\ell$, and we consider a subclass $\C\subseteq\Grid(M)$.

First, clearly if $\C$ contains only bounded length coils then it is lwqo by Theorem~\ref{thm-lwqo-char}, and hence it is wqo. On the other hand, if $\C$ contains arbitrarily long coils, then we cannot immediately conclude that the class is not well quasi-ordered. Indeed, it is shown in~\cite{bev:lwqo} that the class
\[
\av\!\left(\!\!\begin{array}{c}2143, 2413, 3412, 314562, 412563, 415632, 431562,\\ 512364, 512643, 516432, 541263, 541632, 543162\end{array}\!\!\right)
\]
is wqo, yet this class is contained in $\grid{\gctwo{2}{-1,1}{1,-1}}$ and admits arbitrarily long coils.

Roughly speaking, to construct an antichain from a collection of coils, we need to use additional points to mimic the labelling used in Lemma~\ref{lem-coil-antichains}, in which the first and last points of the coil are given a different label from the interior vertices. In the case of cyclic classes, there is essentially only one way in which this can happen, which we now describe.

Given a gridded coil $\OOO = v_1,v_2,\dots,v_n$, let $\alpha_\OOO^\gridded$ be the $M$-gridded permutation with point set
\[
	\{v_1^1,v_1^2,v_2,\dots,v_{n-1},v_n^1,v_n^2\},
\]
in which the point $v_1$ of $\OOO$ is inflated to a pair of points $v_1^1,v_1^2$, forming a copy of 12 or 21 (according to the entry of $M$ corresponding to the cell containing $v_1$), and similarly $v_n$ is inflated to the pair $v_n^1,v_n^2$. We call these \emph{end-inflated gridded coils}; see Figure~\ref{fig-inflated-coil-example} for an example.

\begin{figure}
{\centering
\begin{tikzpicture}[scale=0.25]
\plotpermgrid{4,20,6,18,7,8,16,5,9,3,11,2,1,13,10,15,12,17,14,19}
\draw[thick] (7.5,0.5) -- (7.5,20.5);
\draw[thick] (14.5,0.5) -- (14.5,20.5);
\draw[thick] (0.5,8.5) -- (20.5,8.5);
\draw[thick] (0.5,14.5) -- (20.5,14.5);
\draw[gray] (7)--(8) (1)--(2);
\path [draw=gray,postaction={on each segment={mid arrow=black!75!gray}}]
          ($(7)!.5!(8)$)--(16)--(15)--(10)--(9)--(5)
        --(6)--(18)--(17)--(12)--(11)--(3)
        --(4)--(20)--(19)--(14)--(13)--($(1)!.5!(2)$);
  \draw[->] (1.5,0) -- (6.5,0);
  \draw[<-] (8.5,0) -- (13.5,0);
  \draw[<-] (15.5,0) -- (19.5,0);
  \draw[->] (0,1.5) -- (0,7.5);
  \draw[<-] (0,9.5) -- (0,13.5);
  \draw[<-] (0,15.5) -- (0,19.5);
\end{tikzpicture} \par}
\caption{The end-inflated gridded coil $\alpha^\gridded_\OOO$, where $\OOO$ is a gridded coil of length 18 in the cyclic class $\gridhash{\protect\gcthree{3}{-1,0,1}{0,1,1}{1,-1,0}}$.}\label{fig-inflated-coil-example}
\end{figure}

Note that $v_1^i,v_2,\dots,v_{n-1},v_n^j$ is order isomorphic to $v_1,\dots,v_n$, for any pair $i,j\in\{1,2\}$, and thus each such subpermutation of $\alpha_\OOO^\gridded$ is a coil. By Lemma~\ref{lem-coil-unique-gridding}, if $n\ge (\ell+1)\ell^2+1$ then there are no other $M$-griddings of the underlying permutation $v_1,\dots,v_n$, and thus by considering the subpermutations $v_1^i,v_2,\dots,v_{n-1},v_n^j$ we conclude that $\alpha_\OOO^\gridded$ is the only $M$-gridding of the ungridded permutation $\alpha_\OOO$. We call $\alpha_\OOO$ an \emph{end-inflated coil}.

We now construct a collection of $2\ell^2$ infinite antichains, using these permutations $\alpha_\OOO$. As noted in Subsection~\ref{subsec-coils}, there are $2\ell$ coils of any specified length, determined by the cell containing the first point and which of the two neighbouring cells contains the second point. Now consider the collection of \emph{all} coils of length at least $(\ell+1)\ell^2+1$ in $\Grid(M)$.
We reiterate that each such ungridded coil has a unique $M$-gridding, and hence we can allow ourselves the liberty of conflating these coils with their gridded versions.
Next, we partition the coils under consideration according to the two cells containing the first two points, and also the cell containing the last point.

More precisely, we say that a coil $\OOO$ has \emph{type} $(s_1,s_2,f)$ if its first entry is in cell $s_1$, its second is in $s_2$, and its final entry is in $f$. Note that $s_2$ must be one of the two neighbouring cells of $s_1$, and any two coils of the same type have lengths that differ by a multiple of $\ell$. Now let
\[
	\AAA_{(s_1,s_2,f)} = \big\{\alpha_\OOO : \OOO \text{ has type }(s_1,s_2,f) \text{ and }|\alpha_\OOO|\ge (\ell+1)\ell^2+3\big\}.
\]
(Note that the length of these end-inflated coils accounts for a coil of length at least $(\ell+1)\ell^2+1$ plus the two additional points required to inflate the ends.)

\begin{prop}\label{prop-cycle-antichains}
For any triple of cells $(s_1,s_2,f)$ of $M$ in which $s_2$ is in the same row or column as $s_1$, the set $\AAA_{(s_1,s_2,f)}$ is an infinite antichain.
\end{prop}

\begin{proof}
	Let $\OOO = v_1,\dots,v_n$ and $\QQQ = u_1,\dots,u_m$ be distinct coils of type $(s_1,s_2,f)$, both of lengths at least $(\ell+1)\ell^2+1$. Without loss of generality assume that $m<n$, which implies that $m\leq n-\ell$. To prove the proposition, we will show that $\alpha_{\QQQ}\not\leq \alpha_\OOO$.
	
	First, as both coils have lengths at least $(\ell+1)\ell^2+1$, their $M$-griddings are unique by Lemma~\ref{lem-coil-unique-gridding}. In order to show that $\alpha_{\QQQ}\not\leq \alpha_\OOO$, it suffices to show that $\alpha_{\QQQ}^\gridded \not\leq \alpha_\OOO^\gridded$. Indeed, by Observation~\ref{obs-pi-to-D-pi}, it suffices to show that the orientation digraph $D_{\alpha^\gridded_{\QQQ}}$ is not an induced subdigraph of $D_{\alpha^\gridded_\OOO}$.

	Without loss of generality, we may assume that in $D_{\alpha^\gridded_{\QQQ}}$, we have $u_1^1\rightarrow u_1^2$ and $u_n^1\rightarrow u_n^2$, and similarly in $D_{\alpha^\gridded_{\OOO}}$.
	
	We claim that the only vertex of outdegree 2 in $D_{\alpha^\gridded_{\QQQ}}$ is $u_1^1$, and that there are precisely $\ell-1$ vertices of outdegree 1, namely $u_1^2,u_2,\dots,u_{\ell-1}$. This claim follows readily by inspecting the following diagram, and by noting that $u_j$ has outdegree $\geq 3$ for all $j\geq\ell$.
	
	
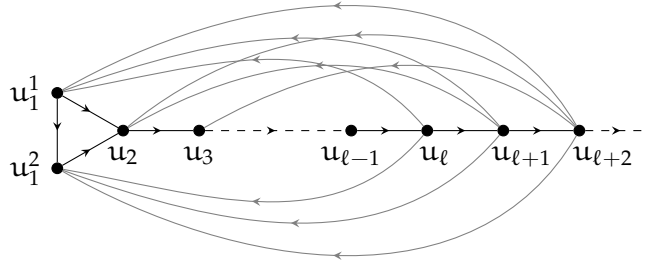
{\centering
	\begin{tikzpicture}
	\path[draw=none,use as bounding box] (0.3,-.8) rectangle (9,2.8);
		\node[permpt,label={below:$u_2$}] (u2) at (2,1) {};
		\node[permpt,label={below:$u_3$}] (u3) at (3,1) {};
		\node[permpt,label={left:$u_1^1$}] (u11) at ($(u2)+(150:1)$) {};
		\node[permpt,label={left:$u_1^2$}] (u12) at ($(u2)+(-150:1)$) {};
	
		\node[permpt,label={below:$u_{\ell-1}$}] (ulm1) at (5,1) {};
		\node[permpt,label={[xshift=3pt]below:$u_{\ell}$}] (ul) at (6,1) {};
		\node[permpt,label={[xshift=7pt]below:$u_{\ell+1}$}] (ulp1) at (7,1) {};
		\node[permpt,label={[xshift=9pt]below:$u_{\ell+2}$}] (ulp2) at (8,1) {};
		\draw (u11) edge[mid arrow] (u12)
			(u11) edge[mid arrow] (u2)
			(u12) edge[mid arrow] (u2)
			(u2) edge[mid arrow] (u3)
			(ulm1) edge[mid arrow] (ul)
			(ul) edge[mid arrow] (ulp1)
			(ulp1) edge[mid arrow] (ulp2);
		\draw[dashed] (u3) edge[mid arrow] (ulm1) (ulp2) edge[mid arrow] ++(1,0);
		\draw[gray] (ul) edge[out=130,in=10,mid arrow] (u11)
		 (ul) edge[out=-130,in=-10,mid arrow] (u12)
		 (ulp1) edge[out=130,in=20,mid arrow] (u11)
		 (ulp1) edge[out=-130,in=-20,mid arrow] (u12)
		 (ulp1) edge[out=140,in=30,mid arrow] (u2)
		 (ulp2) edge[out=120,in=30,mid arrow] (u11)
		 (ulp2) edge[out=-120,in=-30,mid arrow] (u12)
		 (ulp2) edge[out=130,in=40,mid arrow] (u2)
		 (ulp2) edge[out=140,in=30,mid arrow] (u3);
	\end{tikzpicture}\par}

A similar comment (replacing `outdegree' with `indegree') applies to the last $\ell$ vertices, and
then the obvious analogues apply to $D_{\alpha^\gridded_\OOO}$.
	
	Now suppose, for a contradiction, that there exists some embedding of $D_{\alpha^\gridded_{\QQQ}}$ in $D_{\alpha^\gridded_\OOO}$. Such an embedding naturally induces an embedding of the coil $\QQQ$ in $\OOO$, and we note that this induced embedding must use a contiguous set of vertices of $D_\OOO$ by an argument similar to the one used in the proof of Lemma~\ref{lem-coil-antichains}. Consequently, since $m\leq n-\ell \leq n-4$, it is not possible to embed $D_{\alpha^\gridded_{\QQQ}}$ in $D_{\alpha^\gridded_\OOO}$ so that at least one vertex in $\{v_1^1,v_1^2\}$ and at least one vertex in $\{v_n^1,v_n^2\}$ is used. We will consider the case in which an embedding uses neither $v_1^1$ nor $v_1^2$; the other case follows by an analogous argument.
	
	Let $j>1$ be the smallest index such that some vertex of $D_{\alpha^\gridded_{\QQQ}}$ maps to  $v_j$. Note that in the induced subdigraph of $D_{\alpha^\gridded_{\QQQ}}$ on $\{v_j,\dots,v_{n-1},v_n^1,v_n^2\}$, only the vertices $v_j,\dots,v_{j+\ell-1}$ have outdegree less than 3.
All these vertices must be used in the embedding of $D_{\alpha^\gridded_{\QQQ}}$, and thus it must be the vertices $u_1^1,u_1^2,u_2,\dots,u_{\ell-1}$ that map to $v_j,\dots,v_{j+\ell-1}$, in some order. However, the digraph induced on $v_j,\dots,v_{j+\ell-1}$ is a directed cycle of length $\ell$, while that on $u_1^1,u_1^2,u_2,\dots,u_{\ell-1}$ is not, and thus clearly there is no such embedding.
\end{proof}	

For a class $\C\subseteq \Grid(M)$ to be wqo, it must clearly have only finite intersection with each $\AAA_{(s_1,s_2,f)}$. We will show that this condition is sufficient: that is, if $\C\cap\AAA_{(s_1,s_2,f)}$ is finite for all types~$(s_1,s_2,f)$, then $\C$ is wqo. To do this, we take a similar approach to the proof of Theorem~\ref{thm-lwqo-char}: for a class $\C$ that has finite intersection with each $\AAA_{(s_1,s_2,f)}$, we seek to identify a refined acyclic gridding matrix $N$ so that permutations in $\C$ can be constructed using members of $\Grid(N)$. However, unlike the case for labelled well-quasi ordering, here it is possible for $\C$ to contain long $M$-coils (providing the ends of such coils cannot be inflated). This means that the earlier approach of simply taking the coil decomposition of each $\pi\in\C$ to construct a suitable $N$ will not work, as there is no bound on the length of such a decomposition.

Instead, consider any indivisible $\pi^\gridded\in\C^\gridded\subseteq \Gridhash(M)$ with coil decomposition $B_1,B_2,\dots,B_m$ and associated coil $v_1,\dots,v_m$. Roughly speaking, we seek to identify indices $i$ and $j$, with $1\leq i<j\leq m$, so that the points in $B_1,\dots,B_{i-1}$ are order isomorphic to a coil (the `leading' coil), the points in $B_i,\dots,B_j$ are order isomorphic to an element of $\Grid(N)$ (the `body'), and those in $B_{j+1},\dots,B_m$ are order isomorphic to another coil (the `trailing' coil). For the body to belong to an acyclic grid class $\Grid(N)$, we need to ensure that $|j-i|$ is bounded.

Since the points in $B_1,\dots,B_{i-1}$ must form a coil, one might hope that each box $B_k$ in this range contains precisely one point, namely $v_k$. One could then choose the index $i$ so that $B_i$ is the first non-singleton box. Analogous comments apply to the boxes $B_{j+1},\dots,B_m$ and the choice of $j$. However, not all coil decompositions of gridded coils are formed entirely of singleton boxes, as illustrated in Figure~\ref{fig-coil-bad-decomposition}. Thus, our first task is to select `good' coil decompositions when this is possible, and this is done in the following technical proposition.

\begin{figure}
{\centering
\begin{tikzpicture}[scale=0.45]
\plotpermgrid{2,12,4,11,6,9,5,7,3,8,1,10}
\draw[thick] (6.5,0.5) -- (6.5,12.5);
\draw[thick] (0.5,6.5) -- (12.5,6.5);
\path [draw=gray,postaction={on each segment={mid arrow=black!75!gray}}]
          (7)--(5)--(6)--(9)--(8)--(3)--(4)--(11)--(10)--(1)--(2)--(12);
  \draw[->] (1.5,0) -- (5.5,0);
  \draw[<-] (7.5,0) -- (11.5,0);
  \draw[->] (0,1.5) -- (0,5.5);
  \draw[<-] (0,7.5) -- (0,11.5);
\begin{scope}[shift={(14,0)}]
\plotpermgrid{2,12,4,11,6,9,5,7,3,8,1,10}
\draw[thick] (6.5,0.5) -- (6.5,12.5);
\draw[thick] (0.5,6.5) -- (12.5,6.5);
\foreach \x/\y/\dx/\dy [count=\n] in {5/5.2/1/.8,5/7/1/2,7/7/3/2,7/3/3/2,3/3/1/2,3/10/1/1,11/10/1/1,11/1/1/1,1/1/1/1,1/11.5/1/1} {
	\draw[thick] (\x,\y) rectangle ++ (\dx,\dy);
	\node at (\x+\dx/2,\y+\dy/2) {\footnotesize $B_{\n}$};
	}
\end{scope}
\end{tikzpicture} \par}
\caption{On the left, a gridded coil of length 12 in $\gridhash{\protect\gctwo{2}{-1,1}{1,-1}}$, and on the right, a coil decomposition of the same gridded coil into boxes $B_1,\dots,B_{10}$ where $B_3$ and $B_4$ are nonsingleton.}\label{fig-coil-bad-decomposition}
\end{figure}


\begin{prop}\label{prop-specific-coil-decomp}
	Let $M$ be a cyclic partial multiplication matrix and let $\ell$ denote the length of the cycle. For any $M$-indivisible $\pi^\gridded\in\Gridhash(M)$ there exists a coil decomposition
	\[
		B_1,B_2,\dots,B_m
	\]
	and associated coil $v_1,\dots,v_m$ such that at least one of the following holds.
	\begin{enumerate}[(a)]
		\item $|B_1|=|B_2|=\cdots = |B_\ell|=1$;
		\item $|B_{\ell+1}|\ge 2$;
		\item For some $1\le i\le \ell$, there exists $x\in B_i\setminus \{v_i\}$ such that $x\rightarrow v_{i+1}$ in $D_{\pi^\gridded}$. (That is, $x$ and $v_i$ have the same relationship to the rest of the coil $v_1,\dots,v_m$.)
	\end{enumerate}
\end{prop}

\begin{proof}
Take any coil decomposition $B_1,\dots,B_m$ of $\pi^\gridded$, and let the associated coil be $v_1,\dots,v_m$. Let $z_1,\dots,z_\ell$ denote the last points in each cell (indexed so that $z_i\in B_i$ for $1\le i\le\ell$). Note that~$v_1=z_1$, and indeed $v_i=z_i$ if and only if $|B_i|=1$.

Suppose the decomposition $B_1,\dots,B_m$ satisfies none of (a), (b) or (c); we will
show how to modify it to form another decomposition that does satisfy one of the conditions.
Let $i$ ($\le\ell$) denote the smallest index such that $|B_i|> 1$. Thus $v_i\ne z_i$ (and both of these vertices are contained in box $B_i$), but $v_j=z_j$ for all $j<i$. Furthermore, to avoid condition (c) we know that~$v_{i+1}\rightarrow z_i$ in $D_{\pi^\gridded}$. However, we also know that $z_i\rightarrow z_{i+1}$ which establishes that $v_{i+1} \neq z_{i+1}$, and thus $|B_{i+1}|> 1$. Similar comments apply to the relationship between the points in~$B_{i+1}$ and~$B_{i+2}$ (assuming $i+2\le \ell$), and iterating this argument we establish that
\[
	v_j\neq z_j,\quad \text{and}\quad v_{j+1}\rightarrow z_j
\]
for all $i\le j\le\ell$. The following diagram illustrates a typical configuration (with $\ell=6$ and $i=3$).

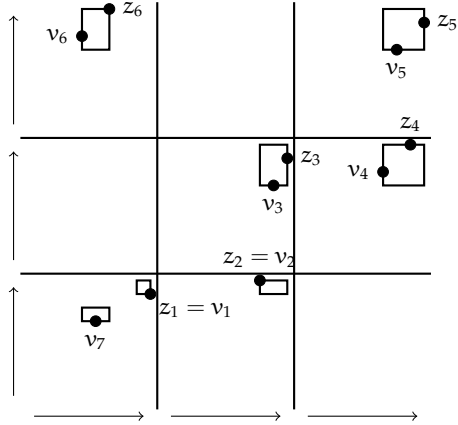
{\centering
\begin{tikzpicture}[scale=0.18,label distance=-1pt]
\node[permpt,label={[label distance=-3pt]below right:\footnotesize$z_1=v_1$}] at (10,9) {};
\node[permpt,label={above:\footnotesize$z_2=v_2$}] at (18,10) {};
\node[permpt,label={right:\footnotesize$z_3$}] at (20,19) {};
\node[permpt,label={below:\footnotesize$v_3$}] at (19,17) {};
\node[permpt,label={above:\footnotesize$z_4$}] at (29,20) {};
\node[permpt,label={left:\footnotesize$v_4$}] at (27,18) {};
\node[permpt,label={right:\footnotesize$z_5$}] at (30,29) {};
\node[permpt,label={below:\footnotesize$v_5$}] at (28,27) {};
\node[permpt,label={right:\footnotesize$z_6$}] at (7,30) {};
\node[permpt,label={left:\footnotesize$v_6$}] at (5,28) {};
\node[permpt,label={below:\footnotesize$v_7$}] at (6,7) {};
\draw[thick] (10.5,0.5) -- (10.5,30.5);
\draw[thick] (20.5,0.5) -- (20.5,30.5);
\draw[thick] (0.5,10.5) -- (30.5,10.5);
\draw[thick] (0.5,20.5) -- (30.5,20.5);
\foreach \x in {1.5,11.5,21.5} {
  \draw[->] (\x,0) -- ++(8,0);
  \draw[->] (0,\x) -- ++(0,8);
}
\foreach \x/\y/\dx/\dy [count=\n] in {9/9/1/1,18/9/2/1,18/17/2/3,27/17/3/3,27/27/3/3,5/27/2/3,5/7/2/1} {
	\draw[thick] (\x,\y) rectangle ++ (\dx,\dy);
	}
\end{tikzpicture} \par}

We now split the considerations into two cases, depending on the sizes of the blocks $B_i,\dots,B_\ell$.

\textbf{Case 1:} \textit{$|B_j|\ge 3$ for some $i\le j\le \ell$.}
 Take $x\in B_j\setminus \{z_j,v_j\}$, and note that $v_j\rightarrow x\rightarrow z_j$. Take the coil decomposition $B_1',B_2',\dots$ that begins with $B_1'=\{z_j\}$. Since
\[
	z_j\rightarrow \cdots \rightarrow z_\ell \rightarrow z_1\rightarrow v_2\rightarrow \cdots \rightarrow v_{j-1}\rightarrow \{v_j,x\}
\]
we see that $B_{\ell+1}'$ contains two entries. Thus the coil decomposition $B_1',B_2',\dots$ satisfies (b).

\textbf{Case 2:} \textit{$|B_j|=2$ for all $i\le j\le \ell$.} Thus $B_j$ consists precisely of the points $v_j$ and $z_j$. In this case, take the coil decomposition starting with $B_1''=\{z_i\}$. By our assumptions, it is straightforward to verify that this decomposition begins with $\ell$ singleton boxes, and thus (a) is satisfied by this coil decomposition.
\end{proof}

We can now present our key structural lemma for permutation classes that have finite intersection with the antichains introduced earlier. We let $\AAA = \bigcup_{(s_1,s_2,f)\in [\ell]^3}\AAA_{(s_1,s_2,f)}$ denote the set of all sufficiently long end-inflated $M$-coils.

\begin{lemma}\label{lem-coil-decomp-antichain}
	Let $M$ be a cyclic partial multiplication matrix, and let $\C\subseteq\Grid(M)$ such that $\C\cap\AAA$ is finite. Then there exists a constant $K$ such that any $M$-indivisible element $\pi^\gridded$ of $\C^\gridded$ has a coil decomposition
	\[
		B_1,\dots,B_m
	\]
	with the property that whenever $|B_i|>1$ and $|B_j|>1$, then $|j-i|\le K$.
\end{lemma}

\begin{proof}
	Consider any $M$-indivisible $\pi^\gridded\in \C^\gridded$ and take a coil decomposition $B_1,\dots,B_m$ with associated coil $v_1,\dots,v_m$ which satisfies one (or more) of the conditions in Proposition~\ref{prop-specific-coil-decomp}. Now take $i,j$ (and without loss suppose $i<j$) as in the statement of the lemma, and consider points~$x,y$, distinct from $v_i$ and $v_j$, such that $x\in B_i$ and $y\in B_j$.
	
	In the digraph $D_{\pi^\gridded}$, we know that $v_i\rightarrow x$ and $v_j\rightarrow y$. Furthermore, by construction we know that $v_{j-1}\rightarrow y$, while there are two possibilities for the relationship between $x$ and the coil point~$v_{i+1}$ in the cell $i+1$: either $x\rightarrow v_{i+1}$, or $v_{i+1}\rightarrow x$. We will show in each case how to construct an element of $\AAA$.
	
\textbf{Case 1:} $x\rightarrow v_{i+1}$. We claim that the set of points
	\[
	S=\{x,v_i,v_{i+1},\dots,v_j,y\}
	\]
	forms a gridded subpermutation of $\pi^\gridded$ that is order isomorphic to an element of $\AAA$. To see this, note that $v_i,\dots,v_j$ is a coil, so it suffices to show that $\{v_i,x\}$ and $\{v_j,y\}$ form intervals within the subpermutation defined by the set $S$. Since $v_i$ and $x$ belong to $B_i$, they belong to the same cell in~$\pi^\gridded$. Consequently the only points in $S$ that can separate $\{v_i,x\}$  must belong to $B_{i-1}\cup B_i\cup B_{i+1}$. However, $S\cap (B_{i-1}\cup B_i\cup B_{i+1})=\{v_i,x,v_{i+1}\}$, yet $x\rightarrow v_{i+1}$ and $v_i\rightarrow v_{i+1}$, which establishes that nothing in $S$ separates $\{v_i,x\}$. A similar argument applies to the pair $\{v_j,y\}$, and so the points in $S$ form an element of $\AAA$ of length $j-i+3$.
	
\textbf{Case 2:} $v_{i+1}\rightarrow x$. There are two subcases, depending on whether $i\leq \ell$ or $i>\ell$.
	
	If $i > \ell$, then we claim that the set of points
	\[
	S=\{x,v_{i-\ell},v_{i-\ell+1},\dots,v_j,y\}
	\]
	forms a gridded subpermutation of $\pi^\gridded$ of length $j-i+\ell+3$ that is order isomorphic to an element of $\AAA$. The argument is similar to the previous case: note that $x\rightarrow v_{i-\ell}$ and $v_{i-\ell}\rightarrow v_{i-\ell+1}$ together imply that $x\rightarrow v_{i-\ell+1}$.
	
	If $i\leq \ell$ then the coil decomposition of $\pi^\gridded$ cannot satisfy condition (a) of Proposition~\ref{prop-specific-coil-decomp}, so it must satisfy (b) and/or (c). If the decomposition satisfies (b), then instead of considering $x\in B_i$, we take $x'\in B_{\ell+1}$ with the property that $x'\neq v_{\ell+1}$. The point $x'$ lies in a cell of index greater than $\ell$ so is covered by one of the two earlier arguments, and we obtain an element of $\AAA$ of length at least $j-i-\ell+2$.
	
	On the other hand, if the decomposition satisfies neither~(a) nor~(b), then it must satisfy~(c), which means that there exists some index $i' \leq \ell$ and a point $x'\in B_{i'} \setminus\{v_{i'}\}$ such that $x'\rightarrow v_{i'+1}$. If we consider the point $x'$ instead of $x$, then this case is again covered by one of the three earlier arguments, and we obtain an element of $\AAA$ of length at least $j-i-\ell+2$.

This completes the case analysis.	In every case we have constructed an end-inflated coil in $\AAA$, and the shortest this permutation can be is  $j-i-\ell+2$. If $K'$ denotes the length of the longest element in $\C\cap \AAA$, then we have $j-i-\ell+2\le K'$, and so the result follows with $K= K'-2+\ell$.
\end{proof}

We are now nearly ready to state and prove our characterisation of wqo for subclasses of cyclic classes. As noted earlier, the proof follows roughly the same argument as that for Theorem~\ref{thm-lwqo-char}, with the added complications that arise by the need to handle leading and trailing coils. Thus, before we state and prove our characterisation, we construct our decomposition in one final technical lemma.

\begin{lemma}\label{lem-coil-body-coil-decomp}
	Let $M$ be a cyclic partial multiplication matrix, let $\C\subseteq\Grid(M)$ such that $\C\cap\AAA$ is finite, and let $K$ be as in Lemma~\ref{lem-coil-decomp-antichain}. For any $M$-indivisible element $\pi^\gridded\in\Gridhash(M)$, there exists an acyclic matrix $N$ with at most $K+2$ non-zero entries, such that $\pi^\gridded$ can be reversibly encoded as a triple $(\sigma_\pi^\natural,a_\pi,b_\pi)\in \Gridhash(N)\times\mathbb{N}\times\mathbb{N}$.
\end{lemma}

\begin{proof}
Consider any $M$-indivisible gridded permutation $\pi^\gridded \in\C^\gridded$. If $\pi^\gridded$ occupies a single cell, then it has length 1. In this case, the lemma is satisfied by constructing $N$ to have the same dimensions as $M$, but with a single non-zero entry that corresponds to the cell containing the one point of $\pi^\gridded$. We set $\sigma_\pi^\natural=\pi^\gridded$, and $a_\pi=b_\pi=0$.

Thus we may assume from now on that $\pi^\gridded$ occupies all $\ell$ cells of the cycle. Take a coil decomposition $B_1,\dots,B_m$ of $\pi^\gridded$ that satisfies Lemma~\ref{lem-coil-decomp-antichain}, and note that $m\geq \ell$.

If there are boxes with more that one point, let $i$ (resp. $j$) be the smallest (resp. the largest) index of such a box. Otherwise set $i=j=2$. Thus, the only nonsingleton boxes are among
	\[
		B_i,B_{i+1},\dots,B_j,
	\]
	and by Lemma~\ref{lem-coil-decomp-antichain} we have $|j-i|\leq K$.

	Now let $\sigma$ be the subpermutation of $\pi^\gridded$ formed from the points in $B_{i-1}\cup B_i\cup\cdots \cup B_{j+1}$.
	Note that  $i>1$ since $|B_1|=1$. If $B_{j+1}$ does not exist (which occurs only if $j=m$), then we `pad' our coil decomposition by creating an empty final box $B_{j+1}$, placed in the same cell as $B_{j-\ell+1}$, and preceding $B_{j-\ell+1}$ according to that cell's orientation. 	
	
	We grid $\sigma$ according to the coil decomposition: let $N$ denote the acyclic partial multiplication matrix whose non-zero entries correspond to the boxes $B_{i-1},\dots,B_{j+1}$ (essentially following the same process as in the proof of Lemma~\ref{lem-griddable-acyclic}). By construction, $\sigma\in\Grid(N)$, and let $\sigma^\natural$ denote the $N$-gridding inherited from the coil decomposition of $\pi^\gridded$. (Note, we use $\natural$ to emphasise that this is an $N$-gridding.)
	
 	Next, we specify the two positive integers $a_\pi$ and $b_\pi$. We set $a_\pi=i-1$ to record the number of coil points that precede the first non-singleton block $B_{i}$ of $\pi^\gridded$. Similarly, we set $b_\pi=m-j$ to record the number of coil points that follow the last non-singleton block $B_{j}$. Note that $a_\pi\geq 1$, and that $b_\pi=0$ if and only if $j=m$ (indicating that there is no final coil).
 	
 	To complete the proof, we need to show that the triple $(\sigma^\natural,a_\pi,b_\pi) \in \Gridhash(N)\times\mathbb{N}\times\mathbb{N}$ is a reversible encoding of $\pi^\gridded$: that is, we must demonstrate how to recover $\pi^\gridded$ from this triple.
This is clear for the encodings of permutations of length $1$, so we need only consider those of permutations of length $\geq \ell$.
We begin by constructing the boxes $B_{i-1},\dots,B_{j+1}$ of $\pi^\gridded$, which contain precisely the points in the corresponding cells of $\sigma^\natural$. The boxes $B_{i-1}$ and $B_{j+1}$ are easily identified.

 	We now re-insert the coil points at the beginning and the end of $\pi^\gridded$, which we do iteratively.
 	Given the placement of a singleton cell $B_k = \{v_k\}$, there is a unique way in which to insert the preceding singleton cell $B_{k-1}=\{v_{k-1}\}$: we have $v_{k-1}\rightarrow v_k$, while for any $k'>k$ such that $B_{k'}$ is in the same row or column of $\pi^\gridded$, we have $x \rightarrow v_{k-1}$ for all $x\in B_{k'}$: that is, $B_{k-1}$ comes after all entries in its row or column except for $v_k$. On the other hand, the position of~$v_{k-1}$ relative to any other entries in $\pi^\gridded$ is determined by $M$.
 	Thus, we may place $v_{k-1}$ relative to all later points in a unique way.
 	A similar process allows us to reinsert entries to the other end of the coil, and the numbers of times we iterate these steps is governed by the integers $a_\pi$ and $b_\pi$, respectively.
 	Note that the case where $b_\pi=0$ corresponds precisely to the case that~$B_{j+1}$ is empty.
 	 Finally, we forget the partition of the points of $\pi^\gridded$ into the boxes $B_i$ in favour of the underlying $M$-gridding, observing that $B_i$, $B_{i+\ell}$, $B_{i+2\ell}$, \dots end up in one cell, $B_{i+1}$, $B_{i+1+\ell}$, \dots in another, and so on. Thus, we have recovered $\pi^\gridded$, and the proof is complete.
\end{proof}


\begin{thm}\label{thm-cycle-wqo-char}
Let $M$ be a cycle partial multiplication matrix, and let $\C\subseteq \Grid(M)$. Then $\C$ is wqo if and only if $\C$ contains only finitely many end-inflated $M$-coils.
\end{thm}

\begin{proof}
	One direction is immediate from Proposition~\ref{prop-cycle-antichains}: If $\C\cap\AAA$ is not finite, then $\C\cap \AAA_{(s_1,s_2,f)}$ is infinite for some triple of cells $(s_1,s_2,f)$, thus $\C$ contains an infinite antichain.
	
	Now suppose that $\C\cap\AAA$ is finite, and let $K$ be the constant from Lemma~\ref{lem-coil-decomp-antichain}. By Lemma~\ref{lem-coil-body-coil-decomp}, for each $M$-indivisible $\pi^\gridded\in\C^\gridded$, there exists an acyclic matrix $N$ with at most $K+2$ non-zero entries such that $\pi^\gridded$ can be encoded as a triple $(\sigma_\pi^\natural,a_\pi,b_\pi)\in \Gridhash(N)\times\mathbb{N}\times\mathbb{N}$. Note that $\Gridhash(N)$ is labelled well quasi-ordered by Proposition~\ref{prop-no-cycles}, and thus $\Gridhash(N)\times\mathbb{N}\times\mathbb{N}$ is well quasi-ordered by Proposition~\ref{prop-direct-prod-wqo}, when $\mathbb{N}$ is endowed with the usual (total) ordering on natural numbers.
 	
 	In general, two $M$-indivisible permutations in $\C^\gridded$ do not need to be encoded using the same acyclic matrix. However, every such matrix contains at most $K+2$ non-zero entries, and so there can be only finitely many distinct matrices arising from the encodings provided by Lemma~\ref{lem-coil-body-coil-decomp}.
 	
    For any such acyclic gridding matrix $N$, define
 	\[
	\C_N = \{\pi^\gridded \in\C^\gridded : \pi^\gridded \text{ is $M$-indivisible and }
	\sigma_\pi^\gridded\in\Grid^\gridded(N)\}.
	\]
	We need to show that any such set $\C_N$ is well quasi-ordered. Since $\Gridhash(N)\times\mathbb{N}\times\mathbb{N}$ is well quasi-ordered, this will follow by Lemma \ref{lem-order-preserving-reflecting}(ii) if we can show that the mapping
	\[
		\phi: \C_N \to \Gridhash(N)\times\mathbb{N}\times\mathbb{N},
	\]
	which sends each $\pi^\gridded$ to its encoding $(\sigma^\natural_\pi,a_\pi,b_\pi)$,
	is order-reflecting.
	
	To this end consider $\pi^\gridded$ and $\tau^\gridded$ in $\C_N$ whose encodings $(\sigma_\pi^\natural,a_\pi,b_\pi)$ and $(\sigma_\tau^\natural,a_\tau,b_\tau)$ have the property that
 	\[
 		(\sigma_\pi^\natural,a_\pi,b_\pi) \le (\sigma_\tau^\natural,a_\tau,b_\tau).
 	\]
 	This means that $\sigma_\pi^\natural \leq \sigma_\tau^\natural$ (as $N$-gridded permutations), $a_\pi\leq a_\tau$, and $b_\pi\leq b_\tau$.
 	
 	To show that $\pi^\gridded\leq \tau^\gridded$, we begin with an embedding that witnesses $\sigma_\pi^\natural \leq \sigma_\tau^\natural$, noting that the points in the singleton cells of $\sigma_\pi^\natural$ at the beginning and end (when they exist) must embed into the corresponding points in the singleton cells of $\sigma_\tau^\natural$. We now follow the process described above to embed successive points at the beginning and end of the coil decomposition of $\pi^\gridded$ into the corresponding points of $\tau^\gridded$. Since $a_\pi\leq a_\tau$, and $b_\pi\leq b_\tau$, we are guaranteed to embed all such points of $\pi^\gridded$ before we run out of points in $\tau^\gridded$. Thus, $\pi^\gridded\leq \tau^\gridded$, which completes the proof that $\C_N$ is well quasi-ordered.
 		
 	Since every $M$-indivisible permutation in $\C^\gridded$ is encoded using some matrix $N$, and since there are only finitely many possible matrices $N$, the set of all $M$-indivisible permutations in $\C^\gridded$ is contained in a finite union of  well quasi-ordered sets $\C_N$, and is thus well quasi-ordered.
	Lemma~\ref{lem-indiv-wqo} now gives that $\C^\gridded$ is well quasi-ordered.

 	Finally the fact that $\C$ is well quasi-ordered now follows by observing that the mapping
 	\[
 		\C^\gridded \to \C,\ \pi^\gridded\mapsto \pi,
 	\]
 	that removes the gridding (first considered in the proof of Lemma~\ref{lem-gridded-ungriddded-lwqo}) is order-preserving.
\end{proof}


%
%
%
%
%
%
%
\section{Concluding remarks}

\paragraph{Enumeration} We have not considered the enumeration of pseudoforest grid classes in this study. It is known (by~\cite{albert:geometric-grid-:}) that acyclic grid classes (and their subclasses) all possess rational generating functions. In his PhD thesis, Bevan~\cite[Theorem 4.6]{bevan:thesis:} used a `diagonalisation' argument to show that the \emph{gridded} permutations in a pseudoforest grid class have an algebraic generating function. He further conjectures that all pseudoforest grid classes have algebraic generating functions.

The barrier to making progress on the question of enumeration lies in our ability (or lack of) to handle permutations that possess multiple griddings. Even if Conjecture~\ref{conj-unique-griddings} could be proved, this is likely to be insufficient for what would be needed for a direct approach. Alternatively, one might wonder whether it is possible to combine the methods in~\cite{albert:geometric-grid-:} (where rationality is established by subtracting `bad' griddings from the set of all griddings, and showing that both sets of griddings are encoded by regular languages) with those in~\cite{bevan:thesis:} (where pseudoforest classes are formed from acyclic classes by identifying the points in pairs of cells), but this has so far not met with success.

\paragraph{Bases} While we have roughly indicated a boundary between grid classes that are finitely based and those that are not, there is certainly still scope for this to be tightened further. The construction employed in Proposition~\ref{prop-bicyclic-inf-basis} only works in cases where the gridding matrix $M$ possesses two identical cycles that are arranged around each other in a very particular way. Other constructions may well exist, but it is probably also the case that the methods used here to show that unicyclic classes are finitely based can be applied slightly more generally.

\paragraph{Well quasi-ordering for unicyclic classes}
The question of well quasi-ordering in unicyclic grid classes, let alone pseudoforest grid classes, is more nuanced than the arguments given in this paper for cyclic classes. While we know that every infinite antichain necessarily must include elements that contain arbitrarily long coils, there are multiple ways for these to be `anchored' in unicyclic classes. An example is given in Figure~\ref{fig-untied-widdershins}. Further results are likely possible, but may descend rapidly into a technical case analysis.

\begin{figure}
{\centering
\begin{tikzpicture}[scale=0.25]
\plotpermgrid{3,5,1,6,4,8,2,7}
\draw[thick] (1.5,0.5) -- (1.5,8.5);
\draw[thick] (7.5,0.5) -- (7.5,8.5);
\draw[thick] (3.5,0.5) -- (3.5,8.5);
\draw[thick] (0.5,4.5) -- (8.5,4.5);
\draw[gray] (3)--(4)--(6)--(5)--(1)--(2)--(8)--(7);
\begin{scope}[shift={(10,0)}]
\plotpermgrid{5,9,1,7,3,8,6,10,4,12,2,11}	
\draw[gray] (5)--(6)--(8)--(7)--(3)--(4)--(10)--(9)--(1)--(2)--(12)--(11);
\draw[thick] (1.5,0.5) -- (1.5,12.5);
\draw[thick] (11.5,0.5) -- (11.5,12.5);
\draw[thick] (5.5,0.5) -- (5.5,12.5);
\draw[thick] (0.5,6.5) -- (12.5,6.5);
\end{scope}
\begin{scope}[shift={(24,0)}]
\plotpermgrid{7,13,1,11,3,9,5,10,8,12,6,14,4,16,2,15}	
\draw[gray] (7)--(8)--(10)--(9)--(5)--(6)--(12)--(11)--(3)--(4)--(14)--(13)--(1)--(2)--(16)--(15);
\draw[thick] (1.5,0.5) -- (1.5,16.5);
\draw[thick] (15.5,0.5) -- (15.5,16.5);
\draw[thick] (7.5,0.5) -- (7.5,16.5);
\draw[thick] (0.5,8.5) -- (16.5,8.5);
\end{scope}
\end{tikzpicture} \par}
\caption{Three elements from a variant of the `Widdershins' antichain, which is contained in $\grid{\gctwo{4}{0,-1,1,-1}{-1,1,-1,0}}$.
}\label{fig-untied-widdershins}
\end{figure}

\paragraph{Polycyclic classes} We have significantly extended our understanding of pseudoforest grid classes in this article, but the corresponding questions for polycyclic grid classes still seem out of reach.
 The first issue is that not all gridding matrices can be replaced by partial multiplication matrices, since Proposition~\ref{prop-doubling-equal} does not apply to matrices that contain a negative cycle connected to another cycle. Indeed, in his PhD thesis, Waton~\cite[Proposition 4.5.14]{waton:on-permutation-:} gives the following example. With $M = \gctwo{3}{1,1,1}{-1,1,1}$, $\Grid(M)$ can be expressed as the union of two distinct proper subclasses, comprising the permutations that can be drawn on the following two diagrams.

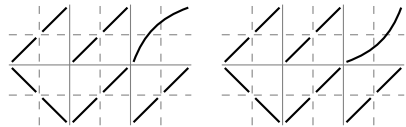
{\centering
\begin{tikzpicture}[scale=.4]
\foreach \x in {1,3,5}
	\draw[thin,dashed,black!50] (\x,0) -- ++(0,4) (\x+7,0) -- ++(0,4);
\foreach \x in {2,4}
	\draw[thin,black!50] (\x,0) -- ++(0,4) (\x+7,0) -- ++(0,4);
\foreach \y in {1,3}
	\draw[thin,dashed,black!50] (0,\y) -- ++(6,0) (7,\y) -- ++(6,0);
\foreach \y in {2}
	\draw[thin,black!50] (0,\y) -- ++(6,0) (7,\y) -- ++(6,0);
\foreach \x/\y in {2/0,4/0,0/2,2/2} {
	\draw[thick] (\x+.1,\y+.1) -- ++ (.8,.8) ++ (.2,.2) -- ++ (.8,.8)
		(\x+7.1,\y+.1) -- ++ (.8,.8) ++ (.2,.2) -- ++ (.8,.8);
}
\draw[thick] (.1,1.9) -- ++(.8,-.8) ++(.2,-.2) -- ++(.8,-.8)
	 (7.1,1.9) -- ++(.8,-.8) ++(.2,-.2) -- ++(.8,-.8)
	(4.1,2.1) edge[out=70,in=-160] (5.9,3.9)
	(11.1,2.1) edge[out=20,in=-110] (12.9,3.9);
\end{tikzpicture}
\par}
Note that both of these pictures are essentially formed from the doubled matrix $M^{\times2}$, but with deformations in the top right corner. In particular, cell $C_{54}$ in the left picture and $C_{63}$ in the right can be non-empty, and thus there are permutations in $\Grid(M)$ that are not present in $\Grid(M^{\times 2})$.

For polycyclic gridding matrices that \emph{can} be expressed as partial multiplication matrices, studies concerning the permutation pattern matching (PPM) problem (in which one investigates the complexity of counting the number of times a given pattern $\pi$ is contained in a `text' permutation $\tau$) both illustrate why these classes are so much harder, and perhaps also offer a glimpse of what the theory might look like.

In his PhD thesis, Opler~\cite[Theorem 6.1]{opler:phd} shows that the tree-width%
\footnote{Tree-width is a parameter most commonly associated with graph structure; its application to the PPM problem originates with Ahal and Rabinovich~\cite{ahal:on-the-complexity:}. For our purposes, it suffices to note that tree-width provides a measure of the structural complexity of a class.} %
 of acyclic grid classes is constant, that of pseudoforest grid classes grows like $\sqrt{n}$ (where~$n$ is the length of the permutation), and that of polycyclic classes grows linearly in $n$. This both agrees with the increase in structural complexity that we have exhibited in this paper between acyclic and pseudoforest classes, and also hints that the structure of polycyclic classes is an order of magnitude more complex again.
 
Indeed, the proof of~\cite[Theorem 6.1]{opler:phd} (which incorporates work from two extended abstracts of Jel\'inek, Opler and Pek\'arek~\cite{jelinek:a-complexity-dichotomy:,jelinek:long-paths:}) provides some insight into the structure we should expect. The following two properties that a permutation class $\C$ can possess are of particular note: the \emph{long path property}, which means that $\C$ contains, for any $k$, a monotone grid class whose graph is path of length $k$, and the \emph{deep tree property}, which means that for any $k$, the class $\C$ contains a grid class whose graph is a subdivision of a binary tree of depth $k$.  

Clearly, an acyclic grid class has neither property, while pseudoforest grid classes have the long path property but not the deep tree property, see~\cite[Lemma 3.5]{jelinek:a-complexity-dichotomy:} or~\cite[Proposition 2.34]{opler:phd} for the full proof. In the context of this paper, long paths can be constructed by considering a refined gridding of permutations whose coil decompositions are arbitrarily long, with arbitrarily many interleaved points in each set of the partition. On the other hand, that we can decompose any permutation in a pseudoforest grid class using the $M$-sum and coil decompositions goes some way towards explaining why these classes do \emph{not} possess the deep tree property.

Polycyclic classes, on the other hand, do possess the deep tree property (and thus also the long path property), see~\cite[Proposition 6]{jelinek:long-paths:} or~\cite[Proposition 2.41]{opler:phd} for the full proof. This leads us to finish with the following question: given polycyclic classes have the deep tree property, what could a `coil decomposition of polycyclic classes' look like?

\paragraph{Acknowledgements} We are grateful to Michal Opler for discussions during the preparation of this article, and to the two referees whose careful reading prompted several improvements to an earlier version.

\bibliographystyle{plainurl}
\bibliography{refs}

\def\cprime{$'$}
\begin{thebibliography}{10}

\bibitem{ahal:on-the-complexity:}
Shlomo Ahal and Yuri Rabinovich.
\newblock On the complexity of the sub-permutation problem.
\newblock In {\em Algorithm theory---{SWAT} 2000 ({B}ergen)}, volume 1851 of
  {\em Lecture Notes in Comput. Sci.}, pages 490--503. Springer, Berlin, 2000.
\newblock \href {https://doi.org/10.1007/3-540-44985-X_41}
  {\path{doi:10.1007/3-540-44985-X_41}}.

\bibitem{ab:grid-basis}
Michael Albert and Robert Brignall.
\newblock $2\times 2$ monotone grid classes are finitely based.
\newblock {\em Discrete Math. Theor. Comput. Sci.}, 18(2):~Paper No. 1, 10pp,
  2016.
\newblock \href {https://doi.org/10.46298/dmtcs.1325}
  {\path{doi:10.46298/dmtcs.1325}}.

\bibitem{abrv:321-subclasses:}
Michael Albert, Robert Brignall, Nik Ru\v{s}kuc, and Vincent Vatter.
\newblock Rationality for subclasses of 321-avoiding permutations.
\newblock {\em European J. Combin.}, 78:44--72, 2019.
\newblock \href {https://doi.org/10.1016/j.ejc.2019.01.001}
  {\path{doi:10.1016/j.ejc.2019.01.001}}.

\bibitem{albert:bevanstheorem}
Michael Albert and Vincent Vatter.
\newblock An elementary proof of {B}evan's theorem on the growth of grid
  classes of permutations.
\newblock {\em Proc. Edinburgh Math. Soc.}, 62(4):975--984, 2019.
\newblock \href {https://doi.org/10.1017/S0013091519000026}
  {\path{doi:10.1017/S0013091519000026}}.

\bibitem{albert:geometric-grid-:}
Michael~H. Albert, M.~D. Atkinson, Mathilde Bouvel, Nik Ru{\v{s}}kuc, and
  Vincent Vatter.
\newblock Geometric grid classes of permutations.
\newblock {\em Trans. Amer. Math. Soc.}, 365(11):5859--5881, 2013.
\newblock \href {https://doi.org/10.1090/S0002-9947-2013-05804-7}
  {\path{doi:10.1090/S0002-9947-2013-05804-7}}.

\bibitem{albert:generating-and-:}
Michael~H. Albert and Vincent Vatter.
\newblock Generating and enumerating 321-avoiding and skew-merged simple
  permutations.
\newblock {\em Electron. J. Combin.}, 20(2):~Paper 44, 11pp, 2013.
\newblock \href {https://doi.org/10.37236/3058} {\path{doi:10.37236/3058}}.

\bibitem{alecu:letter-graphs-and-ggcs:}
Bogdan Alecu, Robert Ferguson, Mamadou~Moustapha Kant\'{e}, Vadim~V. Lozin,
  Vincent Vatter, and Victor Zamaraev.
\newblock Letter graphs and geometric grid classes of permutations.
\newblock {\em SIAM Journal on Discrete Mathematics}, 36(4):2774--2797, 2022.
\newblock \href {https://doi.org/10.1137/21M1449646}
  {\path{doi:10.1137/21M1449646}}.

\bibitem{atminas:classes-of-graphs:}
Aistis Atminas.
\newblock Classes of graphs without star forests and related graphs.
\newblock {\em Discrete Math.}, 345(12):Paper No. 113089, 22, 2022.
\newblock \href {https://doi.org/10.1016/j.disc.2022.113089}
  {\path{doi:10.1016/j.disc.2022.113089}}.

\bibitem{balogh:a-jump-to-the-bell:}
J{\'o}zsef Balogh, B{\'e}la Bollob{\'a}s, and David Weinreich.
\newblock A jump to the {B}ell number for hereditary graph properties.
\newblock {\em J. Combin. Theory Ser. B}, 95(1):29--48, 2005.
\newblock \href {https://doi.org/10.1016/j.jctb.2005.02.004}
  {\path{doi:10.1016/j.jctb.2005.02.004}}.

\bibitem{BFKRT:18}
Laurent Beaudou, Kaoutar Ghazi, Giacomo Kahn, Olivier Raynaud, and Eric
  Thierry.
\newblock Encoding partial orders through modular decomposition.
\newblock {\em J. Comput. Sci.}, 25:446--455, 2018.
\newblock \href {https://doi.org/10.1016/j.jocs.2017.05.008}
  {\path{doi:10.1016/j.jocs.2017.05.008}}.

\bibitem{bevan:growth-rates-ggc}
David Bevan.
\newblock Growth rates of geometric grid classes of permutations.
\newblock {\em Electron. J. Combin.}, 21(4):~Paper 51, 17pp, 2014.
\newblock \href {https://doi.org/10.37236/4834} {\path{doi:10.37236/4834}}.

\bibitem{bevan:growth-rates:}
David Bevan.
\newblock Growth rates of permutation grid classes, tours on graphs, and the
  spectral radius.
\newblock {\em Trans. Amer. Math. Soc.}, 367(8):5863--5889, 2015.
\newblock \href {https://doi.org/10.1090/S0002-9947-2015-06280-1}
  {\path{doi:10.1090/S0002-9947-2015-06280-1}}.

\bibitem{bevan:thesis:}
David Bevan.
\newblock {\em On the growth of permutation classes}.
\newblock PhD thesis, The Open University, 2015.
\newblock \href {https://doi.org/10.21954/ou.ro.0000ab63}
  {\path{doi:10.21954/ou.ro.0000ab63}}.

\bibitem{bevan2015defs}
David Bevan.
\newblock Permutation patterns: basic definitions and notation.
\newblock arXiv:1506.06673, 2015.
\newblock URL: \url{http://arxiv.org/abs/1506.06673}.

\bibitem{BBS:99}
Andreas Brandst\"adt, Van~Bang Le, and Jeremy~P. Spinrad.
\newblock {\em Graph classes: a survey}.
\newblock SIAM Monographs on Discrete Mathematics and Applications. Society for
  Industrial and Applied Mathematics (SIAM), Philadelphia, PA, 1999.
\newblock \href {https://doi.org/10.1137/1.9780898719796}
  {\path{doi:10.1137/1.9780898719796}}.

\bibitem{brignall:pwo-grid-classes:}
Robert Brignall.
\newblock Grid classes and partial well order.
\newblock {\em J. Combin. Theory Ser. A}, 119(1):99--116, 2012.
\newblock \href {https://doi.org/10.1016/j.jcta.2011.08.005}
  {\path{doi:10.1016/j.jcta.2011.08.005}}.

\bibitem{brignall:lwqo-juxt:}
Robert Brignall.
\newblock Labelled well-quasi-order in juxtapositions of permutation classes.
\newblock {\em Electron. J. Combin.}, 31(2):~Paper 21, 11pp, 2024.
\newblock \href {https://doi.org/10.37236/12655} {\path{doi:10.37236/12655}}.

\bibitem{bev:lwqo}
Robert Brignall, Michael Engen, and Vincent Vatter.
\newblock A counterexample regarding labelled well-quasi-ordering.
\newblock {\em Graphs Combin.}, 34(6):1395--1409, 2018.
\newblock \href {https://doi.org/10.1007/s00373-018-1962-0}
  {\path{doi:10.1007/s00373-018-1962-0}}.

\bibitem{BGW:12}
Robert Brignall, Nicholas Georgiou, and Robert~J. Waters.
\newblock Modular decomposition and the reconstruction conjecture.
\newblock {\em J. Comb.}, 3(1):123--134, 2012.
\newblock \href {https://doi.org/10.4310/JOC.2012.v3.n1.a6}
  {\path{doi:10.4310/JOC.2012.v3.n1.a6}}.

\bibitem{bv:lwqo-for-pp:}
Robert Brignall and Vincent Vatter.
\newblock Labelled well-quasi-order for permutation classes.
\newblock {\em Comb. Theory}, 2(3):~Paper No. 14, 54pp, 2022.
\newblock \href {https://doi.org/10.5070/c62359178}
  {\path{doi:10.5070/c62359178}}.

\bibitem{cherlin:forbidden-subst:}
Gregory~L. Cherlin.
\newblock Forbidden substructures and combinatorial dichotomies: {WQO} and
  universality.
\newblock {\em Discrete Math.}, 311(15):1543--1584, 2011.
\newblock \href {https://doi.org/10.1016/j.disc.2011.03.014}
  {\path{doi:10.1016/j.disc.2011.03.014}}.

\bibitem{FV:22}
Robert Ferguson and Vincent Vatter.
\newblock Letter graphs and modular decomposition.
\newblock {\em Discrete Appl. Math.}, 309:215--220, 2022.
\newblock \href {https://doi.org/10.1016/j.dam.2021.11.007}
  {\path{doi:10.1016/j.dam.2021.11.007}}.

\bibitem{Fr:53}
R.~Fra{\"\i}ss\'e.
\newblock On a decomposition of relations which generalizes the sum of ordering
  relations.
\newblock {\em Bull. Amer. Math. Soc.}, 59:389, 1953.

\bibitem{Ga:67}
T.~Gallai.
\newblock Transitiv orientierbare {G}raphen.
\newblock {\em Acta Math. Acad. Sci. Hungar.}, 18:25--66, 1967.
\newblock \href {https://doi.org/10.1007/BF02020961}
  {\path{doi:10.1007/BF02020961}}.

\bibitem{Ga:trans01}
Tibor Gallai.
\newblock A translation of {T}. {G}allai's paper: ``{T}ransitiv orientierbare
  {G}raphen'' [{A}cta {M}ath. {A}cad. {S}ci. {H}ungar. {\bf 18} (1967), 25--66;
  {MR}0221974 (36 \#5026)].
\newblock In {\em Perfect graphs}, Wiley-Intersci. Ser. Discrete Math. Optim.,
  pages 25--66. Wiley, Chichester, 2001.
\newblock Translated from the German and with a foreword by Fr\'ed\'eric
  Maffray and Myriam Preissmann.

\bibitem{gyarfas:problems-from:}
A.~Gy\'arf\'as.
\newblock Problems from the world surrounding perfect graphs.
\newblock In {\em Proceedings of the {I}nternational {C}onference on
  {C}ombinatorial {A}nalysis and its {A}pplications ({P}okrzywna, 1985)},
  volume~19, pages 413--441, 1987.
\newblock URL: \url{https://bibliotekanauki.pl/articles/740597}.

\bibitem{HMMZ:22}
Michel Habib, Fabien de~Montgolfier, Lalla Mouatadid, and Mengchuan Zou.
\newblock A general algorithmic scheme for combinatorial decompositions with
  application to modular decompositions of hypergraphs.
\newblock {\em Theoret. Comput. Sci.}, 923:56--73, 2022.
\newblock \href {https://doi.org/10.1016/j.tcs.2022.04.052}
  {\path{doi:10.1016/j.tcs.2022.04.052}}.

\bibitem{higman:ordering-by-div:}
Graham Higman.
\newblock Ordering by divisibility in abstract algebras.
\newblock {\em Proc. London Math. Soc. (3)}, 2:326--336, 1952.
\newblock \href {https://doi.org/10.1112/plms/s3-2.1.326}
  {\path{doi:10.1112/plms/s3-2.1.326}}.

\bibitem{Huczynska2006}
Sophie Huczynska and Vincent Vatter.
\newblock Grid classes and the {F}ibonacci dichotomy for restricted
  permutations.
\newblock {\em Electron. J. Combin.}, 13:~Paper No. 54, 14pp, 2006.
\newblock \href {https://doi.org/10.37236/1080} {\path{doi:10.37236/1080}}.

\bibitem{jelinek:a-complexity-dichotomy:}
V\'it Jel\'inek, Michal Opler, and Jakub Pek\'arek.
\newblock A complexity dichotomy for permutation pattern matching on grid
  classes.
\newblock In {\em 45th {I}nternational {S}ymposium on {M}athematical
  {F}oundations of {C}omputer {S}cience}, volume 170 of {\em LIPIcs. Leibniz
  Int. Proc. Inform.}, pages Art. No. 52, 18. Schloss Dagstuhl. Leibniz-Zent.
  Inform., Wadern, 2020.
\newblock \href {https://doi.org/10.4230/LIPIcs.MFCS.2020.52}
  {\path{doi:10.4230/LIPIcs.MFCS.2020.52}}.

\bibitem{jelinek:long-paths:}
V\'it Jel\'inek, Michal Opler, and Jakub Pek\'arek.
\newblock Long paths make pattern-counting hard, and deep trees make it harder.
\newblock In {\em 16th {I}nternational {S}ymposium on {P}arameterized and
  {E}xact {C}omputation}, volume 214 of {\em LIPIcs. Leibniz Int. Proc.
  Inform.}, pages Art. No. 22, 17. Schloss Dagstuhl. Leibniz-Zent. Inform.,
  Wadern, 2021.
\newblock \href {https://doi.org/10.4230/LIPIcs.IPEC.2021.22}
  {\path{doi:10.4230/LIPIcs.IPEC.2021.22}}.

\bibitem{kaiser:on-growth-rates:}
Tom{\'a}{\v{s}} Kaiser and Martin Klazar.
\newblock On growth rates of closed permutation classes.
\newblock {\em Electron. J. Combin.}, 9(2):~Paper No. 10, 20pp, 2003.
\newblock \href {https://doi.org/10.37236/1682} {\path{doi:10.37236/1682}}.

\bibitem{kleitman:the-asymptotic:}
D.~J. Kleitman and K.~J. Winston.
\newblock The asymptotic number of lattices.
\newblock {\em Ann. Discrete Math.}, 6:243--249, 1980.
\newblock \href {https://doi.org/10.1016/S0167-5060(08)70708-8}
  {\path{doi:10.1016/S0167-5060(08)70708-8}}.

\bibitem{murphy:profile-classes:}
Maximillian~M. Murphy and Vincent Vatter.
\newblock Profile classes and partial well-order for permutations.
\newblock {\em Electron. J. Combin.}, 9(2):~Paper No. 17, 30pp, 2003.
\newblock \href {https://doi.org/10.37236/1689} {\path{doi:10.37236/1689}}.

\bibitem{opler:phd}
Michal Opler.
\newblock {\em Structural and Algorithmic Properties of Permutation Classes}.
\newblock PhD thesis, Charles University, 2022.
\newblock \href {https://doi.org/20.500.11956/179844}
  {\path{doi:20.500.11956/179844}}.

\bibitem{petkovsek:letter-graphs-a:}
Marko Petkov{\v{s}}ek.
\newblock Letter graphs and well-quasi-order by induced subgraphs.
\newblock {\em Discrete Math.}, 244(1-3):375--388, 2002.
\newblock Algebraic and topological methods in graph theory (Lake Bled, 1999).
\newblock \href {https://doi.org/10.1016/S0012-365X(01)00094-2}
  {\path{doi:10.1016/S0012-365X(01)00094-2}}.

\bibitem{scheinerman:on-the-size:}
Edward~R. Scheinerman and Jennifer Zito.
\newblock On the size of hereditary classes of graphs.
\newblock {\em J. Combin. Theory Ser. B}, 61(1):16--39, 1994.
\newblock \href {https://doi.org/10.1006/jctb.1994.1027}
  {\path{doi:10.1006/jctb.1994.1027}}.

\bibitem{szemeredi:regularity-lemma:}
Endre Szemer\'edi.
\newblock Regular partitions of graphs.
\newblock In {\em Probl\`emes combinatoires et th\'eorie des graphes ({C}olloq.
  {I}nternat. {CNRS}, {U}niv. {O}rsay, {O}rsay, 1976)}, volume 260 of {\em
  Colloq. Internat. CNRS}, pages 399--401. CNRS, Paris, 1978.

\bibitem{vatter:small-permutati:}
Vincent Vatter.
\newblock Small permutation classes.
\newblock {\em Proc. Lond. Math. Soc. (3)}, 103:879--921, 2011.
\newblock \href {https://doi.org/10.1112/plms/pdr017}
  {\path{doi:10.1112/plms/pdr017}}.

\bibitem{vatter:survey}
Vincent Vatter.
\newblock Permutation classes.
\newblock In {\em Handbook of enumerative combinatorics}, Discrete Math. Appl.
  (Boca Raton), pages 753--833. CRC Press, Boca Raton, FL, 2015.
\newblock \href {https://doi.org/10.1201/b18255} {\path{doi:10.1201/b18255}}.

\bibitem{vatter:growth-rates-of:}
Vincent Vatter.
\newblock Growth rates of permutation classes: from countable to uncountable.
\newblock {\em Proc. Lond. Math. Soc. (3)}, 119(4):960--997, 2019.
\newblock \href {https://doi.org/10.1112/plms.12250}
  {\path{doi:10.1112/plms.12250}}.

\bibitem{vatter:on-partial-well:}
Vincent Vatter and Steve Waton.
\newblock On partial well-order for monotone grid classes of permutations.
\newblock {\em Order}, 28:193--199, 2011.
\newblock \href {https://doi.org/10.1007/s11083-010-9165-1}
  {\path{doi:10.1007/s11083-010-9165-1}}.

\bibitem{waton:on-permutation-:}
Steve Waton.
\newblock {\em On Permutation Classes Defined by Token Passing Networks,
  Gridding Matrices and Pictures: Three Flavours of Involvement}.
\newblock PhD thesis, Univ. of St Andrews, 2007.
\newblock \href {https://doi.org/10023/237} {\path{doi:10023/237}}.

\end{thebibliography}

\end{document}
%
%
%
%
%
%
%
%
%
%
%
%
%
%
%
%
%
%
%
%
%
%
%
%
%
%
%
%
%
%
%
%
%
%
%
%
%
%
%
%
%
%
%
%
%
%

%
%
%
%
\section{Orphan from WQO section}
\subsection{Unicyclic classes}

The question of well quasi-orderability in unicyclic grid classes is more nuanced. While we know that every infinite antichain necessarily must include elements that contain arbitrarily long coils, there are multiples ways for these to be `anchored', and the work to catalogue them all would take too long in this current paper. Thus, we will merely highlight some examples of the types of issue that can arise.

First, consider the matrix $M=\gctwo{4}{1,0,-1,1}{0,-1,1,-1}$. The cell graph $G_M$ comprises a cycle with two leaves off adjacent vertices on the cycle, and is thus unicyclic. The class $\Grid(M)$ contains an infinite antichain, three elements of which are shown in Figure~\ref{fig-tied-widdershins}, and the general construction comprises a coil in which the first and last entries are `tied' by a single entry. This antichain first appears in Murphy's PhD~\cite{murphy:restricted-perm:}. Note, in particular, that the two cells that are not part of the cycle each only need to contain a single entry for the construction of the antichain.

\begin{figure}
{\centering
\begin{tikzpicture}[scale=0.25]
\plotpermgrid{7,3,5,1,6,4,8,2}
\draw[thick] (1.5,0.5) -- (1.5,8.5);
\draw[thick] (2.5,0.5) -- (2.5,8.5);
\draw[thick] (4.5,0.5) -- (4.5,8.5);
\draw[thick] (0.5,4.5) -- (8.5,4.5);
\draw[gray] (4)--(6)--(5)--(1)--(2)--(8)--(7)--(3)--(4);

\begin{scope}[shift={(10,0)}]
\plotpermgrid{11,5,9,1,7,3,8,6,10,4,12,2}	
\draw[gray] (6)--(8)--(7)--(3)--(4)--(10)--(9)--(1)--(2)--(12)--(11)--(5)--(6);
\draw[thick] (1.5,0.5) -- (1.5,12.5);
\draw[thick] (2.5,0.5) -- (2.5,12.5);
\draw[thick] (6.5,0.5) -- (6.5,12.5);
\draw[thick] (0.5,6.5) -- (12.5,6.5);
\end{scope}
\begin{scope}[shift={(24,0)}]
\plotpermgrid{15,7,13,1,11,3,9,5,10,8,12,6,14,4,16,2}	
\draw[gray] (8)--(10)--(9)--(5)--(6)--(12)--(11)--(3)--(4)--(14)--(13)--(1)--(2)--(16)--(15)--(7)--(8);
\draw[thick] (1.5,0.5) -- (1.5,16.5);
\draw[thick] (2.5,0.5) -- (2.5,16.5);
\draw[thick] (8.5,0.5) -- (8.5,16.5);
\draw[thick] (0.5,8.5) -- (16.5,8.5);
\end{scope}
\end{tikzpicture} \par}
\caption{Three elements from the `tied Widdershins' antichain, which is contained in $\grid{\gctwo{4}{1,0,-1,1}{0,-1,1,-1}}$.}\label{fig-tied-widdershins}
\end{figure}

For our second example, take the matrix $N=\gctwo{4}{0,-1,1,-1}{-1,1,-1,0}$. The first three elements of an antichain in $\Grid(N)$ are shown in Figure~\ref{fig-untied-widdershins}. This antichain can be obtained from  by moving the first entry of each permutation in the tied Widdershins antichain to the other end. To make matters worse, this example can also be shown to belong to the class
\[
	\grid{\gcsix{6}{0,0,0,0,1,0}{0,0,0,0,0,1}{0,-1,0,1,0,0}{0,0,-1,1,0,0}{-1,0,1,-1,0,0}{0,1,0,0,-1,0}}.
\]
This class has two complications: first, while there is a single cycle, it is possible for a coil to be gridded with several points in the `tendril'. Second, the `anchor' on the right of the permutation is placed in a cell that corresponds to an edge in a completely different component of the graph of the matrix.

Our third example is an antichain in $\grid{\gctwo{4}{1,-1,1,-1}{0,1,-1,0}}$, again using just one entry in the two cells that are not on the cycle:

{\centering
\begin{tikzpicture}[scale=0.2]
\plotpermgrid{4,3,1,6,2,5}
\draw[thick] (1.5,0.5) -- (1.5,6.5);
\draw[thick] (5.5,0.5) -- (5.5,6.5);
\draw[thick] (3.5,0.5) -- (3.5,6.5);
\draw[thick] (0.5,2.5) -- (6.5,2.5);
\draw[gray] (4)--(3)--(1)--(2)--(6)--(5);
\begin{scope}[shift={(10,0)}]
\plotpermgrid{6,7,1,5,3,8,4,10,2,9}	
\draw[gray] (6)--(5)--(3)--(4)--(8)--(7)--(1)--(2)--(10)--(9);
\draw[thick] (1.5,0.5) -- (1.5,10.5);
\draw[thick] (9.5,0.5) -- (9.5,10.5);
\draw[thick] (5.5,0.5) -- (5.5,10.5);
\draw[thick] (0.5,4.5) -- (10.5,4.5);
\end{scope}
\begin{scope}[shift={(24,0)}]
\plotpermgrid{8,11,1,9,3,7,5,10,6,12,4,14,2,13}	
\draw[gray] (8)--(7)--(5)--(6)--(10)--(9)--(3)--(4)--(12)--(11)--(1)--(2)--(14)--(13);
\draw[thick] (1.5,0.5) -- (1.5,14.5);
\draw[thick] (13.5,0.5) -- (13.5,14.5);
\draw[thick] (7.5,0.5) -- (7.5,14.5);
\draw[thick] (0.5,6.5) -- (14.5,6.5);
\end{scope}
\end{tikzpicture} \par}

At this point, it is reasonable to suspect that one can find an antichain in \emph{any} arrangement with two extra cells added to any cycle, and moreover that there is an antichain each of whose elements have precisely one entry in each of the extra cells. The consequences of this would be that a wqo subclass of a unicyclic grid either contains bounded length coils, or has arbitrarily long coils but the `extra' cells can only interact with a bounded central portion.

Our last two examples relate to cycles with a single additional cycle. First, in $\grid{\gctwo{3}{-1,-1,1}{0,1,-1}}$ we have the following antichain (related to the previous one by moving the last entry to the beginning):

{\centering
\begin{tikzpicture}[scale=0.2]
\plotpermgrid{5,4,3,1,6,2}
\draw[thick] (2.5,0.5) -- (2.5,6.5);
\draw[thick] (4.5,0.5) -- (4.5,6.5);
\draw[thick] (0.5,2.5) -- (6.5,2.5);
\draw[gray] (4)--(3)--(1)--(2)--(6)--(5);
\begin{scope}[shift={(10,0)}]
\plotpermgrid{9,6,7,1,5,3,8,4,10,2}	
\draw[gray] (6)--(5)--(3)--(4)--(8)--(7)--(1)--(2)--(10)--(9);
\draw[thick] (2.5,0.5) -- (2.5,10.5);
\draw[thick] (6.5,0.5) -- (6.5,10.5);
\draw[thick] (0.5,4.5) -- (10.5,4.5);
\end{scope}
\begin{scope}[shift={(24,0)}]
\plotpermgrid{13,8,11,1,9,3,7,5,10,6,12,4,14,2}	
\draw[gray] (8)--(7)--(5)--(6)--(10)--(9)--(3)--(4)--(12)--(11)--(1)--(2)--(14)--(13);
\draw[thick] (2.5,0.5) -- (2.5,14.5);
\draw[thick] (8.5,0.5) -- (8.5,14.5);
\draw[thick] (0.5,6.5) -- (14.5,6.5);
\end{scope}
\end{tikzpicture} \par}

This uses two entries in the non-cycle cell; furthermore, it is important that they form a 21 pattern, because if the first two entries appeared in the opposite order, then the higher of the two entries could be gridded in the cycle, and would thus just be a continuation of the coil. However, a minor modification to this example provides an antichain in  $\grid{\gctwo{3}{1,-1,1}{0,1,-1}}$, which again only requires two entries in the cell:

{\centering
\begin{tikzpicture}[scale=0.2]
\plotpermgrid{4,5,6,3,1,7,2}
\draw[thick] (2.5,0.5) -- (2.5,7.5);
\draw[thick] (4.5,0.5) -- (4.5,7.5);
\draw[thick] (0.5,2.5) -- (7.5,2.5);
\draw[gray] (4)--(3)--(1)--(2)--(7)--($(5)!0.5!(6)$);
\begin{scope}[shift={(10,0)}]
\plotpermgrid{6,9,10,7,1,5,3,8,4,11,2}	
\draw[gray] (6)--(5)--(3)--(4)--(8)--(7)--(1)--(2)--(11)--($(9)!0.5!(10)$);
\draw[thick] (2.5,0.5) -- (2.5,11.5);
\draw[thick] (6.5,0.5) -- (6.5,11.5);
\draw[thick] (0.5,4.5) -- (11.5,4.5);
\end{scope}
\begin{scope}[shift={(24,0)}]
\plotpermgrid{8,13,14,11,1,9,3,7,5,10,6,12,4,15,2}	
\draw[gray] (8)--(7)--(5)--(6)--(10)--(9)--(3)--(4)--(12)--(11)--(1)--(2)--(15)--($(13)!0.5!(14)$);
\draw[thick] (2.5,0.5) -- (2.5,15.5);
\draw[thick] (8.5,0.5) -- (8.5,15.5);
\draw[thick] (0.5,6.5) -- (15.5,6.5);
\end{scope}
\end{tikzpicture} \par}

%
%
%
%
%
%
\section{Legacy material A}
\subsection{Characterising gridded containment}

\commentrb{These are the summary of some rough notes on various whiteboards.}

Let $\pi^\gridded\in\Gridhash(M)$, and let $v_1,\dots,v_n$ denote a set of points in $\pi^\gridded$ that form a maximal length gridded coil. We say that $v_i$ is \emph{special} if there exists $x\in D_{\pi^\gridded}$, distinct from $v_1,\dots,v_n$, such that
$x\rightarrow v_i$ and $x\not\rightarrow v_j$ for every $j>i$, or
$v_i\rightarrow x$ and $v_j\not\rightarrow x$ for $j<i$.

Take a gridded coil $v_1,\dots,v_n$ and suppose that $v_i$ is special with $x\rightarrow v_i$. The point $x$ must lie in the same row or column as $v_i$. Furthermore, since $v_i\rightarrow v_{i+1}$ but we cannot have $x\rightarrow v_{i+1}$, it follows that $x$ cannot share a row or column with $v_{i+1}$, which in particular means that $x$ is not in the same cell as $v_i$. We have two options for the position of $x$ (See Figure~\ref{fig-special points}):
\begin{enumerate}
	\item In the same cell as $v_{i-1}$ (which must exist);
	\item Not in a cell on the cycle, but in the row/column common to $v_{i-1}$ and $v_i$.
\end{enumerate}
Similar observations hold if $v_i\rightarrow x$, but in this case we can have $v_{i+1}$ and $x$ sharing a cell.

\begin{figure}
{\centering
\begin{tikzpicture}[scale=0.45]
\foreach \x/\y in {-7/6,-7/10,0/0,0/4,0/6,0/10,0/12,7/0,7/4}
	\draw[dotted] (\x,\y) -- ++(-1.5,0) (\x+4,\y) -- ++(1.5,0);
\foreach \x/\y in {-7/6,-3/6,0/0,4/0,7/0,11/0,0/6,4/6}
	\draw[dotted] (\x,\y) -- ++(0,-1) (\x,\y+4) -- ++(0,1);
	\draw[dotted] (0,12) -- ++(0,-1) (4,12) -- ++(0,-1);
\draw [draw=none,fill=black!20] (1.5,1.5) rectangle ++ (1,1.5) rectangle ++(0.5,1) (1.5,12) rectangle ++(1.5,2);
\draw[dashed] (1.5,-0.5) -- ++(0,14.5)
	(3,-0.5) -- ++(0,14.5)
	(-7.5,1.5) -- ++(19,0);
\draw[->] (-6,-1) -- ++(2,0);
\draw[->] (1,-1) -- ++(2,0);
\draw[->] (8,-1) -- ++(2,0);
\draw[->] (-8,1) -- ++(0,2);
\draw[->] (-8,7) -- ++(0,2);
\draw (-7,6) rectangle ++(4,4);
\draw (0,0) rectangle (4,4);
\draw (7,0) rectangle ++(4,4);
\draw (0,6) rectangle ++(4,4);
\draw (0,14) -- ++(0,-2) -- ++(4,0) -- ++(0,2);
\node[permpt,label={[label distance=-3pt]below right:\footnotesize$v_{i+\ell-2}$}] at (8.5,1) {};
\node[permpt,label={[label distance=-3pt]below right:\footnotesize$v_{i+\ell-1}$}] at (1,1.5) {};
\node[permpt,label={[label distance=-3pt]below right:\footnotesize$v_{i+\ell}$}] at (1.5,7) {};
\node[permpt,label={[label distance=-3pt]below right:\footnotesize$v_{i+\ell+1}$}] at (-6,7.5) {};
\node[permpt,label={[label distance=-3pt]above:\footnotesize$v_{i-2}$}] at (10,2.5) {};
\node[permpt,label={[label distance=0pt]right:\footnotesize$v_{i-1}$}] at (2.5,3) {};
\node[permpt,label={[label distance=-3pt]above right:\footnotesize$v_{i}$}] at (3,8.5) {};
\node[permpt,label={[label distance=-3pt]above:\footnotesize$v_{i+1}$}] at (-4.5,9) {};
\end{tikzpicture} \par}
\caption{The shaded regions represent the possible areas for placing a point $x$ such that $v_i$ is a special point with $x\rightarrow v_i$. Note that $x$ cannot lie in the same cell as $v_i$.}\label{fig-special points}	
\end{figure}

\begin{figure}
{\centering
\begin{tikzpicture}[scale=0.45]
\foreach \x/\y in {-7/6,-7/10,0/0,0/4,0/6,0/10,7/0,7/4}
	\draw[dotted] (\x,\y) -- ++(-1.5,0) (\x+4,\y) -- ++(1.5,0);
\foreach \x/\y in {-7/6,-3/6,0/0,4/0,7/0,11/0,0/6,4/6,14/6}
	\draw[dotted] (\x,\y) -- ++(0,-1) (\x,\y+4) -- ++(0,1);
	\draw[dotted] (12.5,6) -- ++(1.5,0) (12.5,10) -- ++(1.5,0);
\draw [draw=none,fill=black!20] (1.5,1.5) rectangle ++ (1,1.5) rectangle ++(0.5,1) (14,7) rectangle ++(2,1.5);
\draw[dashed] (-4.5,-0.5) -- ++(0,11)
	(-7.5,7) -- ++(23.5,0)
	(-7.5,8.5) -- ++(23.5,0);
\draw[->] (-6,-1) -- ++(2,0);
\draw[->] (1,-1) -- ++(2,0);
\draw[->] (8,-1) -- ++(2,0);
\draw[->] (-8,1) -- ++(0,2);
\draw[->] (-8,7) -- ++(0,2);
\draw (-7,6) rectangle ++(4,4);
\draw (0,0) rectangle (4,4);
\draw (7,0) rectangle ++(4,4);
\draw (0,6) rectangle ++(4,4);
\draw (16,6) -- ++(-2,0) -- ++(0,4) -- ++(2,0);
\node[permpt,label={[label distance=-3pt]below right:\footnotesize$v_{i-2}$}] at (8.5,1) {};
\node[permpt,label={[label distance=-3pt]below right:\footnotesize$v_{i-1}$}] at (1,1.5) {};
\node[permpt,label={[label distance=-3pt]below right:\footnotesize$v_{i}$}] at (1.5,7) {};
\node[permpt,label={[label distance=-3pt]below right:\footnotesize$v_{i+1}$}] at (-6,7.5) {};
\node[permpt,label={[label distance=-3pt]above:\footnotesize$v_{i-\ell-2}$}] at (10,2.5) {};
\node[permpt,label={[label distance=0pt]right:\footnotesize$v_{i-\ell-1}$}] at (2.5,3) {};
\node[permpt,label={[label distance=-3pt]above right:\footnotesize$v_{i-\ell}$}] at (3,8.5) {};
\node[permpt,label={[label distance=-3pt]above:\footnotesize$v_{i-\ell+1}$}] at (-4.5,9) {};
\end{tikzpicture} \par}
\caption{The shaded regions represent the possible areas for placing a point $x$ such that $v_i$ is a special point with $v_i\rightarrow x$. Note that $x$ cannot lie in the same cell as $v_i$.}\label{fig-special points-2}	
\end{figure}
Case 1 is somewhat akin to the usual "anchoring" (e.g. of the Widdershins or increasing oscillating antichain): we double up a point on the coil. Somewhat curiously, the point we double up is $v_{i-1}$ if $x\to v_i$, or $v_{i+1}$ if $v_i\to x$.

Case 2 essentially makes use of the gridding (which can be thought of as a labelling) to anchor entries. However, this is the mechanism for the `tied-by-one' antichains.

A \emph{decorated gridded coil of length $n$} is a permutation of length $n+2$ which comprises a gridded coil $v_1,\dots,v_n$ in which $v_1$ and $v_n$ are special points, together with two associated points.

\begin{lemma}\label{lem-gridded-unlabelled-antichains}
Let $M$ be any partial multiplication matrix, and let $\sigma^\gridded$ and $\pi^\gridded$ be decorated gridded $M$-coils of lengths $m$ and $n$ with $m<n$. Then $\sigma^\gridded$ does not embed in $\pi^\gridded$.
\end{lemma}

\begin{proof}
TODO.	
\end{proof}

\begin{prop}
Let $\C^\gridded\subseteq \Gridhash(M)$, where $M$ is a partial multiplication matrix in which every component is unicyclic or a tree. Then $\C^\gridded$ is well quasi-ordered with respect to gridded containment if and only if $\C^\gridded$ contains finitely many decorated gridded coils.
\end{prop}

\begin{proof}
\noterb{This direction is surely not as easy as written.}One direction is immediate: by Lemma~\ref{lem-gridded-unlabelled-antichains}, if there is no bound on the number of decorated gridded coils, then we may construct an infinite antichain comprising this infinite collection of gridded coils.

For the other direction, let $k$ denote the length of the longest decorated gridded coil in $\C^\gridded$. We want to do something similar to \cite[Theorem 9.3]{abrv:321-subclasses:}, where we encode the extra sections of coils before the first and after the last decoration using an integer as a label on the endpoints of the coil. Classes containing bounded length coils are lwqo, and this is essentially a labelling by the wqo set $\mathbb{N}\cup\{\bullet\}$.
\end{proof}

%
%
%
%
%
\subsection{Ungridded well quasi-ordering}

Start with a gridded antichain element $\pi^\gridded$ with special points. In Cases 1 or 2 of the above section, we are reasonably `happy': we should be able to argument that the presence of such special points translates to the ungridded situation.

In case 3, suppose $x\rightarrow v_1$. We have two distinct scenarios:

A: The coil cannot be extended to include $x$, for example by regridding $\pi^\gridded$. This means that $x$ is placed in $\pi^\gridded$ in such a way that the other entries of $\pi^\gridded$ prevent it from going in the right place.

B: The coil can be extended to include $x$, if we regrid $\pi^\gridded$.

We are, again, quite `happy' in subcase A. For subcase B, we could regrid to include $x$ in the cycle, and now we need to explore the consequences of this regridding. Here's an example of the sort of thing we need to worry about.

Let $M=\gcfive{5}{0,0,0,1,1}{-1,0,1,0,0}{0,-1,1,0,0}{0,1,-1,0,0}{1,0,0,-1,0}$. Notice that we've added a `tendril' that basically enables us to grid several more coil points. Here is a coil whose 5th point is placed in the `wrong' cell (i.e.\ case 3), but where the coil continues some distance further before we encounter a point (the rightmost point) that we can use to anchor at the end:

\begin{tikzpicture}[scale=0.3]
\plotpermgrid{28,1,26,3,24,5,22,7,20,9,18,11,16,13,17,15,19,14,21,12,23,10,25,8,27,6,29,4,31,2,30}
\draw (15) -- (17) -- (16) -- (13) -- (14) -- (19);
\foreach \x in {9.5,14.5,25.5,30.5}
	\draw (\x,0.5) -- ++(0,31);
\foreach \y in {7.5,15.5,18.5,28.5}
	\draw (0.5,\y) -- ++(31,0);
\end{tikzpicture}

With this in mind, we need to  `chase' points around a possible coil: we can appeal to Lemma~\ref{lem-coil-gridding-off-cycle} to argue that there can be at most $2\ell$ such points in each non-cycle cell, and thus there is a limit on how much further we might need to look before we can find a way to anchor (or conclude that no anchor exists).

In summary, we expect the ungridded antichains to arise from special points of types 1 or 2, or of type 3 where we may need to follow the coil for some bounded additional distance.

%
%
%
%
%
%
\section{Cyclic classes}

{\color{red}TODO: Describe the (unlabelled) \emph{coil antichain} that we get from a coil by blowing up the first and last points of finite prefixes. Notation: $v_1^{[2]},v_2,\dots,v_n^{[2]}$. We should observe these are the antichains that appear in Murphy and Vatter~\cite{murphy:profile-classes:}.}

An \emph{opening fork} is a gridded permutation given by an initial segment $v_1^{[2]},v_2,\dots,v_\ell$ of a coil antichain. Similarly, a \emph{closing fork} is a gridded permutation given by an inflated coil $v_1,v_2,\dots,v_\ell^{[2]}$.

\begin{lemma}\label{lem-inflated-coil-embeddings}
	Let $v_1,v_2,\dots,v_n$ be a gridded coil in $\Grid(M)$ for some cyclic oriented grid $M$. Then for $m<n$, $v_1^{[2]},v_2,\dots,v_m^{[2]}$ does not embed in $v_1^{[2]},v_2,\dots,v_n^{[2]}$ as a gridded permutation.
\end{lemma}

\begin{proof}
Label the two entries corresponding to $v_1^{[2]}$ as $\{a_1,a_2\}$, and the two entries of $v_m^{[m]}$ as $\{z_1,z_2\}$. Note that $\xi_1^\gridded = a_1,v_2,\dots,v_{m-1},z_1$ and $\xi_2^\gridded=a_2,v_2,\dots,v_{m-1},z_2$ are both gridded coils. By Lemma~\ref{lem-coil-embeddings}, any embedding $\phi$ of $\xi_1^\gridded$ into the gridded coil $v_1,\dots,v_n$ satisfies $\phi(v_i) = v_{j+i}$ for some $j$. Now consider jointly embedding $\xi_2^\gridded$: since $\phi(v_2,\dots,v_{m-1})$ is fixed, we conclude that $\phi(a_1)=\phi(a_2)$ and $\phi(z_1)=\phi(z_2)$,\noterb{This isn't quite right: depends on $j$} both of which contradict the fact that $\phi$ was taken to be an embedding.

Thus, we conclude that the only way to embed an inflated coil entry is into another inflated coil entry. Since embeddings of gridded coils into longer coils must be contiguous by Lemma~\ref{lem-coil-embeddings}, the statement of the lemma follows.
\end{proof}

Clearly, any class that has infinite intersection with a coil is not wqo.

\begin{prop}
	Any class $\C$ that contains infinite coil antichains is not wqo. \commentrb{This wording/language is a bit stupid!}
\end{prop}	

 We claim the following:

\begin{thm}Any subclass of a cyclic grid class $\Grid(M)$ that has only finite intersection with every coil in $\Grid(M)$ is well quasi-ordered.\end{thm}

\begin{proof}
Let $k+2$ denote the number of entries in the longest gridded coil antichain element. Now take any $M$-indivisible $\pi^\gridded$, and, following the proof of Lemma~\ref{lem-griddable-acyclic}, label the cells of the cycle $c_1,c_2,\dots,c_\ell$, and denote the last entries of $\pi^\gridded$ in these cells by $z_1,\dots,z_\ell$, labelled in such a way that $z_{i-1}\rightarrow z_{i}$ for all $i=2,\dots,\ell$, and $z_\ell\rightarrow z_1$. We now iteratively construct the following sets of points:\noterb{This is a repeat of Lemma~\ref{lem-griddable-acyclic} proof at this stage.}

\begin{itemize}
\item Set $S_1 = \{z_1\}$, $R_1 = \{ p\in\pi^\gridded:p\neq z_1\}$, and $p_1=z_1$.

\item For $i>1$:
Let $S_i=\{p\in R_{i-1}: p\text{ is in cell }j\equiv i \pmod{\ell}\text{ and }p_{i-1}\rightarrow p\}$, $R_i = R_{i-1}\setminus S_i$, and set $p_i$ to be the smallest entry in $S_i$ (with respect to the orientation).

\item When $S_{i+1}=\emptyset$, stop. Note that we will also have $R_i=\emptyset$, otherwise there are no arrows pointing from $\cup_{j\leq i}S_j$ to an entry of $R_i$, contradicting the requirement that $D_{\pi^\gridded}$ be strongly connected.
\end{itemize}

Let $m$ denote the index of the final non-empty set $S_i$, noting that, unlike Lemma~\ref{lem-griddable-acyclic}, we no longer have a bound on $m$. With the coil $p_1,p_2,\dots,p_m$ and the partition $S_1,\dots,S_m$ so constructed, now (for convenience) set $s_i=|S_i|$. We are going to identify potential starting places for opening  forks, based primarily (but not exclusively) on the sizes of the $s_i$.
\begin{enumerate}[(1)]
\item\label{forks-1} If for $\ell <i<m-\ell$ we have $s_i\geq 2$, then pick $q_i\in S_i$ distinct from $p_i$. If $q_i\rightarrow p_{i+1}$ then $\{q_i,p_i\},p_{i+1},\dots,p_{i+\ell-1}$ is an opening fork.

Otherwise, we have $p_{i+1}\rightarrow q_i$, and in this case observe that $\{p_{i-\ell},q_i\},p_{i-\ell+1},\dots, p_{i-1}$ is an opening fork.
\item\label{forks-2} For $i<\ell$, if $s_i\geq 2$ and $z_i\rightarrow p_{i+1}$, then $p_{i+1}=z_{i+1}$ so $z_i\rightarrow p_{i+1}$, and \[\{z_i,p_i\},p_{i+1},\dots,p_{i+\ell-1}\] is an opening fork.

Note that this case must happen if $s_i\geq 2$ and $s_{i+1}=1$ (since then $p_{i+1}=z_{i+1}$).
\item\label{forks-3} For $i<\ell$, if $s_i\geq 3$ then to avoid the case above we have $s_j\geq 2$ for all $i\leq j\leq \ell$. In particular, $z_j$ and $p_j$ are distinct in the range $i\leq j\leq \ell$.

In this case, take $q_i\in S_i$ distinct from $p_i$ and $z_i$. If $q_i\rightarrow p_{i+1}$ then $\{p_i,q_i\},p_{i+1},\dots,p_{i+\ell-1}$ is an opening fork.

Thus we may suppose that $p_{i+1}\rightarrow q_i$, and, by transitivity, this also gives $p_{i+1}\rightarrow z_i$. Now observe that $\{q_i,z_i\}\rightarrow z_{i+1}$, and hence $\{q_i,z_i\}, z_{i+1},\dots,z_\ell, p_1,p_2,\dots,p_{i-1}$ is an opening fork, noting in particular that it can be extended to a coil using $p_i,p_{i+1},\dots$ since $z_j$ is distinct from $p_j$ for all $j$ satisfying $i\leq j\leq \ell$.
\item\label{forks-4} For $i=\ell$, if $s_i\geq 3$ then take $q_\ell$ distinct from $p_\ell$ and $z_\ell$. The sequence \[\{q_\ell,z_\ell\},p_1,p_2,\dots,p_{\ell-1}\] is an opening fork.
\end{enumerate}
While the above list of opening forks may not be exhaustive, it places strong restrictions on the potential values for $s_1,\dots,s_{i-1}$ prior to encountering the earliest fork. Let $i_o$ now denote the index of the first set $S_{i_o}$ which contains one (or more) entries from an opening fork. (If no such $i$ exists, set $i_o=\infty$.)

If $i_o>\ell$, then by~\ref{forks-1} we see that $s_i=1$ for all $\ell<i<i_o$. Furthermore, by~\ref{forks-2} and~\ref{forks-3} if $s_i\geq2$ for some $1\leq i\leq \ell$, then we have $s_i=s_{i+1}=\cdots=s_\ell = 2$. In this case, we have the following potential sequences for $s_1,s_2,\dots,s_{i_o-1}$:
\begin{align*}
1,\dots,1,1,\dots,&1,1,\dots,1\\
1,\dots,1,2,\dots,&\underset{\underset{\ell}{\uparrow}}{2},1,\dots,1.
\end{align*}
We note that both of these are coils (and thus contain no opening forks, as expected). This is clear if the sequence $s_1,s_2,\dots,s_{i_o-1}$ comprises all 1s (since then $S_i =\{p_i\}$), while for the second sequence, note that the sets $S_i$ comprising two elements are among the first $\ell$ entries, and thus contain the entries $z_i$ and $p_i$. Furthermore, by \ref{forks-2} we must have $p_{i+1}\rightarrow z_i$. Thus, if $j$ represents the index of the first $S_j$ to contain two entries, the sequence $z_j,z_{j+1},\dots,z_\ell,p_1,\dots$ is a coil.

If $i_o\leq \ell$, then to avoid a case analysis we simply set $i_o=2$, and note that the true value of $i_o$ is no greater than $\ell$. (Note: it is possible that the first set $S_1$ contains an entry of an opening fork, but we still set $i_o=2$ as $S_1$ will be handled by the analysis in a moment.)

Having fixed $i_o$, we now consider the sets $S_j$ for $j> i_o+k+\ell$. If any such set $S_j$ has size 2 or more, then we may take $q_j\in\S_j$, distinct from $p_j$, such that $p_{j-\ell+1},\dots,p_{j-1},\{p_j,q_j\}$ forms a closing fork. In particular, since the earliest opening fork appears in a set with index at most $i_o+\ell$, we can find a gridded coil antichain with at least $j-(i_o+\ell)+2> k+2$ entries, which is a contradiction. Thus, $|S_j|\leq 1$ for all $j>i_o+k+\ell$.

We now consider the entries of $\pi^\gridded$ that lie in the sets $S_{i_o-1},S_{i_o},\dots, S_{i_o+k+\ell+1}$. We will use $\beta^\gridded$ to refer to the `body' gridded permutation corresponding to these entries. By construction, the gridding of $\beta^\gridded$ can be refined using the sets $S_j$ into $\beta^{\ggridded}$ on a grid whose cell graph is a path of length at most $k+\ell+2$. Furthermore, we have $S_{i_o-1}=\{p_{i_o-1}\}$ and $S_{i_o+k+\ell+1} = \{p_{i_o+k+\ell+1}\}$ are both singletons (or, potentially, $S_{i_o+k+\ell+1}$ is empty if $m<i_o+k+\ell+1$).

To record the coil $p_1,p_2,\dots,p_{i_o-1}$, it now suffices to record the length $n_1=i_o-1$, noting in particular that the placement of entries is unambiguous by construction, and since $S_{i_o-1}$ contains a single entry. Similarly, if $p_{i_o+k+\ell+1}$ exists, we can record the length $n_2=m-(i_o+k+\ell-1)+1$ of the coil following $\beta^\gridded$.

We have now described an injective map $\phi: \pi^\gridded \mapsto (n_1,\beta^{\ggridded},n_2)$ which is order-preserving: setting $\phi(\sigma^\gridded) = (m_1,\alpha^{\ggridded},m_2)$ and $\phi(\pi^\gridded) =(n_1,\beta^{\ggridded},n_2)$, we have that $\sigma^\gridded\leq \pi^\gridded$ if and only if $m_1\leq n_1$, $m_2\leq n_2$ and $\alpha^{\ggridded}\leq \beta^{\ggridded}$. \noterb{This if and only if may be a lie, and we don't need that full strength anyway.} Since the image of $\phi$ is the Cartesian product of three well quasi-ordered sets, we conclude that the set of $M$-indivisible permutations in $\C$ is well quasi-ordered.
\end{proof}

\section{Orphaned material from the LWQO section}

\begin{lemma}\label{lem-coil-embeddings}
Let $M$ be a cyclic matrix with orientation $\omega$ and cycle length $\ell$, and let $u_1,\dots,u_m$ and $v_1,\dots,v_n$ be two gridded coils in $M$.
If \[\phi: \{u_1,\dots,u_m\} \mapsto \{v_1,\dots,v_n\}\] is an embedding of the shorter coil into the longer, then there exists $j$ such that $\phi(u_{i}) = v_{j+i}$ for all $1\leq i \leq m$.
\end{lemma}

\begin{proof}
First, the image of $u_1,\dots,u_m$ under any embedding $\phi$ is a contiguous set of entries of $v_1,\dots,v_n$. This is because $u_1,\dots,u_m$ is (as a digraph) strongly connected, but any noncontiguous collection of entries from $v_1,\dots,v_n$ has a `gap', and by Lemma~\ref{lem-coil-splits} is thus not strongly connected.

We begin by observing that since any sequence $u_1,\dots,u_\ell$ forms a directed cycle, it must embed into a directed cycle in $v_1,\dots,v_n$, and thus embeds into a contiguous segment $v_i,\dots,v_{i+\ell-1}$ in such a way that $u_{j+1}$ embeds into the entry immediately following $u_j$, unless $\phi(u_j) = v_{i+\ell-1}$ in which case $\phi(u_{j+1})=v_i$.

Now consider the gridded coil $u_1,\dots,u_{\ell+1}$. We note that $u_1$ and $u_2$ are the only entries with indegree 2. Thus, the only possible embedding of $u_1,\dots,u_\ell$ that can be extended to include $u_{\ell+1}$ is the one claimed in the lemma.

For the general case $u_1,\dots,u_m$ with $m>\ell+1$, we now argue inductively: both the gridded coils $u_1,\dots,u_{m-1}$ and $u_2,\dots,u_m$ can only embed into $v_1,\dots,v_n$ in the way claimed in the lemma, and this completes the proof.

\commentrb{The above may be improved. I re-introduced it to use later.}
\end{proof}

We can now construct an ungridded labelled infinite antichain.

\begin{prop}\label{prop-coil-ungridded-antichains}
	Let $(v_i)_{i\in\mathbb{N}}$ be an (infinite) coil defined on the oriented cyclic matrix $M$. Then the sequence of (ungridded) permutations $\{\pi_i = v_1,v_2,\dots,v_i: i\geq \gridbound \}$ in which the first and last entries of the sequence are labelled differently from all the other elements forms a labelled infinite antichain.
\end{prop}

\begin{proof}
	For a contradiction, suppose that there exist $i,j$ such that $\pi_i$ embeds in $\pi_j$ in such a way that $\{v_1,v_i\}$ embeds in $\{v_1,v_j\}$ (as required by the labelling).
	
	For any such embedding, any $M$-gridding of $\pi_j$ induces an $M$-gridding of $\pi_i$, but by Lemma~\ref{lem-coil-unique-gridding} both $\pi_i$ and $\pi_j$ have unique griddings, so $\pi^\gridded_i$ embeds in $\pi^\gridded_j$ in such a way that $\{v_1,v_i\}$ embeds in $\{v_1,v_j\}$. This contradicts Lemma~\ref{lem-coil-antichains}.
\end{proof}

%
%
%
\subsection{Rectangle growth}

{\color{red}It's not clear whether we need this subsection in this paper. Keeping it here for now in case it's useful.}

Let $M$ be a $k\times \ell$ gridding matrix, equipped with a fixed (but possible arbitrary) orientation, so that every cell has an origin coordinate $(0,0)$ and a top coordinate of $(1,1)$. Let $\pi$ be an $M$-gridded permutation, and $S$ a set of points in $\pi$.

For each row and column of the gridded permutation $\pi$, identify the greatest (with respect to the orientation of $M$) point in $S$. This defines a set $L(S)=\{h_1,\dots,h_\ell,v_1,\dots,v_k\}$ of $\ell$ horizontal and $k$ vertical lines (one per row/column of $M$).

Define the \emph{span} and \emph{weak span} of $S$ as follows:
\begin{align*}
\Span(S) &= \bigcup_{(i,j)\in [k]\times [\ell]} ([0,v_i]\times [0,h_j]) \cap \pi\\
\wkSpan(S) &= \bigcup_{i\in [k]} ([0,v_i]\times [0,1]) \cup \bigcup_{j\in [\ell]} ([0,1]\times [0,h_j])
\end{align*}
Informally (as the above is probably riddled with inconsistent notation), for a cell whose origin is in the bottom left corner, the span of $S$ comprises all points contained in the bottom left, up to the horizontal and vertical lines in $L(S)$ that cut this cell. The weak span comprises all the points that are either below the horizontal line or to the left of the vertical line, or both.

Critically, it is the existence of points in $\Span\setminus\wkSpan$ that enables us to continue the process of rectangle growth. We have the following observation.

\begin{lemma}
An $M$-gridded permutation $\pi$ is $M$-indivisible if and only if whenever $\wkSpan(S)\setminus \Span(S) = \varnothing$ we have $\Span(S)=\pi$.
\end{lemma}

\begin{proof}
If $\Span(S)\neq \pi$ for some set $S$ that otherwise satisfies the hypothesis in the lemma, then the set of lines $L(S)$ witnesses that $\pi$ is $M$-divisible, with one part comprising $\Span(S)$ and the other part $\pi \setminus \Span(S)$.

Conversely, suppose $\pi$ is $M$-divisible. Write $\pi$ as an $M$ sum $\pi_1\boxplus \pi_2$, and then let $S$ comprise all points in $\pi_1$. Clearly, $\Span(S) = \pi_1\neq \pi$, and $\Span(S) = \wkSpan(S)$.
\end{proof}

For an $M$-indivisible permutation $\pi$ and a set $S$ of points in $\pi$, define $S_1 = S$, and for $i\geq 2$ let $S_i = \wkSpan(S_{i-1})$. The \emph{rectangle growth} of $S$ is\note{TODO: Check indices are not ``off by one'' from what we expect.}
\[\rg(S) = \min \{k : S_k=\pi\}.\]
The \emph{rectangle growth} of $\pi$ is the largest of these,
\[\rg(\pi) = \max_{S\subseteq\pi} \rg(S).\]
Note that if $S\subseteq T$, then $\rg(S)\geq \rg(T)$, which means that $\rg(\pi)$ will always be realised by starting with a set comprising a \emph{single} point of $\pi$.

Rectangle growth is fundamentally related to Murphy-Vatter sequences for constructing antichains.\note{TODO: We ought to define Murphy-Vatter sequence.}

\begin{prop}
Let $\pi$ be an $M$-gridded permutation for some gridding matrix $M$. If $\rg(\pi)=r$, then there exists a Murphy-Vatter sequence of length at least $r$ inside $\pi$.

Conversely, if $\pi$ is a Murphy-Vatter sequence on $r$ points, then$\rg(\pi)=r$.
\end{prop}

\begin{proof}
For the first part, pick  $S_1=\{s\}$ such that $\rg(S_1)=\rg(\pi)$. For $i=2,\dots,r$, let $S_i = \wkSpan(S_{i-1}$. Now, choose $s_r \in S_r\setminus S_{r-1}$. By symmetry, we may suppose that $s_r$ lies in a cell oriented from bottom left to top right, and also that $s_r$ lies to the right of all points of $S_{r-1}$ in the same column.

Since $s_r\in\wkSpan(S_{r-1})$, there must exist $s_{r-1}\in S_{r-1}$ in the same row as, but higher than, $s_r$. Note that $s_{r-1}\not\in S_{r-2}$, else we would have $s_r \in \wkSpan(S_{r-2})=S_{r-1}$. Now, from $s_{r-1}$, we identify $s_{r-2} \in S_{r-2}$ similarly, then $s_{r-3}\in S_{r-3}$, and so on. By inspection, $s_1,\dots,s_r$ is a Murphy-Vatter sequence.

For the converse, let $s_1,\dots,s_r$ be a Murphy-Vatter sequence, and consider $\rg(\{s_1\}$. If $r=1$ then clearly $\rg(\{s_1\})=1$. By induction on $r$, we may assume $\rg(\{s_1,\dots,s_{r-1}\})=r-1$. Now by construction we have $s_r\not\in\wkSpan(\{s_1,\dots,s_{r-2}\})=\Span(\{s_1,\dots,s_{r-1})$ but $s_r$ is in the weak span of $\{s_1,\dots,s_{r-1}\}$, thus the rectangle growth is $r$.\note{Technically, I think all we have shown is $\rg \geq r$.}
\end{proof}

Our task now is to look for Murphy-Vatter sequences in finitely based classes. For example, we may want to claim that it suffices to look at all Murphy-Vatter sequences of length $f(m)$, where $m=\max_{\pi\in\text{basis}}|\pi|$ and $f$ is some suitably-growing function. We would then wish to appeal to an argument to guarantee an infinite periodic Murphy-Vatter sequence. The difficulty, of course, is to identify an appropriate gridding matrix over which to look for rectangle growth.

We have examples where the basic gridding matrix has bounded rectangle growth, but where a refined matrix can contain arbitrarily large rectangle growth. Thus, when we find that we have bounded rectangle growth in some matrix, the task is to refine the matrix in such a way as to reduce the number of cycles the matrix contains, and to then study this structure.

%
%
%
%
%
\subsection{Older stuff -- basis}

Every basis element $\beta$ of a monotone grid class $\CCC=\Grid(M)$ has the following properties:
\begin{bullets}
  \item $\beta$ has no gridding in $\CCC$.
  \item $\beta$ has a gridding in one or more one-point extension of $\CCC$. \\[3pt]
  Notation and terminology: $\CCC^\bullet=\Grid(M^\bullet)$, where $M^\bullet$ has entries from $\{0,1,-1,\bullet,1^\bullet,-1^\bullet\}$  and is identical to $M$ except for the addition of a \emph{spot} to a single cell.
  \item If $\beta^-$ is \emph{any} one-point contraction of $\beta$, then $\beta^-$ has a gridding in $\CCC$.
\end{bullets}

\subsection*{Prerequisites for this section}
\begin{bulletnums}
  \item Any antichain in a one-point extension of a unicyclic grid class, $\CCC^\bullet$, contains elements with arbitrarily long \emph{spirals} (a.k.a. ``$M$-components'' or ``Murphy--Vatter sequences'').

      An $n$-point spiral gridded in a positive $k$-cycle consists of $n/k$ \emph{circuits} around the cycle, each cell containing at least $\floor{n/k}$ points in each cell. Successive points in a cell (from successive circuits) form a monotone sequence.

      A negative $k$-cycle in $\Grid(M)$ needs to be considered as a positive $2k$-cycle in $\Grid(M^{\times2})$ as far as spirals are concerned.

  \item A long enough spiral $\sigma$ contains a run $\rho$ of many consecutive points, such that in any gridding of $\sigma$, the run $\rho$ has a unique gridding in (the cells composing the cycle of) $\CCC^\bullet$. That is, each point in $\rho$ has only one cell in which it can be gridded.

      In a cyclic class, $\rho$ appears to consist of all of $\sigma$ (assuming it has at least one full circuit).
      A single cell not in the cycle can only take points from one of the two circuits at the ends of the spiral.
      Each additional cell not in the cycle can take points from one additional circuit.
      So, it appears that, in a unicyclic class consisting of a (positive) $k$-cycle and $\ell$ additional non-zero cells,
      if $\sigma$ has $n$ points, then $\rho$ has at least $n-k\ell$ points.\footnote{\label{fnWQO}It appears that we could allow any class in the non-cycle cells without needing to modify this argument very much.}

  \item If $\sigma$ is a sufficiently long spiral with uniquely-gridded run $\rho$,
      and $\sigma^-$ is a one-point contraction of $\sigma$ formed by deleting a cell-medial point $x$ from $\rho$,
      then in any gridding of $\sigma^-$, the \emph{split run} $\rho\setminus\{x\}$ still has a unique gridding (the same as for $\rho$).

      To be able to choose a cell-medial $x$ from any cell of the cycle, we need $\rho$ to contain at least three circuits.
      The split run $\rho\setminus\{x\}$ consists of (the $M$-sum of) two runs, $\rho_1$ and $\rho_2$ say. \emph{A~priori}, it might be expected that the ends of $\rho_1$ and $\rho_2$ adjacent to $x$ in $\rho$ might be griddable in some other cell, but this does not seem to be the case. (If it is, we still have sufficiently long runs in $\rho_1$ and $\rho_2$ with a unique gridding.)
\end{bulletnums}

\begin{figure}[ht]
\begin{center}
  \begin{tikzpicture}[scale=0.30]
    \plotperm{7,9,5,4,2,10,12,13,11,1,8,3,6}
    \draw (5.5,.5) -- (5.5,13.5);
    \draw (8.5,.5) -- (8.5,13.5);
    \draw (.5,6.5) -- (13.5,6.5);
  \end{tikzpicture}
  $\qquad\qquad$
  \begin{tikzpicture}[scale=0.30]
    \draw[blue,thick] (3,.5) -- (3,13.5);
    \draw[blue,thick] (5,.5) -- (5,13.5);
    \draw[blue,thick] (9,.5) -- (9,13.5);
    \draw[blue,thick] (.5,2) -- (13.5,2);
    \draw[blue,thick] (.5,5) -- (13.5,5);
    \draw[blue,thick] (.5,11) -- (13.5,11);
    \plotperm{7,9,5,4,2,10,12,13,11,1,8,3,6}
    \setplotptradius{5pt}
    \plotperm[white]{0,0,5,0,2, 0, 0, 0,11}
    \setplotptradius{4pt}
    \plotperm[blue]{0,0,5,0,2, 0, 0, 0,11}
    \draw (5.5,.5) -- (5.5,13.5);
    \draw (8.5,.5) -- (8.5,13.5);
    \draw (.5,6.5) -- (13.5,6.5);
  \end{tikzpicture}
\end{center}
\caption{The gridding of a permutation in \protect\gctwo{3}{1,1,-1}{-1,0,1}, shown at the right with three rigid points}
\label{figGriddings}
\end{figure}

\subsection*{Terminology and notation}
Let's refer to an $M$-gridded permutation whose underlying permutation is~$\pi$
as an \emph{$M$-gridded~$\pi$}.
If $\pi_\gridded$ is an $M$-gridded $\pi$ and $p$ is a point of $\pi$, then let $\pi_\gridded(p)$ denote the cell of $M$ containing~$p$.

\subsection*{Rigidity}
Let $\pi_\gridded$ be an $M$-gridded $\pi$, and $R$ be a set of points of $\pi$.
Let
\[
S(\pi_\gridded,R) \;=\;
\{
\pi'_\gridded \::\: \pi'_\gridded \text{~is an $M$-gridded $\pi$ and~} \pi'_\gridded(r) = \pi_\gridded(r) \text{~for all~} r\in R
\}
\]
be the set of $M$-gridded $\pi$s for which the points of $R$ are in the same cells as they are in $\pi_\gridded$.
We say that the points in $R$ are \emph{rigid}.

Let $M(\pi_\gridded,R)$ be the subdivision of $M$ constructed by adding an additional row divider and an additional column divider through each point of $R$ in $\pi_\gridded$. Let's call these additional row and column dividers \emph{rigid} dividers.
These rigid dividers are always in the same place in a plot of any $M$-gridded permutation in $S(\pi_\gridded,R)$, the original dividers being constrained to lie between the rigid dividers as they do in $\pi_\gridded$.
See Figure~\ref{figGriddings}.

\textbf{Observation.} Any point of $\pi_\gridded$ that is between two rigid points in the same cell of $M$ is in that cell in every member of $S(\pi_\gridded,R)$.

\subsection*{Outline of proof}
Assume $\CCC=\Grid(M)$ has an infinite basis and let $\beta$ be a basis element with a sufficiently long spiral~$\sigma$.

By the prerequisites, there is a run $\rho$ of many consecutive points in $\sigma$ which has the same unique gridding in \emph{any} gridding of $\sigma$ in \emph{any} one-point extension of $\CCC$.

Let $\beta^\bullet_\gridded$ be an $M^\bullet$-gridded $\beta$ for some one-point extension $M^\bullet$ of $M$.
The set of $M$-gridded permutations
$S(\beta^\bullet_\gridded,\rho)$, in which the points in $\rho$ are rigid,
is precisely the set of all $M^\bullet$-gridded~$\beta$s, whatever the choice of $\beta^\bullet_\gridded$.

Let $s$ be the point that is gridded on the spot in $\beta^\bullet_\gridded$.

We need to deal with a specific case.
Suppose that in $\beta^\bullet_\gridded$, the point $s$ is between two rigid points $r_1$ and $r_2$ in the same cell $C$ of $M^\bullet$ (making a 1324 or 4231 in that cell with $r_1$ and $r_2$ acting as the 1 and 4).
Then, by the Observation, the point $s$ is in $C$ in every member of $S(\beta^\bullet_\gridded,\rho)$.
Moreover, by the prerequisites, $r_1$ and $r_2$ are in $C$ in every gridding of $\beta$ in \emph{any} one-point extension of $\CCC$.
Therefore the points \emph{between} $r_1$ and $r_2$ are also in $C$ in all such griddings, forming a 1324 or 4231.
Thus $C$ is always the cell with the spot and $\beta$ only has griddings in this particular one-point extension of $\CCC$.
Let's call this situation the \emph{bad case}.

Our goal is to find some one-point contraction of $\beta$ for which it can be shown that if it has a gridding in $\CCC$ then so does $\beta$, contradicting the fact that $\beta$ is a basis element of $\CCC$.

Let $\beta^-$ be a one-point contraction of $\beta$, where the deleted point $x$ is a cell-medial (rigid) point of $\rho$, between two rigid points $r$ and $r'$ in the same cell.
In the bad case, choose $x$ from some cell other than the cell with the spot (otherwise the choice is free).

In all cases, in every gridding of $\beta$ in any one-point extension $\CCC^\bullet$ of $\CCC$, the point gridded on the spot is outside the rectangular region formed by the rigid dividers through $r$ and $r'$.
Thus if $\beta^-$ were to have a gridding in $\CCC$, the reinstatement of $x$ would yield a gridding of $\beta$ in $\CCC$.

%
%
%
%
%
%
%
%
\section{Enumeration of unicyclic classes}

This is not yet established. We know that the \emph{gridded} permutations in a unicyclic grid class have an algebraic generating function (this follows by diagonalising the regular language of the gridded permutations on an associated acyclic grid~\cite[Chapter~4]{bevan:thesis:}). Can we do any of the following:
\begin{enumerate}
\item Choose some kind of greedy gridding (such as is done in~\cite{albert:geometric-grid-:}), and diagonalise those (the gridding would need to match
\commentdb{If a permutation can be gridded with the same number of points in the two ``end'' cells, then the greedy gridding must be such a gridding});
\item Diagonalise the regular language, then carefully factor out a single greedy gridding (this means subtracting stuff from a context-free class, which is not usually a good idea);
\item Exploit structural insights from the above (e.g. enumerate $M$-indivisibles) to count more directly?
\end{enumerate}

Bonus question: can we derive results for wqo subclasses of unicyclic classes? Must they be rational? \commentrb{The more I think about this, the more I'm certain our 321 paper can be adapted to prove it. This is almost surely too much to add to this paper, though.}

\section{Recently old stuff from section 2}

\subsection{Murphy-Vatter sequences}

We now focus our attention solely on \emph{unicyclic} partial multiplication matrices. Furthermore, we will restrict our attention to matrices which possess a single component, because of the following.

\begin{prop}[{See Vatter~\cite[Proposition 2.9]{vatter:small-permutati:}}]\label{prop-wqo-components}
If the grid classes of the connected components of $\Grid(M)$ are well quasi-ordered, then so is $\Grid(M)$.
\end{prop}

Our aim is to describe the $M$-indivisible objects of unicyclic grids in such a way as to characterise labelled well quasi-ordering precisely.
Critical to our structural description will be \emph{Murphy-Vatter sequences} (or \emph{MV-sequences} for short), so-named because of their first appearance in \cite{murphy:profile-classes:} where they were used to characterise well quasi-ordering in grid classes.

\noterb{This is the 2nd attempt at a definition. It may still be better to define via the digraphs. Is an MV sequence a gridded permutation, or does it merely \emph{induce} one?}Given a unicyclic partial multiplication matrix $M$ with orientation $\omega$, we say that the gridded permutation $\mu^\gridded$ is a \emph{Murphy-Vatter sequence} for $M$ if its entries belong entirely to the cells on a cycle, and if they can be ordered $p_1,\dots,p_m$ such that:
\begin{enumerate}[label=\textbf{MV\arabic*}]
\item\label{mv1} $p_1$ is the first entry in its cell.
\item\label{mv2} For $i\geq 2$, $p_i$ shares a row or column (but not a cell) with $p_{i-1}$, and precedes $p_{i-1}$ in their common orientation.
\item\label{mv3} For $2\leq i < m$, any entry in $p_{i+1},\dots,p_m$ that shares a row or column with $p_{i-1}$ succeeds $p_{i-1}$ in their common orientation.
\end{enumerate}

We make a few observations about MV-sequences: In the associated digraph $D_{\mu^\gridded}$, we have $p_i\rightarrow p_{i-1}$ for all $i\geq 2$, and either $p_i \rightarrow p_j$ or $p_i$ and $p_j$ are not adjacent for all $i,j$ with $j>i+1$. Thus, $p_1\leftarrow p_2\leftarrow\dots\leftarrow p_m$ is a directed path, and any other directed edges that exist point from an earlier entry to a later one. While consecutive entries of an MV sequence must share a row or column, this comment about $D_{\mu^\gridded}$ also gives us the following.

\begin{obs}
For all $1 < i  < m$, the entries $p_{i-1}$ and $p_{i+1}$ do not lie in a common row or column.
\end{obs}

This observation underpins the desired property of MV sequences, that they must `wind' around the cycle in the same direction from beginning to end.

\commentrb{I think it might help to be much more explicit about these MV sequences: they visit each cell of the cycle in turn. This will help us to relate it to earlier results about $M$-divisibility, as well as handling embeddings/antichain constructions that appear shortly.}

\begin{lemma}\label{lem-mv-firstpoints}For a Murphy-Vatter sequence $p_1,\dots,p_m$ containing at least one point in each cell of the cycle, the associated digraph is strongly connected.

Furthermore, the removal of any entry $p_2,\dots,p_{m-1}$ produces a gridded permutation that is $M$-divisible.
\end{lemma}

\begin{proof}
For the first part, since there is at least one entry in every cell of the cycle, there exists $a>2$ such that $p_1$ and $p_a$ share a row and/or column, and by~\ref{mv3} we therefore have $p_1\rightarrow p_a$ (and, by \ref{mv2}, $p_a\rightarrow p_{a-1}\rightarrow \dots \rightarrow p_1$). Similarly, there exists $z<m-1$ such that $p_z\rightarrow p_m$ (and $p_m\rightarrow p_{m-1}\rightarrow \dots \rightarrow p_z$). Choose $a$ maximal, and $z$ minimal with this property, which means $p_a$ is the latest entry in its cell, and $p_z$ is the first entry in its cell. We claim that $a\geq z$, from which it follows by \ref{mv2} that there exists a directed path from $p_a$ to $p_z$, and hence that $D_{\mu^\gridded}$ is strongly connected.

\noterb{This isn't very pretty}Suppose, for a contradiction, that $a<z$. By \ref{mv2} we have $p_a \leftarrow p_{a+1} \leftarrow \dots \leftarrow p_z$. Since $p_a$ is the latest entry in its cell, we see that no entries in $p_{a+1},\dots,p_m$ can be placed in the same cell as $p_a$. In other words, after $p_a$, the MV sequence does not follow the whole cycle around again. However, we know that $p_z\rightarrow p_m$, which means that the sequence $p_z,p_{z+1},\dots,p_m$ must visit every cell of the cycle at least once, a contradiction.

The second part of the lemma follows by \ref{mv2}: the removal of $p_i$ means that there are no arrows pointing from any entry in $\{p_{i+1},\dots,p_m\}$ to any entry in $\{p_1,\dots,p_{i-1}\}$, and hence by Lemma~\ref{lem-Dpi-strongly-connected} the permutation must be $M$-divisible.
\end{proof}

Furthermore, MV-sequences possess strong embedding properties:

\begin{lemma}[Essentially due to Murphy and Vatter~\cite{murphy:profile-classes:}]\label{lem-mv-embeddings}
Let $M$ be a unicyclic partial multiplication matrix with orientation $\omega$ and cycle length $c$. For integers $m,n$ with $c\leq m\leq n$, let $p_1,\dots,p_m$ and $q_1,\dots,q_n$ be two MV-sequences in the same grid, and suppose that $\phi: \{p_1,\dots,p_m\} \mapsto \{q_1,\dots,q_n\}$ is an embedding of the shorter gridded permutation into the longer.
Then, for all $1\leq i \leq m$ there exists $j$ such that $\phi(p_{i}) = q_{j+i}$.
\end{lemma}

\begin{proof}
Fix $j$ so that $\phi(p_1) = q_{j+1}$. Now, since $m\geq c$, the length of the cycle of $M$, we observe that $p_1,\dots,p_m$ has at least one entry in each cell of the cycle, and thus by Lemma~\ref{lem-mv-firstpoints} the associated digraph $D_{p_1,\dots,p_m}$ is strongly connected. \commentrb{TODO. A helpful observation might be that the digraphs of MV-sequences have shortest directed cycle equal to $c$.}
\end{proof}

\begin{lemma}\label{lem-mv-antichains}
Let $p_1,p_2,\dots$ be an (infinite) MV-sequence in the grid class of some oriented (unicyclic) partial multiplication matrix $M$. Then the sequence of gridded permutations $\{\pi_i = p_1,p_2,\dots,p_i: i\geq 1\}$ in which the first and last entries of the sequence are labelled differently from all the other elements forms a labelled infinite antichain.
\end{lemma}

\begin{proof}
By Lemma~\ref{lem-mv-embeddings}, any embedding $\phi$ of $\pi_i$ into $\pi_j$ must be contiguous. However, the labelling of $\pi_i$ and $\pi_j$ requires that $\phi(p_1)=p_1$ and $\phi(p_i)=p_j$, which is impossible (and so no embedding can exist).
\end{proof}

Finally, we also need to argue that MV-sequences continue to possess this embedding property when we remove the grid lines of $M$. That is, we need to argue that the ungridded permutations still form an infinite labelled antichain. Note that we require the existence of a single component in $M$ at this point to avoid the entire MV sequence being gridded in two identical components.

\begin{lemma}[Murphy and Vatter~{\cite[Lemma 4.2]{murphy:profile-classes:}}]\label{lem-mv-unique-gridding}
Let $M$ be a unicyclic grid containing at most $\ell$ non-empty cells. Then every MV sequence of length at least $(\ell+1)\ell^2+1$ is uniquely griddable.
\end{lemma}

\begin{proof}[Sketch of proof]
\commentrb{This is done in~\cite{murphy:profile-classes:}, it is quite long. Clearly the `real' bound is much smaller than this -- is there a neater method? e.g.\ one using the directed graph?}
Let $\pi$ be an MV sequence containing $n\geq (\ell+1)\ell^2+1$ points, and let $\pi^\gridded$ denote the standard gridding of $\pi$ (i.e.\ the gridding used to generate the MV sequence). Now consider some other gridding $\pi^\natural$. By the pigeonhole principle, there exists some collection of $\ell+2$ points that belong to a single cell in each of $\pi^\gridded$ and $\pi^\natural$ (though note: not necessarily the same cell of $M$ yet).

Now use these entries first to conclude that there are $\ell+1$ entries in the two neighbouring cells on the cycle that lie together in a cell in both griddings, and $\ell$ entries in the next cells on the cycle, and so on.\noterb{This seems suboptimal: we can go both ways round the cycle...} By following these entries around the cycle (which has length at most $\ell$), we conclude that all the entries so far considered in fact lie in the same cells in both griddings. It remains to collect the remaining entries into the correct cells, and this is readily checked using the properties of MV sequences.
\end{proof}

As a direct consequence of the above two lemmas, we obtain.

\begin{prop}\label{prop-long-mv-sequences}
Any permutation class which contains arbitrarily long MV sequences around some cycle is not labelled well quasi-ordered.
\end{prop}
